\documentclass[11pt]{article}
\usepackage[colorlinks,citecolor=blue,urlcolor=blue,breaklinks]{hyperref}
\usepackage{latexsym, amsmath,amssymb,amsfonts,amsthm,bbm,bm,enumitem,soul,graphicx, dsfont,enumitem,natbib,threeparttable,subcaption, comment,cleveref,multirow,caption,makecell}
\usepackage[dvipsnames]{xcolor}
\usepackage[normalem]{ulem}
\usepackage[ruled, vlined]{algorithm2e}
\usepackage{multirow, soul}
\SetKwInput{KwSet}{Set}

\newtheorem{thm}{Theorem}
\newtheorem{prop}[thm]{Proposition}
\newtheorem{lemma}[thm]{Lemma}
\newtheorem{cor}[thm]{Corollary}
\newtheorem{defn}{Definition}
\newtheorem{remark}{Remark}

\newtheorem{condition}{Condition}

\newcommand{\RN}[1]{%
  \textup{(\lowercase\expandafter{\romannumeral#1})}%
}

\DeclareMathOperator*{\argmin}{argmin}

\newcommand\nnfootnote[1]{%
  \begin{NoHyper}
  \renewcommand\thefootnote{}\footnote{#1}%
  \addtocounter{footnote}{-1}%
  \end{NoHyper}
}

\allowdisplaybreaks

\title{Robust mean change point testing in high-dimensional data with heavy tails}
\author{Mengchu Li$^{*, 1, 2}$, Yudong Chen$^{*, 2, 3}$, Tengyao Wang$^{3}$, Yi Yu$^{2}$ \\
$^1$School of Mathematics, University of Birmingham \\
$^2$Department of Statistics, University of Warwick \\
$^3$Department of Statistics, London School of Economics and Political Science
}
\date{(\today)}

\begin{document}
\maketitle

\begin{abstract}
      We study mean change point testing problems for high-dimensional data, with exponentially- or polynomially-decaying tails.  In each case, depending on the $\ell_0$-norm of the mean change vector, we separately consider dense and sparse regimes. We characterise the boundary between the dense and sparse regimes under the above two tail conditions for the first time in the change point literature and propose novel testing procedures that attain optimal rates in each of the four regimes up to a poly-iterated logarithmic factor. By comparing with previous results under Gaussian assumptions, our results quantify the costs of heavy-tailedness on the fundamental difficulty of change point testing problems for high-dimensional data. 
      
      To be specific, when the error distributions possess exponentially-decaying tails, a CUSUM-type statistic is shown to achieve a minimax testing rate up to $\sqrt{\log\log(8n)}$. As for polynomially-decaying tails, admitting bounded $\alpha$-th moments for some $\alpha \geq 4$, we introduce a median-of-means-type test statistic that achieves a near-optimal testing rate in both dense and sparse regimes. In the sparse regime, we further propose a computationally-efficient test to achieve optimality. Our investigation in the even more challenging case of $2 \leq \alpha < 4$, unveils a new phenomenon that the minimax testing rate has no sparse regime, i.e.\ testing sparse changes is information-theoretically as hard as testing dense changes. Finally, we consider various extensions where we also obtain near-optimal performances, including testing against multiple change points, allowing temporal dependence as well as fewer than two finite moments in the data generating mechanisms. We also show how sub-Gaussian rates can be achieved when an additional minimal spacing condition is imposed under the alternative hypothesis. 
\end{abstract}
\nnfootnote{*Equal contribution}

\addtocontents{toc}{\protect\setcounter{tocdepth}{0}}

\section{Introduction}
\label{Sec:intro}
In this paper, we study change point testing problems when the observations are corrupted by heavy-tailed errors. To be specific, consider the `signal plus noise' model 
\begin{equation*}
    X_t = \theta_t + E_t \in \mathbb{R}^p, \qquad t = 1, \ldots, n,
\end{equation*}
where $X_t$ represents the $p$-variate observation at time $t$, $\theta_t$ the signal and $E_t$ the error term. Writing $X := (X_1, \ldots, X_n) \in \mathbb{R}^{p \times n}$, and similarly for $\theta$ and $E$, we express the model in matrix form as:
\begin{equation}\label{eq:model}
    X = \theta + E,
\end{equation}
where $X$, $\theta$ and $E$ are all $p \times n$ matrices. We start by assuming the entries of $E$ are independent random variables with zero mean and unit variance and we denote the distribution of $E$ as $P_e \in \mathcal{Q}$. We are interested in understanding the fundamental difficulty of testing whether the columns of $\theta$ undergo a change at some unknown location when the class $\mathcal{Q}$ contains heavy-tailed distributions. Focusing on the single change point alternative hypothesis for now, our goal can be formalised as testing 
\begin{equation} \label{Eq:H_0_H_1}
\mathrm{H}_0: \theta \in \Theta_0(p,n) \quad \mathrm{vs.}\quad  \mathrm{H}_1: \theta \in \Theta(p,n,s,\rho) := \bigcup_{t_0 = 1}^{n-1}\Theta^{(t_0)}(p,n,s,\rho),
\end{equation}
with
\begin{align}
    \label{Eq:null_space}
    \Theta_0(p,n) := \{&\theta: \theta_t = \mu \ \text{for all}\ t = 1,\dotsc,n, \ \text{for some}\ \mu \in \mathbb{R}^p\}
\end{align}
    and
\begin{align}
    \Theta^{(t_0)}(p,n,s,\rho):=\biggl\{&\theta: \theta_t = \mu_1\ \text{for}\ t = 1, \ldots, t_0, \ \theta_t = \mu_2 \ \text{for}\ t = t_0+1, \ldots, n,  \nonumber \\
    &\text{for some}\ \mu_1, \mu_2 \in \mathbb{R}^p \ \text{s.t.}\ \|\mu_1-\mu_2\|_0 \leq s,\, \frac{t_0(n-t_0)}{n}\|\mu_1-\mu_2\|_2^2 \geq \rho^2\biggr\}. \label{Eq:alternative_space_original}
\end{align}
For any $v = \bigl(v(1), \dotsc, v(d)\bigr)^\top\in \mathbb{R}^d$, we denote $\|v\|_0 := \sum_{i=1}^d \mathbbm{1}_{\{v(i) \neq 0\}}$ and $\|v\|_2 := \bigl\{\sum_{i=1}^d v(i)^2\bigr\}^{1/2}$. To put it in words, we use $\Theta_0(p,n)$ to denote the space of signals without a change point, and $\Theta^{(t_0)}(p,n,s,\rho)$ to denote the space of signals with a change at location $t_0$ of entry-wise sparsity level~$s$ and (normalised) signal strength $\rho$. The multiplicative factor $t_0(n-t_0)n^{-1}$ of $\|\mu_1-\mu_2\|_2^2$ can be regarded as the effective sample size of the problem.  It reflects the fact that the difficulty of testing change point is related to where the change happens. 

Change point analysis as a broad topic has received increasing attention in recent years. Various models \citep[e.g.][]{WangSamworth2018,liu2021minimax,wang2021optimal,wang2022testingregress, xu2022regresstemporal,verzelen2020optimal1d} are considered in the literature focusing on different tasks, including testing the existence of change points, estimating their locations and quantifying the uncertainty of the proposed estimators. From a theoretical point of view, {many of the problems studied are shown to exhibit a phase transition phenomenon,}  i.e.\ a change point can only be reliably tested or accurately localised when its signal strength, measured in some problem-dependent way, exceeds some threshold. It is, therefore, crucial to understand the boundary of this phase transition behaviour. For the testing problem that we are concerned with here, the key quantity is the minimax testing rate, $v^*_{\mathcal{Q}}(p,n,s)$, defined below. For a given $\theta$ and $E\sim P_e$, we write $\mathbb{P}_{\theta,P_e}$ the probability measure of the data $X$ generated from~\eqref{eq:model} and $\mathbb{E}_{\theta,P_e}$ the corresponding expectation operator.

\begin{defn} [Minimax testing rate]\label{def:minimax_testing}
    Let $\Phi$ denote the set of all measurable test functions $\phi: \mathbb{R}^{p \times n} \rightarrow \{0,1\}$. Consider the minimax testing error 
    \[
\mathcal{R}_{\mathcal{Q}}(\rho) :=  \inf_{\phi \in \Phi} \mathcal{R}_{\mathcal{Q}}(\rho, \phi) := \inf_{\phi \in \Phi} \biggl\{ \sup_{P_e \in \mathcal{Q}}\sup_{\theta \in \Theta_0(p,n)}\mathbb{E}_{\theta, P_e} (\phi) + \sup_{P_e \in \mathcal{Q}}\sup_{\theta \in \Theta(p,n,s,\rho)}\mathbb{E}_{\theta, P_e} (1-\phi) \biggr\}.
\]
For a fixed $\varepsilon\in(0,1/2)$, we say that $v^*_{\mathcal{Q}}(p,n,s)$ is the minimax testing rate if $\mathcal{R}_{\mathcal{Q}}(\rho) \leq \varepsilon$ when $\rho^2 \geq Cv^*_{\mathcal{Q}}(p,n,s)$, and $\mathcal{R}_{\mathcal{Q}}(\rho) \geq 1/2$ when $\rho^2 \leq cv^*_{\mathcal{Q}}(p,n,s)$, where $c, C > 0$ are constants depending only on $\varepsilon$ and $\mathcal{Q}$.
\end{defn}

Note that in \Cref{def:minimax_testing}, $C$ is allowed to depend on $\varepsilon$. Since the primary goal of the paper is to characterise the minimal size of the signal, in terms of various model parameters, where the testing problem starts to become feasible, we will treat $\varepsilon$ as a constant throughout the rest of the paper.

A minimax testing rate is previously studied in \cite{liu2021minimax} under model \eqref{eq:model}, where the entries of noise matrix $E$ are assumed to be independent standard normal random variables.  It is shown that 
\begin{equation}
    v^*_{{N}^{\otimes}(0, 1)}(p, n, s) = \Bigl\{\sqrt{p\log \log(8n)} \wedge \bigl[s\log\big\{eps^{-2} \log\log(8n)\big\}\bigr] \Bigr\} \vee \log \log(8n),
    \label{Eq:GaussianRate}
\end{equation}
where ${N}^{\otimes}(0, 1)$ denotes the joint distribution of all $pn$ independent $N(0, 1)$ entries in $E$. Our main contribution, presented in \Cref{Sec:main result}, is to characterise the impact of heavy-tailed distributions on the minimax testing rate. More specifically, we consider two classes of error distributions.

\begin{defn}[$\mathcal{G}_{\alpha, K}$ class of distributions] \label{def-galphak}
For $K >0$ and $\alpha \in (0,2]$, let $\mathcal{G}_{\alpha, K}$ denote the class of distributions on $\mathbb{R}$ such that for any $P \in \mathcal{G}_{\alpha, K}$ and random variable $W \sim P$, it holds that 
\[
    \mathbb{E}(W) = 0, \quad \mathbb{E}({W}^2) = 1 \quad \text{and} \quad \mathbb{E}\big(\exp\bigl\{|W/K|^\alpha\bigr\}\big) \leq 2.
\]
\end{defn}

The $\mathcal{G}_{\alpha, K}$ class consists of sub-Weibull distributions of order $\alpha$ with mean 0, variance 1 and the Orlicz $\psi_\alpha$-norm upper bounded by $K$ (see Definitions~\ref{def:Orlicz} and~\ref{def:subWeibull}). By Proposition~\ref{prop:subweibull-basic-prop}(a), they possess exponentially-decaying tails, as $\mathbb{P}(|W| \geq x) \leq 2e^{-(x/K)^\alpha}$, for $x > 0$.
    
\begin{defn}[$\mathcal{P}_{\alpha, K}$ class of distributions]\label{def-palphak}
    For $K > 0$ and $\alpha \geq 2$, let $\mathcal{P}_{\alpha, K}$  denote the class of distributions on $\mathbb{R}$ such that for any $P \in \mathcal{P}_{\alpha, K}$ and random variable $W \sim P$, it holds that
\[
    \mathbb{E}(W) = 0, \quad \mathbb{E}({W}^2) = 1 \quad \text{and} \quad \mathbb{E}\big(|W/K|^\alpha\big) \leq 1.
\]
\end{defn}
In words, each distribution within this class has its $\alpha$-th moment bounded above by $K^\alpha < \infty$ and possesses a polynomially-decaying tail. This is typically much heavier than an exponentially-decaying tail and thus poses a much bigger statistical challenge.

We study the minimax rate of testing $v^*_{\mathcal{Q}}(p,n,s)$ defined in Definition~\ref{def:minimax_testing} for $\mathcal{Q} = \mathcal{G}_{\alpha, K}^{\otimes}$ and $\mathcal{Q} = \mathcal{P}_{\alpha, K}^{\otimes}$, respectively. Let $\mathcal{G}_{\alpha, K}^{\otimes}$ and $\mathcal{P}_{\alpha, K}^{\otimes}$ denote the class of joint distributions of all the entries in the error matrix $E \in \mathbb{R}^{p \times n}$, when each entry of $E$ independently follows a distribution on $\mathbb{R}$ that belongs to the class $\mathcal{G}_{\alpha, K}$ and $\mathcal{P}_{\alpha, K}$, respectively. Throughout the paper, we treat $K$ and $\alpha$ as constants.

\subsection{Main results and outline}
\label{Sec:main result}
Our main results, developed in Sections~\ref{Sec:hidimtest} and \ref{Sec:robusttest}, characterise the minimax testing rates under both exponentially- and polynomial-decaying tails. In particular, we establish that in the case of exponentially-decaying tails
\begin{equation}\label{eq:rate-result-exp}
    v^*_{\mathcal{G}_{\alpha, K}^{\otimes}}(p,n,s)\asymp L \min\{\sqrt{p}, \,s\log^{2/\alpha}(ep/s)\} + \log\log (8n)
\end{equation}
for some $L \in \bigl[1, \sqrt{\log\log(8n)}\bigr]$, and in the case of polynomial-decaying tails
\begin{equation}\label{eq:rate-result-poly}
    v^*_{\mathcal{P}_{\alpha, K}^{\otimes}}(p,n,s)\asymp L \min\{p^{\frac{2}{\alpha}\vee \frac{1}{2}}, \, s(p/s)^{\frac{2}{\alpha}}\} + \log\log (8n)
\end{equation}
for some $L \in [1, \log\log (8n)]$. Upper and lower bounds with explicit dependence on iterated logarithmic factors are detailed in \Cref{subsec:discuss_gap_subweibull} and \cref{subsec:discuss_gap_robust}.

Note that as the level of sparsity $s$ increases from $1$ to $p$, there is a changeover in the dominating term in both \eqref{eq:rate-result-exp} and \eqref{eq:rate-result-poly}.  Depending on whether directly involving $s$ in the final rates, we refer to the resulting regimes the sparse and dense regimes,  with their transition boundaries $s_\mathcal{G}^*$ and $s_{\mathcal{P}}^*$ determined by
$s_\mathcal{G}^*\log^{2/\alpha}(ep/s_\mathcal{G}^*) = \sqrt{p}$ for $\mathcal{G}_{\alpha, K}^{\otimes}$ and $s_{\mathcal{P}}^*(p/s_{\mathcal{P}}^*)^{\frac{2}{\alpha}} = p^{\frac{2}{\alpha}\vee \frac{1}{2}}$ for $\mathcal{P}_{\alpha, K}^{\otimes}$.

The transition boundaries are demonstrated in \Cref{fig:rate_plot}. When $P_e \in \mathcal{P}_{\alpha, K}^{\otimes}$, the minimax testing rate transition occurs at  $s^*_{\mathcal{P}} = p^{1/2-1/(\alpha-2)}$ when $\alpha \geq 4$.  When $\alpha \in [2,4)$, there is essentially no sparse regime, since in this range of $\alpha$, there does not exist any $s \in [1,p]$ such that $s(p/s)^{\alpha/2} < p^{2/\alpha}$. This observation implies that testing sparse change is information-theoretically as hard as testing dense changes when $\alpha \in [2,4]$, as the minimax testing rate is independent of $s$. When $P_e \in \mathcal{G}^{\otimes}_{\alpha, K}$, the transition boundary {takes a simpler form of} $s^*_{\mathcal{G}} \asymp \sqrt{p}\log^{-2/\alpha}(ep)$ for $\alpha \in (0, 2]$. 

\begin{figure}[!h]
    \centering
    \includegraphics[width = 0.8\linewidth]{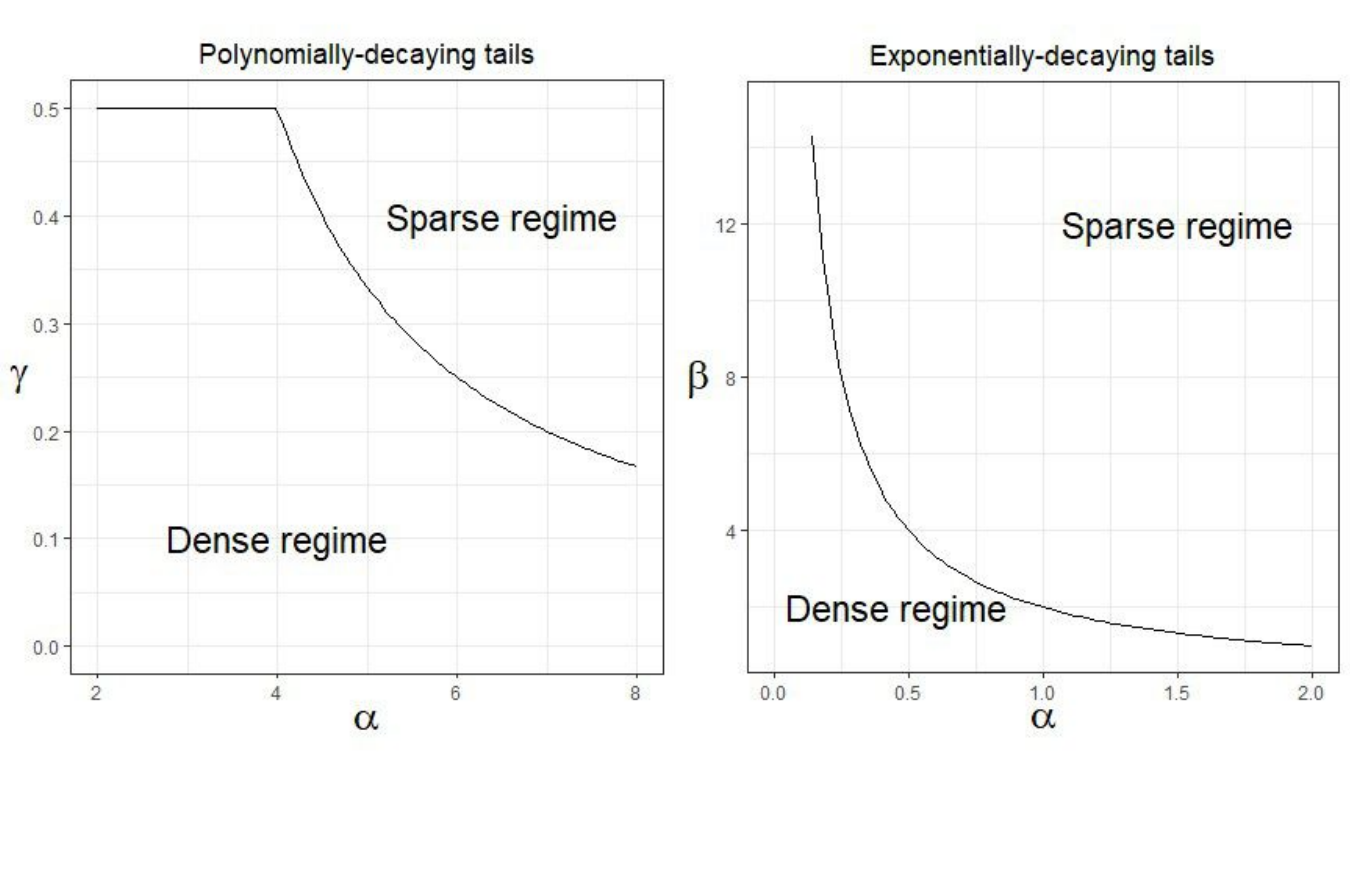}
    \vspace{-4em}
    \caption{Minimax testing rate transition boundaries between dense and sparse regimes when the distribution of the error matrix belongs to $\mathcal{P}_{\alpha, K}^{\otimes}$ (left panel) and $\mathcal{G}_{\alpha, K}^{\otimes}$ (right panel). The left panel plots the curve $\gamma(\alpha) = (\alpha-2)^{-1} \wedge 1/2$ for $\alpha \in [2, \infty)$, and the two regimes are separated by $s^*_{\mathcal{P}} = p^{1/2-\gamma}$.
    The right panel plots the curve $\beta(\alpha) = 2/\alpha$ for $\alpha \in (0,2]$, and the two regimes are separated by $s^*_{\mathcal{G}} \asymp \sqrt{p}\log^{-\beta}(ep)$.}
     \label{fig:rate_plot}
\end{figure}

The upper bounds on \eqref{eq:rate-result-exp} and \eqref{eq:rate-result-poly} are each obtained by analysing two different testing procedures separately, targeting at dense and sparse regimes. In practice, the level of sparsity is usually unknown and we address the issue of adaptation to sparsity in \Cref{sec:adaptsparsity}. We show that there is no additional cost of adaptation in achieving the optimal minimax testing rates. We then illustrate our main results via simulation studies in \Cref{Sec:simulations}. Finally, in \Cref{Sec:extensions},  we consider extensions in four interesting directions, including 
(1) testing against multiple change points, (2) accounting for temporal dependence among observations, (3) addressing the case where the noise matrix entries have fewer than two finite moments, and (4) examining the situation where {an additional minimal spacing condition is imposed between the potential change point and the boundary time points}.  Generally speaking, we develop near-optimal procedures under these more general settings and demonstrate an interesting phenomenon in (4)---tests can achieve sub-Gaussian performances under heavy-tailed noise assumptions if it is known that the potential change point is away from the endpoints by a small distance that only depends on the model parameters $p$, $n$ and~$s$ through logarithmic terms.

To highlight our contributions relative to the existing literature, we note that in previous works on robust mean change point testing problems \citep[e.g.][]{yu2022robust, jiang2023robust}, change point locations are required to be comparable to the length of time series in order to achieve near-optimal guarantees.  Our results, in comparison, cover a much more general parameter space, where the change point locations may be arbitrarily close to the boundary. Compared to recent works on optimal mean change point testing problems without robustness \citep[e.g.][]{liu2021minimax, verzelen2020optimal1d}, our results allow for general classes of distributions beyond Gaussian and sub-Gaussian cases and quantify the costs of heavy-tailedness.  Finally, compared to relevant recent works on robust estimation in sequence models \citep[e.g.][]{comminges2021adaptive}, we investigate the more challenging case where noise entries have fewer than four  finite moments, and unveil a new phenomenon on the effects of sparsity that was previously unknown even in sequence models. More in-depth discussions on these works can be found in \Cref{subSec:literature}.

\subsection{Relation to existing literature}\label{subSec:literature}

Many real-world data such as financial returns and macroeconomic variables exhibit heavy-tail phenomena, which often violate the convenient sub-Gaussian/exponential assumptions adopted by data analysts. Statistical procedures that mitigate the effects of heavy-tailed and/or contaminated data, therefore, have been sought after in practice, see \cite{resnick2007heavy} for more in-depth discussions. 
In the realm of change point analysis, one line of recent works \citep[e.g.][]{cho2022regressheavytail,wang2022testingregress,xu2022regresstemporal} consider change point models with exponentially-decaying heavy-tailed noise and study the performance of non-robust algorithms that perform well under sub-Gaussian noise assumptions. Theoretical results therein all require stronger assumptions on the strength of change points compared to the setting under sub-Gaussian assumptions.  One motivation for our work is thus to investigate to what extent ideas from robust statistics are useful in analysing change points within high-dimensional heavy-tailed data streams.

Another line of work develops algorithms with robust components for change point analysis. In particular, in the univariate mean change setting, \citet{Fearnhead2019outlier} propose to swap the commonly used $\ell_2$-loss with other loss functions, including the biweight and Huber loss functions to enhance robustness against heavy-tailed errors in localising change points. \cite{li2021adversarially} deploy a robust mean estimator with a scanning window idea to estimate multiple change point locations under a more general Huber contamination framework. Their results show that, in terms of the minimax detection boundary, there is essentially no cost of relaxing the sub-Gaussian assumption to more flexible finite moment assumptions. Robust change point analysis methodologies have also been proposed in other contexts including change point detection in stump models \citep{mukherjee2022robust}, high-dimensional linear models \citep{liu2022change} and functional time series \citep{wegner2022robust}, as well as detecting covariance changes \citep{ramsay2020robust} and distributional changes \citep{chenouri2020robust}. There is also work exploring rank-based methods and focusing on univariate time series data \citep{Gombay1998rank,dehling2013longrange,betken2016wilcoxon,gerstenberger2018robustwilcoxon,betken2022rank}. We remark that some works focus on robust online change point detection \citep[e.g.][]{unnikrishnan2011minimax,cao2017robust,molloy2017misspecified}, which is different from the offline version that we study here\footnote{In an online change point analysis problem, one monitors the change points while collecting data.  In the offline context, the change point analysis is conducted retrospectively. }.

Closer to our high-dimensional mean change point setting, \cite{yu2022robust} and \cite{jiang2023robust}  both consider the testing problem \eqref{Eq:H_0_H_1} and propose robust methodology targeting at sparse and dense changes, respectively.  \citet{yu2022robust} formulate the problem as testing location parameter change, which in contrast to our model, allows the noise distribution to have mean parameter being infinite. Their methodology involves a U-statistic with an anti-symmetric and bounded kernel, followed by an $\ell_\infty$ aggregation. The power analysis of their proposed test (cf.~Theorem 3.3 therein) along with subsequent remarks provide finite sample results showing that their test is able to detect the change point when it is sufficiently away from the boundary. In particular, their Remark 4 suggests that detection is only possible for local alternative when the change point location satisfies
\[
t_0 \wedge (n-t_0) \geq c\sqrt{n\log(np)},
\]
 for some absolute constant $c >0$. In comparison, our results hold for the parameter space $\Theta(p, n, s, \rho)$ that covers all possible locations of change points. Moreover, as discussed in Remark 5 therein, their procedure achieves the sparse regime rate in $v^*_{{N}^{\otimes}(0, 1)}(p, n, s)$ up to a poly-logarithmic factor in $n$ and $p$ only when $t_0 = c n$ for some fixed constant $c \in (0,1)$. \cite{jiang2023robust} consider the same mean change point testing problem as ours but without sparsity constraints, while allowing for a form of weak spatial dependence across coordinates. In terms of methodology, they also utilise a robustified U-statistic and combine it with the self-normalisation technique. They derive the limiting distributions of the proposed test under the sequential asymptotics. It is discussed in Remark 2 therein that, asymptotically, their test achieves the dense rate $v^*_{{N}^{\otimes}(0, 1)}(p, n, p)$ up to a logarithmic factor in $n$, when the change point location satisfies $t_0 = c n$ for some fixed constant $c \in (0,1)$.
 
  In comparison to the results in \cite{yu2022robust} and \cite{jiang2023robust}, our results are non-asymptotic and reveal that when considering the whole parameter space $\Theta(p, n, s, \rho)$, where the change point locations may be arbitrarily close to the boundary, the fundamental difficulty of the testing problem changes drastically. In particular, the heavy-tailed distributions manifest a strong impact on the minimax testing rates and one can no longer achieve the Gaussian-like minimax testing rates, especially in the sparse regime. Moreover, our results are generally sharper in the sense that we characterise the minimax testing rates up to a factor of at most $\log\log(8n)$.

Lastly, we mention two recent works---\cite{comminges2021adaptive} and \cite{liu2021minimax}---that are technically related to ours. \cite{comminges2021adaptive} study the sparse sequence models where
\[
Y_i = \theta_i + \sigma \xi_i, \quad i = 1,\dotsc,p.
\]
The noise random variables $\xi_i$'s are i.i.d.\ with some distribution belonging to either $\mathcal{G}_{\alpha,K}$ or $\mathcal{P}_{\alpha,K}$, and the signal $\theta$ is assumed to be $\ell_0$-sparse with sparsity $s$. They provide minimax rates for estimating~$\|\theta\|_2$ among other results (cf.\ Table 1 therein) under these two noise classes. Our results recover theirs when $n$ is of constant order and provide a link between these two problems, while significantly generalising to the arbitrary $n$ case. To achieve the minimax estimation rates, \cite{comminges2021adaptive} first estimate $\theta$ via a penalised least squares estimator $\hat{\theta}$ in the sparse regime, and use $\|\hat{\theta}\|_2$ as an estimator for $\|\theta\|_2$. We adopt a different yet more intuitive hard-thresholding methodology in extracting information from sparse changes. Moreover, their upper bound rate under $\mathcal{P}_{\alpha, K}$ requires the assumption of bounded fourth moments, i.e.\ $\alpha \geq 4$. We investigate the more challenging case when $\alpha \in (2, 4)$, as well as $\alpha \leq 2$ in \Cref{Sec:extensions}, and unveil a previously unknown phase transition behaviour even when $n$ is of constant order. 

\cite{liu2021minimax} study the same testing problem \eqref{Eq:H_0_H_1} as ours under the Gaussian noise assumption while also considering spatial and temporal dependence. Their proposed testing procedure computes CUSUM-type statistics \citep[e.g.][]{page1955test} at each location on a dyadic grid. This also serves as the starting point of various procedures in our work. By comparing the results in \Cref{table:testing_rate} with the rate $v^*_{{N}^{\otimes}(0, 1)}(p, n, s)$ derived by \cite{liu2021minimax} under the Gaussianity assumption, we show that the heavy-tailed errors mainly affect the difficulty of testing sparse changes, whereas in the $\mathcal{P}_{\alpha,K}^{\otimes}$ case with $\alpha \in [2,4)$, the dense rate also changes dramatically. In the special case of $p =s =1$, our results (both upper and lower bounds in all cases) reduce to $\log\log(8n)$, which is the same rate as $v^*_{{N}^{\otimes}(0, 1)}(1, n, 1)$. This shows that, in the univariate setting, there is no extra cost of allowing for heavy-tailed errors in testing change point compared to Gaussian errors. 

\subsection{Notation}
We introduce the notation used throughout the paper. Let $\mathbb{Z}^+$ denote the set of positive integers. For $d \in \mathbb{Z}^+$, write $[d]:=\{1, \dotsc, d\}$. Let $\lceil \cdot \rceil$, $\lfloor \cdot \rfloor$ and $\Gamma(\cdot)$ denote the ceiling, floor and Gamma functions, respectively. Given $a, b \in \mathbb{R}$, denote $a \wedge b := \min(a,b)$ and $a \vee b := \max(a,b)$. For a set~$S$, use $\mathbbm{1}_S$ and $|S|$ to denote its indicator function and cardinality respectively. For a vector $v = \bigl(v(1), \dotsc, v(d)\bigr)^\top\in \mathbb{R}^d$, define 
$\|v\|_1 := \sum_{i=1}^d |v(i)|$
and $\|v\|_\infty := \max_{i \in [d]} |v(i)|$. For two vectors $v,w \in \mathbb{R}^{d}$, we use $\langle v,w\rangle$ to denote their inner product. For a matrix $A = (A_{ij})_{i \in [d_1], j \in [d_2]} = \bigl(A_j(i)\bigr)_{i \in [d_1], j \in [d_2]} \in \mathbb{R}^{d_1 \times d_2}$, denote the Frobenius norm $\|A\|_{\mathrm{F}} := \bigl(\sum_{i=1}^{d_1} \sum_{j=1}^{d_2} A_{ij}^2\bigr)^{1/2}$, the operator norm $\|A\|_2 := \max_{v\in \mathbb{R}^{d_2}, v\neq 0} \|Av\|_2/\|v\|_2$, the two-to-infinity norm $\|A\|_{2 \rightarrow \infty} := \max_{v\in \mathbb{R}^{d_2}, v\neq 0} \|Av\|_\infty/\|v\|_2$ and the max norm $\|A\|_{\max} := \max_{i \in [d_1], j \in [d_2]} |A_{ij}|$. For two probability measures $P$ and $Q$ on a measurable space $(\mathcal{X}, \mathcal{A})$, denote the total variation distance between them as $\mathrm{TV}(P,Q) := \sup_{A \in \mathcal{A}} |P(A) - Q(A)|$. If, in addition, $P$ and $Q$ are absolute continuous with respect to some base measure $\lambda$, then define the squared Hellinger distance between them as $\mathrm{H}^2(P,Q) := \int_{\mathcal{X}} \bigl(\sqrt{p(x)} - \sqrt{q(x)}\bigr)^2 \, \lambda(\mathrm{d}x)$, where $p$ and $q$ are the Radon--Nikodym derivatives of $P$ and $Q$ with respect to $\lambda$ respectively. When the distribution is clear from the context, let $\mathbb{P}$, $\mathbb{E}$ and $\mathrm{Var}$ be probability, expectation and variance operators respectively. {Finally, we write $a \gtrsim b$ if $a \geq C_1 b$, write $a \lesssim b$ if $a \leq C_2 b$, and write $a \asymp b$ if $C_3 b \leq a \leq C_4 b$, for some constants $C_1, C_2, C_3, C_4 > 0$ that depend only on $\alpha$, $K$, and $\varepsilon$, which are treated as constants throughout this work.}

\section{Testing under sub-Weibull noise distributions}
\label{Sec:hidimtest}

In this section, we consider the entries of the noise matrix $E$ to be independent random variables and each follows a  distribution belonging to the class $\mathcal{G}_{\alpha, K}$; see \Cref{def-galphak}. Recall the definitions of minimax testing rates, $v^*_{\mathcal{G}_{\alpha, K}}(p,n,s)$, the worst case testing error of a given test $\phi$, $\mathcal{R}_{\mathcal{G}}(\rho, \phi)$, and the minimax testing error $\mathcal{R}_{\mathcal{G}}(\rho)$, from \Cref{def:minimax_testing}. For notational simplicity, we use ${\mathcal{G}}$ in place of~${\mathcal{G}_{\alpha, K}^\otimes}$. 

As mentioned in \Cref{Sec:main result}, we shall establish an upper bound on $\mathcal{R}_{\mathcal{G}}(\rho)$ by developing two testing procedures, targeting at  dense and sparse change signals. We provide the details of these two testing procedures with corresponding theoretical guarantees leading to the dense and sparse rates in Sections~\ref{subSec:subWeibull_dense} and \ref{subSec:subWeibull_sparse}. The minimum between the two rates serves as an upper bound on $\mathcal{R}_{\mathcal{G}}(\rho)$ and we prove and discuss its optimality in \Cref{subsec:discuss_gap_subweibull}.

\subsection{Testing for dense signals}
\label{subSec:subWeibull_dense}
To derive the dense rate, we consider the testing procedure that is used in \cite{liu2021minimax}. Consider $\mathcal{T} := \bigl\{1,2,4,\dotsc,2^{\lfloor\log_2(n/2)\rfloor}\bigr\}$ and a CUSUM-type statistic
\begin{equation*}
Y_t := \frac{\sum_{i=1}^t X_i - \sum_{i=1}^t X_{n+1-i}}{\sqrt{2t}}.
\end{equation*}
We define our test as
\begin{equation} \label{eq:test_sub_dense}
\phi_{\mathcal{G}, \mathrm{dense}}:= \mathbbm{1}_{\{\max_{t \in \mathcal{T}} A_t > r\}},
\end{equation}
where 
\begin{equation} \label{eq:A_dense}
    A_{t} := \sum_{j=1}^p \bigl\{Y_{t}^2(j)-1\bigr\}
\end{equation}
and $r > 0$ is the detection threshold specified in \Cref{thm:weibullupperbound_dense}. Note that it suffices to test for a change point over the dyadic grid $\mathcal{T}$ since for any true change point location $t_0\in[n-1]$ under the alternative, there exists some $t\in\mathcal{T}$ such that $t\leq t_0\leq 2t$, approximating the true change location up to a constant factor. The logarithmic size of $\mathcal{T}$ is the main reason behind the appearance of the $\log\log (8n)$ terms in our bounds below. The following theorem establishes the theoretical guarantee of the test $\phi_{\mathcal{G}, \mathrm{dense}}$. 

\begin{thm} \label{thm:weibullupperbound_dense}
Let $0 < \alpha \leq 2$ and $K > 0$. For any $\varepsilon \in (0,1)$, there exist constants $C_1, C_2> 0$ depending only on $\alpha$, $K$ and $\varepsilon$, such that the test $\phi_{\mathcal{G}, \mathrm{dense}}$ defined in~\eqref{eq:test_sub_dense} with
\[
r = C_1\bigl(\sqrt{p\log\log(8n)} + \log\log(8n)  \bigr)
\]
satisfies 
\[
\mathcal{R}_{\mathcal{G}}(\rho, \phi_{\mathcal{G}, \mathrm{dense}}) \leq \varepsilon,
\]
as long as $\rho^2 \geq C_2 v_{\mathcal{G}, \mathrm{dense}}^{\mathrm{U}},$ where 
\[
v_{\mathcal{G}, \mathrm{dense}}^{\mathrm{U}} := \sqrt{p \log \log(8n)} + \log \log(8n).
\]
\end{thm}

Note that this simple test actually achieves the same rate in the dense regime  as $v_{N^{\otimes}(0,1)}^{*}$ defined in~\eqref{Eq:GaussianRate}, even though the noise distributions possess heavier tails than Gaussian and sub-Gaussian distributions. One key technical ingredient is a careful analysis of the probability of the Type I error using \Cref{lemma:exp_decay_sum} instead of a crude union bound.

\subsection{Testing for sparse signals} \label{subSec:subWeibull_sparse}

To derive the sparse rate, we employ a sample-splitting testing procedure similar to that proposed in \citet[Section 5.1]{kovacs2024optimistic}. Intuitively, we first use half of the data to identify coordinates that exhibit strong signals of change, and then use the other half to aggregate the selected `signal' coordinates. Such a methodology is applicable for testing potential change locations $t \in \mathcal{T}\setminus\{1\}$ and we deal with testing the special case of $t=1$ separately. 

To be specific, for $t \in \mathcal{T}\setminus\{1\}$, we define a sample-splitting version of \eqref{eq:A_dense} that
\begin{equation}\label{eq:test_sub_sparse}
    \phi_{\mathcal{G}, \mathrm{sparse}}:= \mathbbm{1}_{\{\max_{t \in \mathcal{T}\setminus\{1\}} A_{t,a} > r \}} \vee \mathbbm{1}_{\{A_{1,a} > r_1\}},
\end{equation}
with
\begin{equation}\label{eq:samplespliting}
    Y_{t,1} := \frac{\sum_{i=1}^{t/2} X_{2i-1}-\sum_{i=1}^{t/2} X_{n+2-2i}}{\sqrt{t}}, \qquad Y_{t,2} := \frac{\sum_{i=1}^{t/2} X_{2i}-\sum_{i=1}^{t/2} X_{n+1-2i}}{\sqrt{t}}
\end{equation}
and
\begin{align}\label{eq:A_t,a}
A_{t,a} := \begin{cases}
    \sum_{j=1}^p \bigl\{Y_{t,1}^2(j)-1\bigr\}\mathbbm{1}_{ \{ |Y_{t,2}(j)| \geq a \} }, &t \geq 2,  \\
    \sum_{j=1}^p \bigl\{Y_{t}^2(j)-1\bigr\}\mathbbm{1}_{ \{ |Y_{t}(j)| \geq a \} }, &t = 1,
\end{cases}
\end{align}
where $a, r, r_1$ are specified in \Cref{thm:weibullupperbound_sparse}. 

\begin{thm} \label{thm:weibullupperbound_sparse}
Let $0 < \alpha \leq 2$ and $K > 0$. For any $\varepsilon \in (0,1)$, there exist constants $C_1, C_2, C_3, C_4> 0$ depending only on $\alpha$, $K$ and $\varepsilon$, such that the test $\phi_{\mathcal{G}, \mathrm{sparse}}$ defined in~\eqref{eq:test_sub_sparse} with
\[
a = C_1\bigl[\log^{1/\alpha}(ep/s) + s^{-1/2} \log^{1/2}\{\log(8n)\}\bigr],
\]
\[
r = C_2\bigl(\sqrt{s\log\log(8n)} + \log\log(8n)  \bigr) \quad \mbox{and} \quad r_1 = C_3 s\log^{2/\alpha}(ep/s),
\]
satisfies that
\[
\mathcal{R}_{\mathcal{G}}(\rho, \phi_{\mathcal{G}, \mathrm{sparse}}) \leq \varepsilon,
\]
as long as $\rho^2 \geq C_4 v_{\mathcal{G}, \mathrm{sparse}}^{\mathrm{U}},$ where 
\[
v_{\mathcal{G}, \mathrm{sparse}}^{\mathrm{U}} := s\log^{2/\alpha}(ep/s) + \log\log(8n).
\]
\end{thm}

The idea of selecting coordinates via hard-thresholding has been widely used and in particular, in the change point context, considered in \citet{liu2021minimax} and \citet{kovacs2024optimistic} under the Gaussian noise assumption. Our use of sample-splitting prompts the independence between the coordinate selection step and the $\ell_2$-aggregation step.  It simplifies the analysis while achieving the optimal testing rate, as we will show in \Cref{subsec:discuss_gap_subweibull}. 

\subsection{Minimax optimality}\label{subsec:discuss_gap_subweibull}

We derive lower bounds on the minimax testing rate and discuss the optimality of our testing procedures in this section. First, note that the theoretical guarantees in
Theorems~\ref{thm:weibullupperbound_dense} and~\ref{thm:weibullupperbound_sparse} hold for any sparsity level $s$. In other words, for any given $s \in [1,p]$, we can simultaneously run the two testing procedures described in Sections~\ref{subSec:subWeibull_dense} and~\ref{subSec:subWeibull_sparse} and take $\phi_{\mathcal{G},\mathrm{dense}} \vee \phi_{\mathcal{G}, \mathrm{sparse}}$ as our test. This leads to an upper bound 
\begin{equation} \label{Eq:weibull_minimum_upperbound}
v_{\mathcal{G}, \mathrm{dense}}^{\mathrm{U}} \wedge v_{\mathcal{G}, \mathrm{sparse}}^{\mathrm{U}} = \Big\{s\log^{2/\alpha}(ep/s) \wedge \sqrt{p\log\log(8n)}\Big\} + \log\log(8n)
\end{equation}
on the minimax testing rate $v^*_{\mathcal{G}_{\alpha, K}^{\otimes}}(p,n,s)$. The following theorem presents a corresponding lower bound.
\begin{thm}\label{thm:subweibull-lowerbound-minimum}
    Let $0 < \alpha \leq 2$, $K \geq K_\alpha$ and $s \geq c$, for some absolute constant $c \geq 1$ and some constant $K_\alpha > 0$ depending only on $\alpha$. There exists some constant $c' > 0$ depending only on $\alpha$ and~$K$, such that $\mathcal{R}_{\mathcal{G}}(\rho) \geq 1/2$
whenever $\rho^2 \leq c' v_{\mathcal{G}}^{\mathrm{L}}$, where 
\[
v_{\mathcal{G}}^{\mathrm{L}} := \Big\{s\log^{2/\alpha}(ep/s) \wedge \sqrt{p \{\log \log(8n)\}^{\omega_1}}\Big\} + \log \log(8n)
\]
and $\omega_1 = \mathbbm{1}_{\bigl\{s > \sqrt{p \log \log(8n)}\bigr\}}$.
\end{thm}
By combining the lower bound $v_{\mathcal{G}}^{\mathrm{L}}$ in \Cref{thm:subweibull-lowerbound-minimum} and the upper bound in \eqref{Eq:weibull_minimum_upperbound}, we are able to quantify the minimax test rate up to a factor of $\sqrt{\log\log(8n)}$ and conclude that our testing procedures are minimax rate optimal up to an iterated logarithmic factor. Moreover, a closer look into the sparse and dense regimes, as defined in \Cref{Sec:main result}, reveals that our results exactly quantify $v^*_{\mathcal{G}_{\alpha, K}^{\otimes}}(p,n,s)$ in almost all regimes of sparsity. 

Both the upper and lower bounds consist of a minimum of two terms, directly involving the sparsity level $s$ and not.  As $s$ grows, the final rate undergoes a phase transition, namely from a sparse regime to a dense one.  To better understand the transition, we focus on the relationship between $p$ and $s$, setting the boundary between the sparse and the dense regime to be the solution to $s^*_{\mathcal{G}}\log^{2/\alpha}(ep/s^*_{\mathcal{G}}) = \sqrt{p}$, i.e.
\begin{equation} \label{Eq:boundary_weibull}
s^*_{\mathcal{G}} \asymp \frac{\sqrt{p}}{\log^{2/\alpha}(ep)}.
\end{equation}

We summarise this phenomenon in \Cref{table:testing_rate}.  From the table, it is clear that our upper and lower bounds match exactly in the entire sparse regime (i.e.\ $s < s^*_{\mathcal{G}}$) and the majority region of the dense regime (i.e. $s > \sqrt{p\log\log(8n)}$). The $\sqrt{\log\log(8n)}$ gap between the upper and lower bounds, only exists in the region $s^*_{\mathcal{G}} \leq s \leq \sqrt{p\log\log(8n)}$ within the dense regime.  Closing such gap is challenging and a similar gap exists even when each entry of the noise matrix follows a sub-Gaussian, yet Gaussian, distribution; see \citet[][Section~7.1]{pilliat2020optimalhighdim}, where it is suggested that a procedure exploring the exact distribution of the noise might be required to close this gap.

\begin{table}[ht!]
\centering
\def\arraystretch{1.8}
\begin{tabular}{|cc|c|c|}
\hline
\multicolumn{2}{|c|}{}               &  Upper bound &  Lower bound \\ \hline
\multicolumn{1}{|c|}{\multirow{2}{*}{$\mathcal{G}_{\alpha, K}^{\otimes}$}} & Dense  & $\sqrt{p \log \log(8n)} + \log \log(8n)$          &     $\sqrt{p \{\log \log(8n)\}^{\omega_1}} + \log \log(8n)$       \\ \cline{2-4} 
\multicolumn{1}{|c|}{}                   & Sparse &   $s\log^{2/\alpha}(ep/s) + \log\log(8n)$        &    $s\log^{2/\alpha}(ep/s) + \log\log(8n)$       \\ \hline
\end{tabular}
\caption{Bounds on the minimax testing rates in the sub-Weibull noise distribution class $\mathcal{G}_{\alpha, K}^\otimes$, where $\omega_1 = \mathbbm{1}_{\bigl\{s >\sqrt{p \log \log(8n)}\bigr\}}$. Upper bounds are obtained in Theorems~\ref{thm:weibullupperbound_dense} and \ref{thm:weibullupperbound_sparse}. Lower bounds are obtained in \Cref{thm:subweibull-lowerbound-minimum}.}
\label{table:testing_rate}
\end{table}

To highlight the effects of sub-Weibull distributions on the minimax test rate, we note that allowing heavier tails does not affect the minimax testing rate when $s > \sqrt{p\log\log(8n)}$, relative to the results under Gaussian noise assumptions; see \eqref{Eq:GaussianRate}. However, the tail behaviour, quantified by the parameter $\alpha$, does affect the minimax rate in the sparse regime and hence the transition boundary. Specifically, as $\alpha$ decreases (i.e.\ the tail becomes heavier), the sparse rate increases, meaning that it is fundamentally more difficult to detect sparse changes as the tail of the noise distribution becomes heavier. As a prelude to \Cref{Sec:awayfrombonudary}, we also show that a modification of the test $\phi_{\mathcal{G}, \text{sparse}}$ in \Cref{subSec:subWeibull_sparse} can achieve a sparse rate that is independent of $\alpha$, if a different alternative hypothesis is considered, where the change point is known to be at least $\log(ep/s)+s^{-1}\log\log(8n)$ away from the end points $1$ and $n$.

\section{Testing under finite moment noise distributions} \label{Sec:robusttest}

In this section, we consider the case when $P_e \in \mathcal{P}_{\alpha,K}^{\otimes}$, or equivalently, we assume that the distribution of each entry in the noise matrix $E$ has only finite $\alpha$-th moments, for some constant $\alpha \geq 2$, see \Cref{def-palphak}. Compared to the $\mathcal{G}_{\alpha, K}$ class of distributions considered in \Cref{Sec:hidimtest}, where standard CUSUM-type testing procedures already achieve near-optimal minimax testing rates, the $\mathcal{P}_{\alpha, K}$ class of distributions include a much wider range of noise distributions, e.g.~$t$ distributions and {centred} Pareto distributions. As a result, it poses a much larger statistical challenge.  New approaches to tackle the testing problem are thus required.

Similar to \Cref{Sec:hidimtest}, we write the worst case testing error as $\mathcal{R}_{\mathcal{P}}(\rho, \phi)$ and the minimax testing error as $\mathcal{R}_{\mathcal{P}}(\rho)$. We again  assume the sparsity level to be known and derive the dense and sparse testing rates separately in Sections~\ref{subSec:robust_dense} and~\ref{subSec:robust_sparse}.

\subsection{Testing for dense signals} \label{subSec:robust_dense}

We consider a testing procedure built on the median-of-means-type statistics. For $i \leq n/2$, we denote $Z_i := (X_i - X_{n-i+1})/\sqrt{2}$. For $t \in \mathcal{T}$, we split $\{Z_1, \dotsc, Z_t\}$ into $G_t$ groups of equal size (assuming that $t$ is always a multiple of chosen $G_t$ for simplicity) that \[
\mathcal{Z}_{t,1}, \mathcal{Z}_{t,2}, \dotsc, \mathcal{Z}_{t,G_t},
\]
where each group contains $t/G_t \geq 1$ elements and the number of groups $G_t$ is specified later in~\eqref{Eq:robust_theorem_parameters}. 
Set $V_{t,g} \in \mathbb{R}^p$ with
\begin{equation} \label{eq:V_tg(j)}
V_{t,g}(j) :=   \overline{Z}^2_{t,g}(j) - \frac{G_t}{t}, \quad g \in [G_t],
\end{equation}
where $\overline{Z}_{t,g} \in \mathbb{R}^p$ is the sample mean of the $g$-th group. This quantity $V_{t,g}$ can be thought as a scaled version of the statistic $A_t$ defined in \eqref{eq:A_dense}, but computed using only a subset of the data. To achieve robustness against heavy-tailed errors, we consider the following median-of-means-type statistic
\begin{equation}\label{eq:A_mom}
    A_{t}^{\mathrm{MoM}}:= t \cdot \mathrm{median} \Biggl(  \sum_{j=1}^p V_{t,1}(j),  \sum_{j=1}^p V_{t,2}(j), \dotsc, \sum_{j=1}^p V_{t,G_t}(j)\Biggl).
\end{equation}
Our test is denoted as 
\begin{equation}\label{eq:finitemoment_dense_test}
    \phi_{\mathcal{P},\mathrm{dense}} := \mathbbm{1}_{\left\{\max_{t \in \mathcal{T}} A_{t}^{\mathrm{MoM}}/r_t > 1 \right\}},
\end{equation}
with the detection threshold $r_t$ specified in \eqref{Eq:robust_theorem_parameters}. Before presenting the theoretical guarantee of the test $\phi_{\mathcal{P},\mathrm{dense}}$ in \Cref{thm:finitemoment_upperbound_dense}, we first briefly explain the significance of median-of-means-type statistics and the novelty of our procedure. 

Median-of-means-type statistics like \eqref{eq:A_mom} have been applied in a wide range of statistical problems \citep[e.g.][]{lugosi2019sub,lerasle2011robust,lecue2020robust,humbert2022robust,kwon2021mom}. The most well-known and simplest form is its univariate mean estimation version. Suppose that we have i.i.d.~data of sample size $n$ with mean $\mu$ and variance $\sigma^2$. The median-of-means estimator $\hat{\mu}^{\mathrm{MoM}}$ is obtained by first partitioning the data into $G$ groups of equal size, then calculating the sample mean within each group and finally computing the median of these $G$ sample means. It is shown in \citet[][Theorem 2]{lugosi2019survey} that, for $\delta \in (0, 1)$, when the number of groups $G$ is chosen to be at least $8\log(1/\delta)$, with probability at least $1-\delta$, the estimator $\hat{\mu}^{\mathrm{MoM}} = \hat{\mu}^{\mathrm{MoM}}(\delta)$ satisfies that
\[
|\hat{\mu}^{\mathrm{MoM}} - \mu| \leq \sigma\sqrt{\frac{32\log(1/\delta)}{n}}.
\]
Thus, the median-of-means estimator can achieve sub-Gaussian performance in mean estimation under only the assumption of finite second moment. 

However, in our context, the aforementioned methodology is not applicable for testing potential change point that is too close to the boundary, as we will not have enough data to split into the required number of groups to ensure good statistical guarantees. Therefore, for $t \in \mathcal{T}$ such that $t \leq \Delta$ with the threshold $\Delta$ specified in \eqref{Eq:robust_theorem_parameters}, we directly take the median of $t$ statistics in \eqref{eq:A_mom}, i.e.\ $G_t=t$. We now present the theoretical guarantee of the test $\phi_{\mathcal{P},\mathrm{dense}}$ in~\eqref{eq:finitemoment_dense_test}.

\begin{thm} \label{thm:finitemoment_upperbound_dense}
Assume $\alpha \geq 2$. For any $\varepsilon \in (0,1)$, there exist $C_1, C_2> 0$ depending only on $\alpha$, $K$ and~$\varepsilon$, such that the test $\phi_{\mathcal{P},\mathrm{dense}}$ defined in~\eqref{eq:finitemoment_dense_test} with
\begin{align} \label{Eq:robust_theorem_parameters}
 r_t =  C_1p^{(1/2)\vee(2/\alpha)}G_t, \quad G_t = t \wedge \Delta \quad  \mbox{and} \quad \Delta = 2^{3+ \lceil\log_2 \log \log(8n)\rceil},
\end{align}
satisfies that
\[
\mathcal{R}_{\mathcal{P}}(\rho, \phi_{\mathcal{P},\mathrm{dense}}) \leq \varepsilon,
\]
as long as $\rho^2 \geq C_2 v_{\mathcal{P}, \mathrm{dense}}^{\mathrm{U}}$, where 
\[
v_{\mathcal{P}, \mathrm{dense}}^{\mathrm{U}} := p^{(2/\alpha) \vee (1/2)} \log \log(8n)
\]
\end{thm}

One challenge in our context is analysing the performance of the test $\phi_{\mathcal{P}, \mathrm{dense}}$ when $\alpha \in [2,4]$. Since we compute a second-order statistic $V_{t,g}$ within each group $g$, standard variance-based analysis would require a bounded fourth moment condition on the distribution. However, through a more refined analysis, we extend our results to this more demanding case of $\alpha \in [2,4]$. In this setting, the dense testing rate $v_{\mathcal{P}, \mathrm{dense}}^{\mathrm{U}}$ is affected by $\alpha$, and we reveal a phase transition in the rate at $\alpha = 4$.

An even more challenging scenario arises when the distribution of each entry of $E$ lacks a finite variance. In this case, an alternative test to $\phi_{\mathcal{P},\mathrm{dense}}$ is required, as the mean of the aforementioned second-order statistic $V_{t,g}$ is no longer guaranteed to be finite. We defer a detailed discussion of this setting to Section~\ref{sec:<2moment}.

\subsection{Testing for sparse signals} \label{subSec:robust_sparse}

To derive the sparse rate, we employ a mean estimator satisfying a general condition detailed in \Cref{Con:RSM} in \Cref{subsubsection:proof_rsm} to construct our test. There are potentially many choices of such a mean estimator, but one specific choice $\hat{\mu}^{\mathrm{RSM}}$ is given in \cite{prasad2019unified}: 
\begin{equation}\label{eq:prasad_robust_sparse_mean}
    \hat{\mu}^{\mathrm{RSM}}_{n,s}(\{W_i\}_{i=1}^n;\eta) := \inf_{\mu \in \mathcal{L}_s} \sup_{u \in \mathcal{N}^{1/2}_{2s}(\mathcal{S}^{p-1})} \big| u^{\top}\mu - \mathrm{1DRobust}(\{u^{\top}W_i\}_{i=1}^n,\eta/ (6ep/s)^s) \big|,
\end{equation}  
where $W_1, \dotsc, W_n \in \mathbb{R}^p$ are input data, $\mathcal{L}_s := \{v\in \mathbb{R}^p: \|v\|_0 \leq s\}$ is the set of $s$-sparse vectors in~$\mathbb{R}^p$, $\mathcal{N}^{1/2}_{2s}(\mathcal{S}^{p-1})$ is a $(1/2)$-cover of the set of $2s$-sparse unit vectors with cardinality $|\mathcal{N}^{1/2}_{2s}(\mathcal{S}^{p-1})| \leq (6ep/s)^s$ \citep{vershynin2009role}, and $\mathrm{1DRobust}$ is a univariate robust mean estimator defined in \citet[][Algorithm 2]{prasad2019unified}. Other univariate robust mean estimators can be considered in place of $\mathrm{1DRobust}$, including the median-of-means and trimmed mean variants \citep{lugosi2021robust}. Notably, the estimator in \eqref{eq:prasad_robust_sparse_mean} achieves a near-optimal statistical guarantee for sparse mean estimation \citep[][Corollary 11]{prasad2019unified}, despite its high computational complexity, which scales exponentially in $s$.

We describe our test $\phi^{\mathrm{RSM}}_{\mathcal{P},\mathrm{sparse}}$ using $\hat{\mu}^{\mathrm{RSM}}_{n,s}(\{W_i\}_{i=1}^n;\eta)$. For $\tilde{\Delta}_1$ specified in \eqref{Eq:robust_theorem_parameters_sparse_rsm} and for $t \leq \tilde{\Delta}_1$, we use the non-robust statistic $A_{t,a}$ as defined in~\eqref{eq:A_t,a}. For $t \in \mathcal{T}\cap\{t>\tilde{\Delta}_1\}$, we construct the statistic from the $\ell_2$-norm of this robust sparse mean estimator:
\[
    A^{\mathrm{RSM}}_{t} := t\bigl\|\hat{\mu}_{t,s,\eta_t}^{\mathrm{RSM}}\bigr\|_2^2 = t\bigl\|\hat{\mu}_{t,s}^{\mathrm{RSM}}(\{Z_i\}_{i=1}^t; \eta_t)\bigr\|_2^2.
\]
With all the parameters $a, \tilde{\Delta}_1, \tilde{r}_t, r^{\mathrm{RSM}}_t$ and $\eta_t$ specified later in \eqref{Eq:robust_theorem_parameters_sparse_rsm}, we define
\begin{equation}\label{eq:finitemoment_sparse_rsm}
    \phi^{\mathrm{RSM}}_{\mathcal{P},\mathrm{sparse}} := \mathbbm{1}_{\left\{\max_{t \in \mathcal{T}\cap \{t \leq \tilde{\Delta}_1\}} A_{t,a}/\tilde{r}_t > 1 \right\}} \vee \mathbbm{1}_{\left\{\max_{t \in \mathcal{T}\cap \{t > \tilde{\Delta}_1\}} A^{\mathrm{RSM}}_{t}/r^{\mathrm{RSM}}_t > 1 \right\}}.
\end{equation}
The theoretical guarantee of $\phi^{\mathrm{RSM}}_{\mathcal{P},\mathrm{sparse}}$ is established as follows.

\begin{thm} \label{thm:finitemoment_upperbound_sparse_improve}
 Assume $\alpha \geq 4$. For any $\varepsilon \in (0,1)$, there exist $C_1, C_2, C_3, C_4, C_5> 0$ depending only on $\alpha$, $K$ and $\varepsilon$, such that the test $\phi^{\mathrm{RSM}}_{\mathcal{P},\mathrm{sparse}}$ defined in~\eqref{eq:finitemoment_sparse_rsm} with
\begin{equation} \label{Eq:robust_theorem_parameters_sparse_rsm}
\begin{aligned} 
a = C_1\bigl((p/s)^{1/\alpha} + s^{-1/2} \log^{1/2}(\log\tilde{\Delta}_1)\bigr), &\quad \tilde{r}_t  = C_2 \Bigl( s(p/s)^{2/\alpha}\mathbbm{1}_{\{t=1\}}+ \sqrt{s\log \tilde{\Delta}_1}\mathbbm{1}_{\{t>1\}}\Bigr),  \\
\eta_t = \exp\biggl\{s\log(ep/s) - \frac{t\wedge \tilde{\Delta}_2}{C_3}\biggr\}, &\quad r^{\mathrm{RSM}}_t = C_4(t \wedge \tilde{\Delta}_2), \\
\tilde{\Delta}_1 =  C_3\bigl(s\log(ep/s) + \log(16/\varepsilon)\bigr) \quad &\mbox{and} \quad \tilde{\Delta}_2 = C_3\bigl(s\log(ep/s) + \log(16\log(2n)/\varepsilon)\bigr),
\end{aligned}
\end{equation}
satisfies that
\[
\mathcal{R}_{\mathcal{P}}(\rho,  \phi^{\mathrm{RSM}}_{\mathcal{P},\mathrm{sparse}} ) \leq \varepsilon,
\]
as long as $\rho^2 \geq C_5 v_{\mathcal{P}, \mathrm{sparse}}^{\mathrm{U}}$, where 
\begin{equation}\label{eq-sparse-upper-poly}
v_{\mathcal{P}, \mathrm{sparse}}^{\mathrm{U}} := s(p/s)^{2/\alpha} + \log\log(8n).
\end{equation}
\end{thm}
The main reason we separate $\mathcal{T}$ into different regions in our test \eqref{eq:finitemoment_sparse_rsm} is that \eqref{eq:prasad_robust_sparse_mean} cannot be applied when $t$ is too close to the boundary in order to achieve the required statistical performance.  We, therefore, resort to the non-robust testing statistics $A_{t,a}$ for $t$ close to boundaries. 

A significant drawback of using the robust sparse mean estimator to construct our test, $\phi^{\mathrm{RSM}}_{\mathcal{P},\mathrm{sparse}}$, is its high computational cost, which scales exponentially with $s$. We defer the discussion of this issue and a two-component remedy that achieves the same rate, which is in fact optimal, in polynomial time to Section~\ref{subsubsec:polynomial-time-alg}.

\subsection{Minimax optimal testing using a polynomial-time procedure}
\label{subsec:discuss_gap_robust}

Similar to \Cref{subsec:discuss_gap_subweibull}, we first derive lower bounds on the minimax testing rate under $\mathcal{P}_{\alpha, K}$ and examine the optimality of our testing procedures in both the dense and sparse regimes in \Cref{subsubsec:minimax-polynomial}. Then, in Sections~\ref{subsubsec:mom-sparse} and \ref{subsubsec:polynomial-time-alg}, we address the computational intractability issue of $\phi^{\mathrm{RSM}}_{\mathcal{P},\mathrm{sparse}}$ by combining it with a median-of-means-type test, yielding a procedure that is both minimax optimal in the sparse regime and computationally feasible, with complexity polynomial in $p$ and $n$.

\subsubsection{Minimax optimality}
\label{subsubsec:minimax-polynomial}
For any given $s \in [1,p]$, by simultaneously running $\phi_{\mathcal{P},\mathrm{dense}}$ and $\phi^{\mathrm{RSM}}_{\mathcal{P},\mathrm{sparse}}$ when $\alpha \geq 4$ and only running $\phi_{\mathcal{P},\mathrm{dense}}$ when $2 \leq \alpha < 4$, we obtain an upper bound
\begin{equation}
\label{Eq:finitemoment_minimum_upperbound}
    \begin{cases}
\bigl\{s(p/s)^{2/\alpha} \wedge \sqrt{p}\log\log (8n) \bigr\} + \log \log(8n), &\text{when } \alpha \geq 4, \\
p^{2/\alpha}\log \log (8n), &\text{when } 2 \leq \alpha < 4,
\end{cases}
\end{equation}
on the minimax testing rate $v^*_{\mathcal{G}_{\alpha, K}^{\otimes}}(p,n,s)$. This upper bound also implies the one presented in \eqref{eq:rate-result-poly}, i.e.\ 
$$\Big\{s(p/s)^{2/\alpha}\wedge p^{\frac{2}{\alpha}\vee \frac{1}{2}}\Big\}\log\log(8n).$$
The following result provides a corresponding lower bound.

\begin{thm}\label{thm:finitemoment-lowerbound-minimum}
    Let $\alpha \geq 2$, $K \geq K_\alpha$ and $s \geq c$, for some absolute constant $c \geq 1$ and some constant $K_\alpha > 0$ depending only on $\alpha$. There exists some constant $c' > 0$ depending only on $\alpha$ and $K$, such that $\mathcal{R}_{\mathcal{P}}(\rho) \geq 1/2$
whenever $\rho^2 \leq c' v_{\mathcal{P}}^{\mathrm{L}}$, where 
\[
v_{\mathcal{P}}^{\mathrm{L}} := \Big\{s(p/s)^{2/\alpha} \wedge p^{(2/\alpha) \vee (1/2)}  (\log \log(8n))^{\omega_2} \Bigr\} + \log \log(8n)
\]
and $\omega_2 =(1/2) \mathbbm{1}_{\bigl\{s > \sqrt{p \log \log(8n)}\bigr\} \cap \{\alpha \geq 4\} }$.
\end{thm}

Similar to the exponentially-decaying tail case, by combining the lower bound $v_{\mathcal{P}}^{\mathrm{L}}$ from \Cref{thm:finitemoment-lowerbound-minimum} with the upper bound in \eqref{Eq:finitemoment_minimum_upperbound}, we quantify the minimax test rate up to a factor of $\log\log(8n)$ and conclude that our testing procedures are minimax rate near-optimal. Importantly, our results reveal a critical phenomenon that the minimax testing rate is independent of $s$ when $\alpha \in [2,4]$.  As in this case, both \eqref{Eq:finitemoment_minimum_upperbound} and $v_{\mathcal{P}}^{\mathrm{L}}$ reduce to $p^{2/\alpha}$, up to a factor of $\log\log(8n)$. In other words, knowing the change is sparse does not make the testing problem easier, and it is fundamentally impossible to exploit the sparse structure of the change in pursuit of better results.
We further discuss its consequence in the language of sparse and dense regimes below.

To understand the effect of sparsity, we again ignore the iterated logarithmic factor in $n$ and focus on the relationship between $s$ and $p$.  The boundary between the dense and sparse regimes is obtained by determining which term dominates the upper/lower bound rate in the minimum operator.  As defined in \Cref{Sec:main result}, the boundary is   
\begin{equation} \label{Eq:sparsity_boundary_robust}
    s^*_{\mathcal{P}} := p^{\frac{1}{2} - (\frac{1}{\alpha-2} \wedge \frac{1}{2})},
\end{equation}
which satisfies $s_{\mathcal{P}}^*(p/s_{\mathcal{P}}^*)^{\frac{2}{\alpha}} = p^{\frac{2}{\alpha} \vee \frac{1}{2}}$. The dense and sparse regimes are those where the sparsity level~$s$ is directly involved in the rate or not. 
The characterisation is summarised in \Cref{table:testing_rate_finite}.
Notably, when $2 \leq \alpha \leq 4$, we always have  
\[
    \frac{1}{\alpha-2} \wedge \frac{1}{2} = \frac{1}{2},
\]
which means that there is no sparse regime in this extremely heavy-tailed setting. 

\begin{table}[ht!]
\centering
\def\arraystretch{1.8}
\begin{tabular}{|cc|c|c|}
\hline
\multicolumn{2}{|c|}{}               &  Upper bound &  Lower bound \\  \hline
\multicolumn{1}{|c|}{\multirow{2}{*}{$\mathcal{P}_{\alpha, K}^{\otimes}$}} & Dense ($\alpha \geq 2$)  &     $p^{(2/\alpha) \vee (1/2)}\log\log(8n)$       &       $p^{(2/\alpha) \vee (1/2)}  (\log \log(8n))^{\omega_2} + \log \log(8n)$      \\ \cline{2-4} 
\multicolumn{1}{|c|}{}                   & Sparse ($\alpha \geq 4$) &    $s(p/s)^{2/\alpha} + \log\log(8n)$         &   $s(p/s)^{2/\alpha} + \log\log(8n)$         \\ \hline
\end{tabular}
\caption{Bounds on the minimax testing rates under finite moment noise distribution class $\mathcal{P}_{\alpha, K}^\otimes$ with $\alpha \geq 2$, where $\omega_2 =(1/2) \mathbbm{1}_{\bigl\{s > \sqrt{p \log \log(8n)}\bigr\} \cap \{\alpha \geq 4\} }$. 
}
\label{table:testing_rate_finite}
\end{table}

From the table, we observe that our upper and lower bounds match exactly across the entire sparse regime ($s < s^*_{\mathcal{P}}$) when it exists ($\alpha \geq 4$). In the dense regime, the upper and lower bounds are off by a factor of order at most $\log\log(8n)$. We briefly note that the upper bound in the special case of $\alpha = 2$ can be improved to $p+\log\log(8n)$, matching the lower bound in this case up to constants. This is achieved as a by-product when we consider noise distributions with no more than two finite moments in \Cref{sec:<2moment}.

We now discuss the effects of $\alpha$ on the minimax testing rates. When both the dense and sparse regimes exist ($\alpha \geq 4$), we observe from \Cref{table:testing_rate_finite} that the dense rate is not affected by $\alpha$ since $2/\alpha \leq 1/2$. Moreover, even compared to the dense rate $\sqrt{p\log\log(8n)}$ under Gaussian noise assumptions, the cost of heavy-tailedness is minimal. However, the sparse rates are completely different from their counterparts in \Cref{subsec:discuss_gap_subweibull}, implying a significant increase of difficulty in detecting sparse changes under heavy-tailed noises. We note that this difficulty can be largely mitigated if one assumes that the change point is away from the boundary, and we discuss this interesting extension in \Cref{Sec:awayfrombonudary}. Finally, as $\alpha$ further decreases to between $2$ and $4$, the sparse regime becomes empty, and the dense rates are also affected by $\alpha$.

\subsubsection{A median-of-means-type test}
\label{subsubsec:mom-sparse}
As mentioned in \Cref{subSec:robust_sparse}, one challenge of using a robust estimator such as~\eqref{eq:prasad_robust_sparse_mean} to construct our testing procedure $\phi^{\mathrm{RSM}}_{\mathcal{P},\mathrm{sparse}}$ is its computational intractability. This stems from the need to project data onto every $2s$-sparse unit vector or its covering set, causing the computational complexity to scale exponentially in $s$. This issue is, in fact, common in high-dimensional robust statistics.

From \Cref{table:testing_rate_finite}, we observe that $\phi^{\mathrm{RSM}}_{\mathcal{P},\mathrm{sparse}}$ is used to establish the upper bound rate only in the sparse regime $s < s_{\mathcal{P}}^* = p^{\frac{1}{2} - (\frac{1}{\alpha-2} \wedge \frac{1}{2})}$, where it achieves minimax optimality. A natural question, then, is whether a polynomial-time algorithm can also attain minimax optimality. To answer this, we first examine a computationally-efficient test that combines the median-of-means approach with a hard-thresholding step for coordinate selection.

Recall that $Z_i = (X_i - X_{n-i+1})/\sqrt{2}$, for $i \in [n/2]$. For $t \in \mathcal{T} \backslash \{1\}$, we split $\{Z_1, \dotsc, Z_t\}$ into two halves: $\{Z_1, Z_3 ,\dotsc, Z_{t-1}\}$ and $\{Z_2, Z_4, \dotsc, Z_t\}$. We further split the first set into $G_t$ groups of equal size, denoted as $\mathcal{Z}_{t,1,1}, \mathcal{Z}_{t,2,1}, \dotsc, \mathcal{Z}_{t,G_t,1}$, with the number of groups $G_t$ specified later in \eqref{Eq:robust_theorem_parameters_sparse}, and use $\overline{Z}_{t,g,1}$ to denote the sample mean of the $g$-th group. The set $\{Z_2, Z_4, \dotsc, Z_t\}$ is reserved for selecting the signal coordinates as we did in \Cref{subSec:subWeibull_sparse}. Consider the statistic $V_{t,g,a} \in \mathbb{R}^p$ with 
\begin{equation} \label{Eq:V_tga}
V_{t,g,a}(j) :=  \biggl( \overline{Z}^2_{t,g,1}(j) - \frac{2G_t}{t} \biggr) \mathbbm{1}_{ \{ |Y_{t,2}(j)| \geq a \} }, \quad j \in [p],
\end{equation}
where $Y_{t,2}(j)$ is defined in \eqref{eq:samplespliting} and $a$ is a selection threshold to be specified in \eqref{Eq:robust_theorem_parameters_sparse}. Our test statistic takes the same form as in the dense case that
\begin{equation} \label{Eq:A_ta_MoM}
A_{t,a}^{\mathrm{MoM}} := \frac{t}{2}\cdot \mathrm{median} \Biggl(  \sum_{j=1}^p V_{t,1,a}(j),  \sum_{j=1}^p V_{t,2,a}(j), \dotsc, \sum_{j=1}^p V_{t,G_t,a}(j)\Biggl).   
\end{equation}
For $t=1$, we cannot perform sample-splitting and therefore we deal with it separately by considering
\begin{equation}\label{eq:finitemoment_sparse_mom-1}
A_{1,a}^{\mathrm{MoM}} := A_{1,a} = \sum_{j=1}^p (Z_1^2(j)-1) \mathbbm{1}_{\{|Z_1(j)| \geq a\} }.
\end{equation}
Finally, the test is given by 
\begin{equation}\label{eq:finitemoment_sparse_mom}
    \phi_{\mathcal{P}, \mathrm{sparse}}^{\mathrm{MoM}}:= \mathbbm{1}_{\{\max_{t \in \mathcal{T}} A_{t,a}^{\mathrm{MoM}} / r_t > 1\}}.
\end{equation}

\begin{prop} \label{thm:finitemoment_upperbound_sparse}
Assume $\alpha \geq 4$. For any $\varepsilon \in (0,1)$, there exist $C_1, C_2, C_3> 0$ depending only on $\alpha$, $K$ and $\varepsilon$, such that $\phi_{\mathcal{P}, \mathrm{sparse}}^{\mathrm{MoM}}$ defined in~\eqref{eq:finitemoment_sparse_mom} with
\begin{equation} \label{Eq:robust_theorem_parameters_sparse}
\begin{aligned} 
 a = C_1\bigl((p/s)^{1/\alpha} + s^{-1/2} \log^{1/2}(\log(8n))\bigr), &\quad r_t =  C_2\bigl(s(p/s)^{2/\alpha}\mathbbm{1}_{\{t=1\}} +  \sqrt{s} G_t\mathbbm{1}_{\{t>1\}}\bigr),  \\
G_t = (t \wedge \Delta)/2 \quad &\mbox{and} \quad  \Delta = 2^{4 + \lceil\log_2 \log \log(8n)\rceil},
\end{aligned}
\end{equation}
satisfies that
\[
\mathcal{R}_{\mathcal{P}}(\rho,  \phi_{\mathcal{P}, \mathrm{sparse}}^{\mathrm{MoM}}) \leq \varepsilon,
\]
as long as $\rho^2 \geq C_3 v_{\mathcal{P}, \mathrm{sparse}}^{\mathrm{U, MoM}}$, where 
\[
v_{\mathcal{P}, \mathrm{sparse}}^{\mathrm{U, MoM}} := s\bigl((p/s)^{2/\alpha} + \log\log(8n)\bigr).
\]
\end{prop}
We observe that the rate $v_{\mathcal{P}, \mathrm{sparse}}^{\mathrm{U, MoM}}$ is dominated by $v_{\mathcal{P}, \mathrm{sparse}}^{\mathrm{U}} = s(p/s)^{2/\alpha} + \log\log(8n)$ in \Cref{thm:finitemoment_upperbound_sparse_improve}, meaning that using solely this MoM-type test $\phi_{\mathcal{P}, \mathrm{sparse}}^{\mathrm{MoM}}$ is not statistically optimal. However, the computational complexity of each step in constructing this test (defined in \eqref{Eq:V_tga}, \eqref{Eq:A_ta_MoM}, \eqref{eq:finitemoment_sparse_mom-1} and \eqref{eq:finitemoment_sparse_mom}) is polynomial in $n$ and $p$, making it a feasible polynomial-time testing procedure.

It is also worth mentioning that in the hard-thresholding step, we simply use the non-robust quantity $Y_{t,2}$ to estimate the signal of each coordinate instead of its robust counterparts to avoid further complication of the procedure. If, however, we know that the change point is sufficiently far away from the endpoints, employing a robust procedure for coordinate selection can significantly improve the testing rate. This additional assumption requires modifying the alternative space in~\eqref{Eq:H_0_H_1}, leading to a different testing problem. A detailed discussion of this scenario is provided in Section~\ref{Sec:awayfrombonudary}.

\subsubsection{An optimal polynomial-time test in the sparse regime}
\label{subsubsec:polynomial-time-alg}

By comparing $v_{\mathcal{P}, \mathrm{sparse}}^{\mathrm{U}}$ and $v_{\mathcal{P}, \mathrm{sparse}}^{\mathrm{U}, \mathrm{MoM}}$, as established in \Cref{thm:finitemoment_upperbound_sparse_improve} and \Cref{thm:finitemoment_upperbound_sparse}, we observe that the improvement offered by $\phi^{\mathrm{RSM}}_{\mathcal{P},\mathrm{sparse}}$ over $\phi_{\mathcal{P}, \mathrm{sparse}}^{\mathrm{MoM}}$ occurs only when $(p/s)^{2/\alpha} < \log\log(8n)$, in which case $\phi_{\mathcal{P}, \mathrm{sparse}}^{\mathrm{MoM}}$ attains the suboptimal rate $s\log\log(8n)$. Given the range of $s$ in the sparse regime, we deduce that the computationally expensive $\phi^{\mathrm{RSM}}_{\mathcal{P},\mathrm{sparse}}$ is only necessary when $p < \log^{\alpha-2}(\log(8n))$, whereas $\phi_{\mathcal{P}, \mathrm{sparse}}^{\mathrm{MoM}}$ can be used otherwise. We define the following combined testing procedure:
\begin{equation} \label{eq:finite_sparse_combined_test}  
\phi_{\mathcal{P},\mathrm{sparse}} :=  
\begin{cases}
    \phi_{\mathcal{P}, \mathrm{sparse}}^{\mathrm{RSM}}, & \text{if } p < \log^{\alpha-2}(\log(8n)), \\  
    \phi_{\mathcal{P}, \mathrm{sparse}}^{\mathrm{MoM}}, & \text{otherwise}.  
\end{cases}  
\end{equation}  
The corollary below confirms that this combined test runs in polynomial time in both $n$ and $p$ while achieving the optimal rate $v_{\mathcal{P}, \mathrm{sparse}}^{\mathrm{U}}$ in the sparse regime.
\begin{cor} \label{thm:finitemoment_combinerate}
Assume $\alpha \geq 4$ and $s < s_{\mathcal{P}}^*$. Consider the test $\phi_{\mathcal{P},\mathrm{sparse}}$ defined in~\eqref{eq:finite_sparse_combined_test}, with its two components $\phi_{\mathcal{P}, \mathrm{sparse}}^{\mathrm{RSM}}$ and $\phi_{\mathcal{P}, \mathrm{sparse}}^{\mathrm{MoM}}$ described in Sections~\ref{subSec:robust_sparse} and~\ref{subsubsec:mom-sparse}, respectively. For any $\varepsilon \in (0,1)$, there exists a constant $C > 0$ depending only on $\alpha$, $K$, and $\varepsilon$, such that $\phi_{\mathcal{P},\mathrm{sparse}}$, with the parameters of its two components chosen according to  \eqref{Eq:robust_theorem_parameters_sparse_rsm} and \eqref{Eq:robust_theorem_parameters_sparse}, satisfies that
\[
\mathcal{R}_{\mathcal{P}}(\rho, \phi_{\mathcal{P},\mathrm{sparse}}) \leq \varepsilon,  
\]
as long as $\rho^2 \geq C_1 v_{\mathcal{P}, \mathrm{sparse}}^{\mathrm{U}}$, with $v_{\mathcal{P}, \mathrm{sparse}}^{\mathrm{U}}$ defined in \eqref{eq-sparse-upper-poly}.
Moreover, the computational complexity of $\phi_{\mathcal{P},\mathrm{sparse}}$ is polynomial in both $n$ and $p$.  
\end{cor}

We conclude this section by summarising in \Cref{table:test_sumary} the main features of each test constructed to achieve the upper bounds of the minimax testing rates in Tables~\ref{table:testing_rate} and~\ref{table:testing_rate_finite} for $\mathcal{G}_{\alpha, K}$ and $\mathcal{P}_{\alpha, K}$ and under both dense and sparse signals.

\begin{table}[ht!]
\centering
\def\arraystretch{1.8}
\begin{tabular}{|c|c|c|}
\hline
& Dense  & Sparse \\ \hline
$\mathcal{G}_{\alpha, K}^{\otimes}$ & CUSUM-type, $\ell_2$ aggregation      &    CUSUM-type, thresholding, $\ell_2$ aggregation    \\   \hline
$\mathcal{P}_{\alpha, K}^{\otimes}$ &   MoM-type, $\ell_2$ aggregation   & \makecell{Big $p$: MoM-type, thresholding, $\ell_2$ aggregation \\ Small $p$: Robust sparse mean estimator based test} \\ \hline
\end{tabular}
\caption{A summary of main features of each test constructed to achieve the upper bounds of the minimax testing rates in Tables~\ref{table:testing_rate} and~\ref{table:testing_rate_finite}.}
\label{table:test_sumary}
\end{table}

When the noise has exponentially decaying tails, CUSUM-type statistics with $\ell_2$ aggregation are sufficient to achieve near-optimal testing. In contrast, for polynomially decaying tails, robust methods such as median-of-means or robust sparse mean estimation are required to construct effective tests. For both types of heavy-tailed noise, when testing sparse signals, a thresholding step is applied to identify the signal coordinates before aggregation. Finally, under finite-moment noise distributions, the two robust methods, chosen according to whether $p$ is large or small, combine to yield an optimal and computationally efficient testing procedure.

\section{Adaptation to sparsity} \label{sec:adaptsparsity}
In Sections~\ref{Sec:hidimtest} and \ref{Sec:robusttest}, we have studied the change point testing problem under two types of heavy-tail assumptions on the error distributions: (1) exponentially-decaying/sub-Weibull tails and (2) polynomially-decaying/finite~$\alpha$-th moment assumption with $\alpha \geq 2$. The corresponding upper bound rates,~e.g.~$v_{\mathcal{G}, \mathrm{sparse}}^{\mathrm{U}}$ and~$v_{\mathcal{P}, \mathrm{sparse}}^{\mathrm{U}}$, are currently achieved by testing procedures that take the sparsity level $s$ as an input. In this section, we study the adaptation of these procedures to unknown sparsity levels.

First off, in the very heavy-tailed setting, i.e.\ each entry of $E$ has only finite $\alpha$-th moments for $\alpha \in [2, 4]$, there is no sparse regime; see the discussion in \Cref{subsubsec:minimax-polynomial}. The test~$\phi_{\mathcal{P},\mathrm{dense}}$ defined in~\eqref{eq:finitemoment_dense_test} with its parameters specified in Theorem~\ref{thm:finitemoment_upperbound_dense} does not require the knowledge of the sparsity, and, therefore, the corresponding rate $v_{\mathcal{P}, \mathrm{dense}}^{\mathrm{U}} = p^{2/\alpha} \log \log(8n)$ can already be achieved by an adaptive procedure.

We focus on the case where $P_e \in \mathcal{P}_{\alpha, K}^{\otimes}$ for $\alpha > 4$. Recall from \Cref{thm:finitemoment_upperbound_dense} and \Cref{thm:finitemoment_combinerate} that the tests $\phi_{\mathcal{P},\mathrm{dense}}$ and $\phi_{\mathcal{P},\mathrm{sparse}}$ achieve the rates $v_{\mathcal{P}, \mathrm{dense}}^{\mathrm{U}}$ and $v_{\mathcal{P}, \mathrm{sparse}}^{\mathrm{U}}$, respectively, when the sparsity level is known. To address the scenario where sparsity is unknown, we propose the following adaptive testing procedure that combines these two tests:
\begin{equation} \label{Eq:adaptive_test_main}
\phi_{\mathcal{P}, \mathrm{adaptive}} := \phi_{\mathcal{P},\mathrm{dense}} \vee \max_{s \in \mathcal{K}} {\phi}_{\mathcal{P}, \mathrm{sparse}, s},
\end{equation}
where the dependence of $\phi_{\mathcal{P},\mathrm{sparse}}$ on $s$ is made explicit by writing it as ${\phi}_{\mathcal{P}, \mathrm{sparse}, s}$, and the set $\mathcal{K} := \{1, 2, 4, \dotsc, 2^{\lceil \log_2(p) \rceil - 1}\}$ is a dyadic grid. The details of this test along with its parameter choices are provided in \Cref{sec:adapt-test}.

\begin{thm} \label{thm:adaptive_upperbound}
Assume $\alpha \geq 4$. For any $\varepsilon \in (0,1)$, there exists a test $\phi_{\mathcal{P}, \mathrm{adaptive}}$ of the form \eqref{Eq:adaptive_test_main} that satisfies 
\[
\mathcal{R}_{\mathcal{P}}(\rho,  \phi_{\mathcal{P},\mathrm{adaptive}} ) \leq \varepsilon,
\]
as long as $\rho^2 \geq C\bigl(v_{\mathcal{P}, \mathrm{dense}}^{\mathrm{U}} \wedge v_{\mathcal{P}, \mathrm{sparse}}^{\mathrm{U}}\bigr)$, with $C>0$ being a constant depending only on $\alpha$, $K$ and $\varepsilon$.
\end{thm}
\Cref{thm:adaptive_upperbound} establishes that the adaptive test $\phi_{\mathcal{P},\mathrm{adaptive}}$ achieves the same rate as \eqref{Eq:finitemoment_minimum_upperbound} without requiring the sparsity parameter as an input. Therefore, the discussion on the optimality in \Cref{subsubsec:minimax-polynomial} also holds. When the errors have exponentially-decaying tails instead, a similar adaptive testing procedure $\phi_{\mathcal{G},\mathrm{adaptive}}$ can be constructed using ~\eqref{Eq:adaptive_test_main} with $\phi_{\mathcal{G},\mathrm{dense}}$ and $\phi_{\mathcal{G},\mathrm{sparse}}$ instead and achieve the same rate as in~\eqref{Eq:weibull_minimum_upperbound}. For brevity, we omit further details here.

\section{Illustrative simulations}
\label{Sec:simulations}
In this section, we illustrate our main results, as presented in \Cref{Sec:main result} via simulation studies. Specifically, we numerically verify the shrinkage of sparse regimes as the tail becomes heavier in both exponentially- and polynomially-decaying tail cases.

In our simulations, we set $p=100$ and $n=300$. The data matrices $X \in \mathbb{R}^{100\times 300}$ are generated according to~\eqref{eq:model}, where the entries of the noise matrix $E$ are i.i.d.~following some distribution $F$. Let $t_k$ denotes the $t$-distribution with $k$ degrees of freedom. We consider the following four choices for the distribution $F$: 
\begin{itemize}
    \item GG(2): standard normal distribution, which belongs to the class $\mathcal{G}_{2,K}$ for some $K$;
    \item GG(0.5): Generalised Gaussian distribution with mean $0$, shape parameter $0.5$, and scale parameter chosen so that the variance is $1$; this belongs to the class $\mathcal{G}_{0.5,K}$ for some $K$; 
    \item Nt(8): $\sqrt{3/4}\cdot t_8$, which has variance $1$ and belongs to the class $\mathcal{P}_{8-\epsilon, K}$ for any $\epsilon > 0 $\footnote{A $t$ distribution with $\nu$ degrees of freedom has $(\nu-\epsilon)$-th moment finite for any $\epsilon >0$.};
    \item Nt(3): $\sqrt{1/3} \cdot t_3$, which has variance $1$ and belongs to the class $\mathcal{P}_{3-\epsilon, K}$ for any $\epsilon > 0 $.
\end{itemize}

We adopt the adaptive tests $\phi_{\mathcal{P}, \mathrm{adaptive}}$ defined in \eqref{Eq:adaptive_test_main}, and $\phi_{\mathcal{G}, \mathrm{adaptive}}$, the corresponding version for the class $\mathcal{G}_{\alpha,K}$. To calibrate the constants in our tests, we set $\theta$ to be the matrix with all entries being zero. In both the exponentially- and polynomially-decaying tail cases, the adaptive test consists of a dense test and a collection of sparse tests. For the dense test $\phi_{\mathcal{G}, \mathrm{dense}}$  (or $\phi_{\mathcal{P},\mathrm{dense}}$), we calibrate the constant in the detection threshold $r$ (or $r_t$) so that the empirical Type I error is below $0.025$. For the sparse tests, we specify two constants that appear in $a$ and $r$ (or $r_t$), used in the thresholding and detection steps, respectively. We select these constants so that the empirical Type I error is below $0.025$. Taken together, this ensures that the overall Type I error of the adaptive test is controlled at $0.05$.

The power of our adaptive test is evaluated empirically on a grid of $(s, \mathtt{Signal})$ values, 
where $s \in \{1, \ldots, 50\}$ and $\mathtt{Signal}$ varies over a fine grid within a suitable interval, 
which may differ across the tail cases. For each fixed $(s, \mathtt{Signal})$ pair, 
the mean matrix $\theta = (\theta_1, \ldots, \theta_{300}) \in \mathbb{R}^{100 \times 300}$ 
is generated as follows. The change point location $t_0$ is chosen uniformly at random from 
$\{1, \ldots, 150\}$, and we set 
\[
\theta_t = 0 \in \mathbb{R}^{100}, \quad 1 \leq t \leq t_0,
\] 
and 
\[
\theta_t = \sqrt{\frac{n}{t_0(n-t_0)}}\, \frac{\mathtt{Signal}}{\sqrt{s}} \, 
(\underbrace{1,\ldots, 1}_{s}, \underbrace{0,\ldots, 0}_{p-s})^\top, \quad t_0+1 \leq t \leq n.
\] 
The power for each pair $(s, \mathtt{Signal})$ is calculated by repeating the process 500 times of generating~$\theta$,~$E$, and thus $X$, and computing $\phi_{\mathcal{G}, \mathrm{adaptive}}$ 
(or $\phi_{\mathcal{P}, \mathrm{adaptive}}$) with the constants calibrated as described above. We plot two heatmaps showing the power of $\phi_{\mathcal{G}, \mathrm{adaptive}}$, 
with \texttt{Signal} on the horizontal axis and sparsity $s$ on the vertical axis, 
for GG(2) and GG(0.5), respectively, in Figure~\ref{fig:subweibull}. 
Similarly, heatmaps of the power of $\phi_{\mathcal{P}, \mathrm{adaptive}}$ 
for Nt(8) and Nt(3), respectively, are shown in Figure~\ref{fig:t-dist}.

\begin{figure}[!ht]
    \centering
    \includegraphics[width=0.45\linewidth]{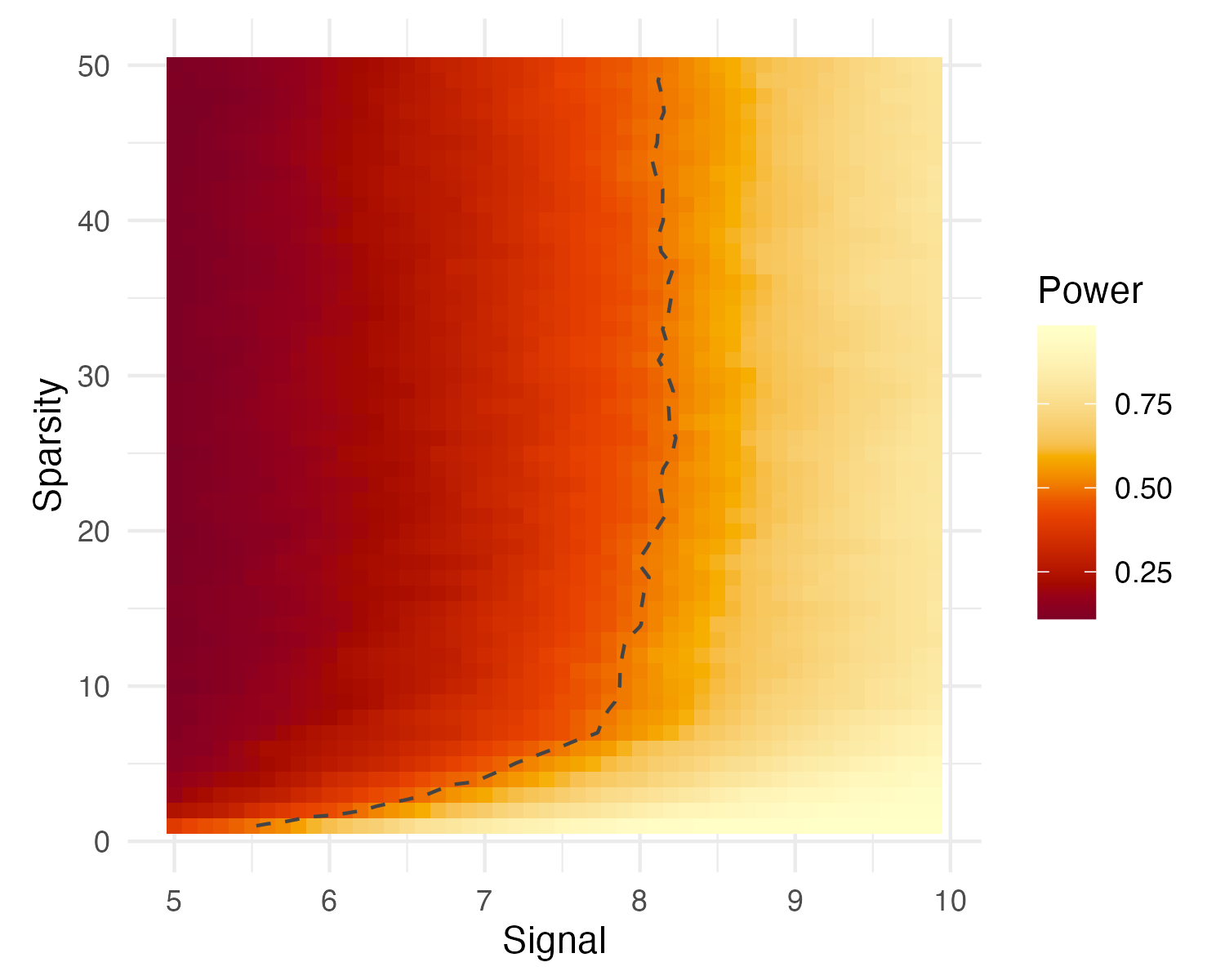}
    \includegraphics[width=0.45\linewidth]{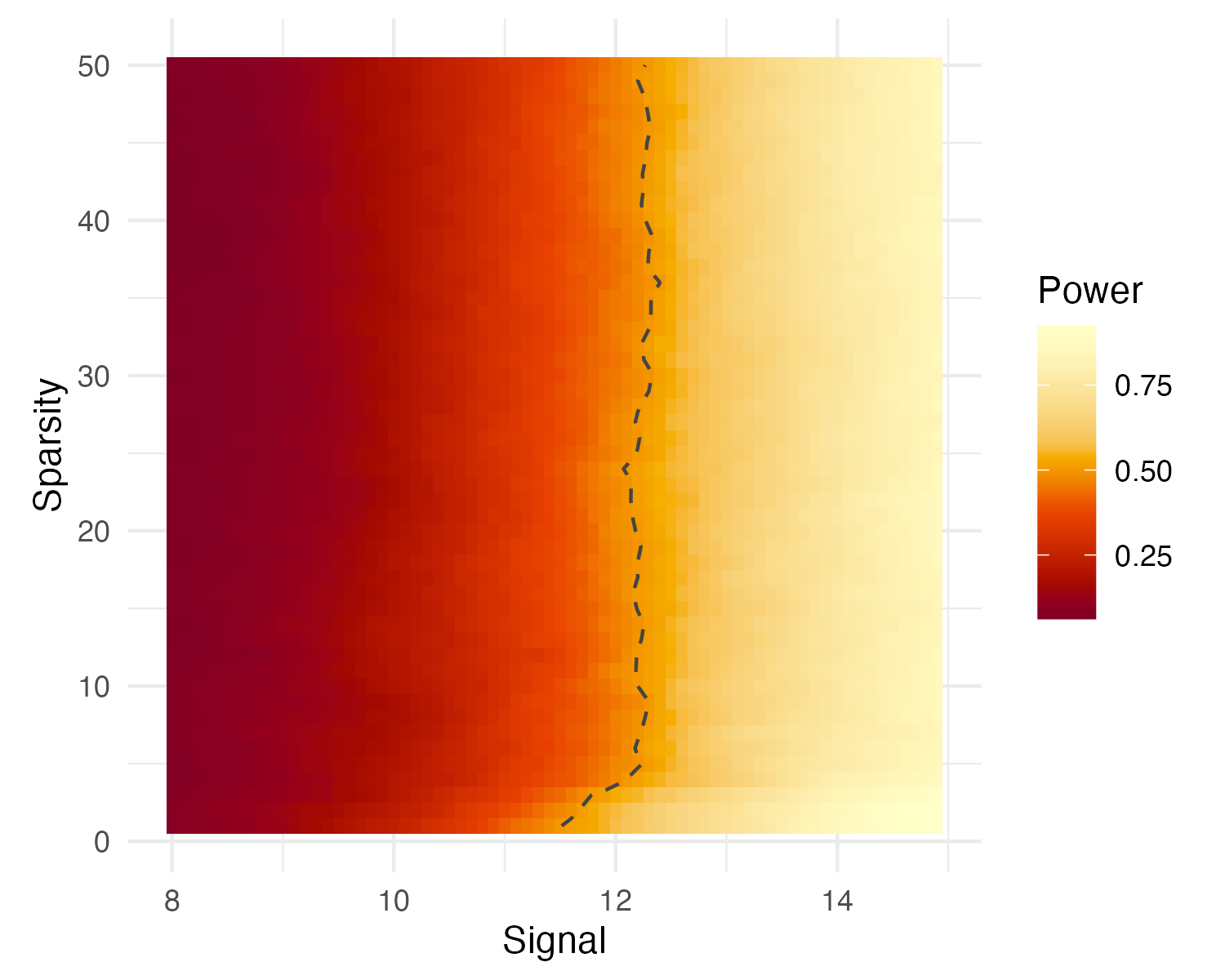}
    \caption{Simulation results for noise GG(2) (left panel) and GG(0.5) (right panel). The dashed line represents the 0.5 contour of the power.}
    \label{fig:subweibull}
\end{figure}

\begin{figure}[!ht]
    \centering
    \includegraphics[width=0.45\linewidth]{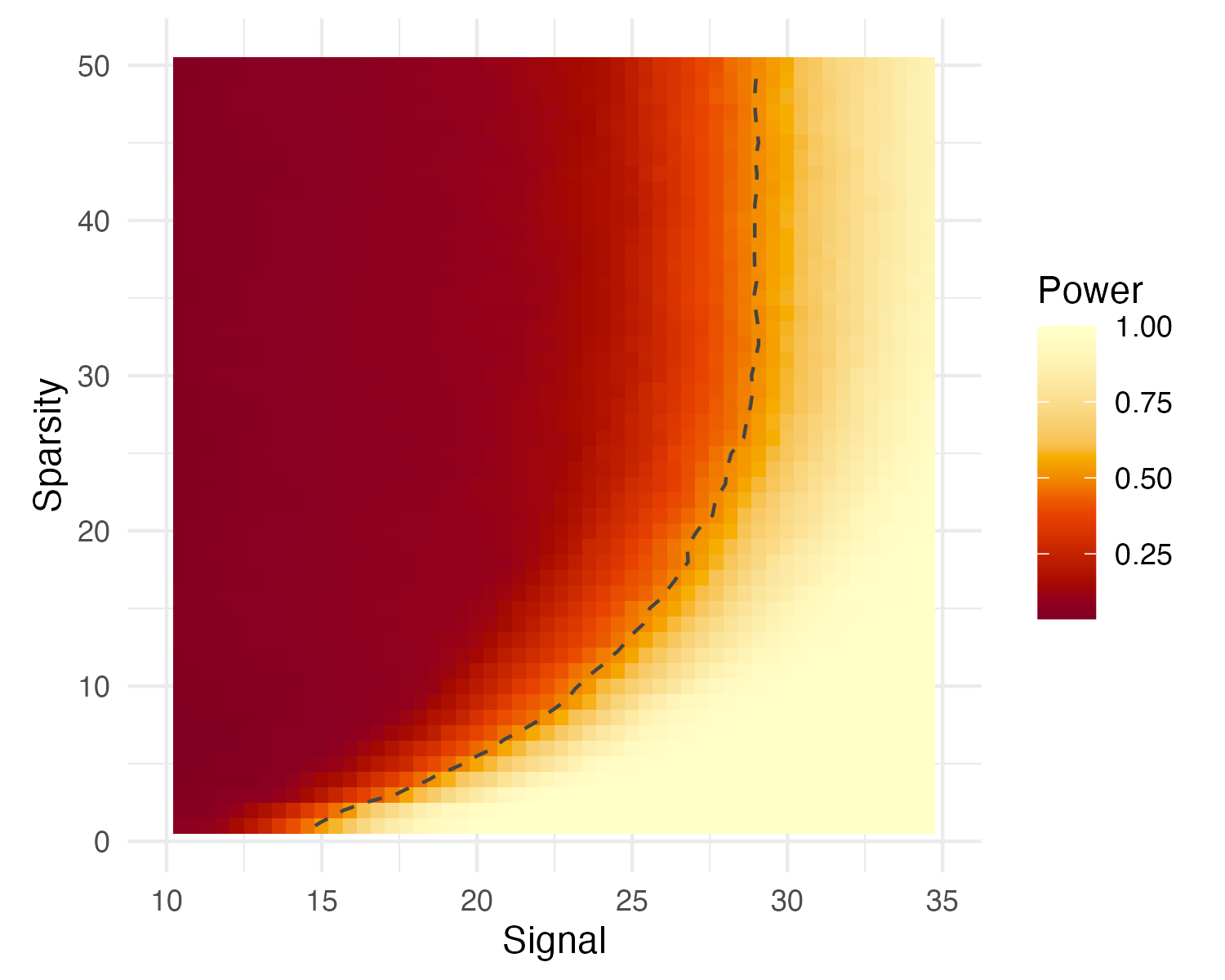}
    \includegraphics[width=0.45\linewidth]{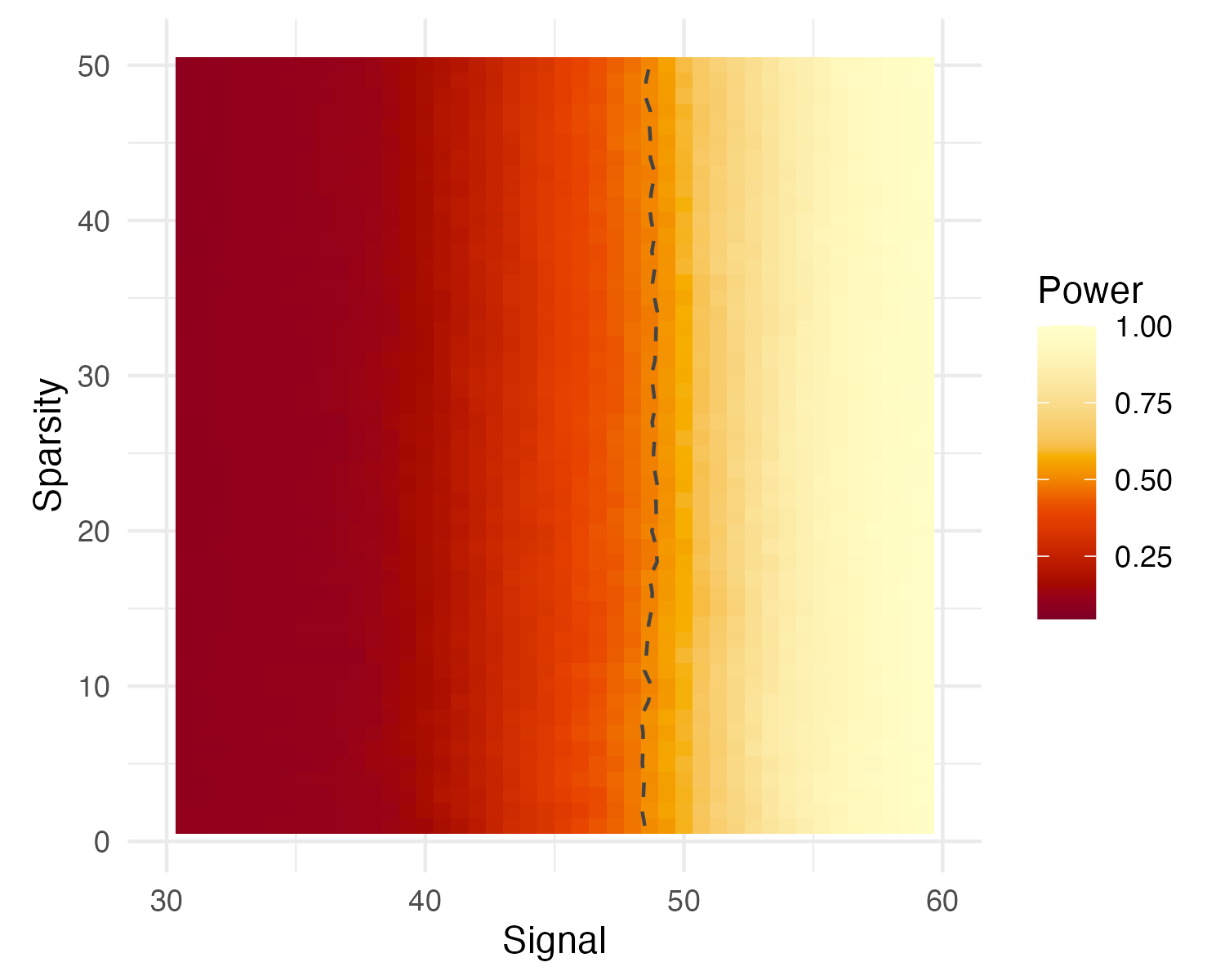}
    \caption{Simulation results for noise Nt(8) (left panel) and Nt(3) (right panel). The dashed line represents the 0.5 contour of the power.}
    \label{fig:t-dist}
\end{figure}

In both Figures~\ref{fig:subweibull} and~\ref{fig:t-dist}, we observe that as the tails become heavier (moving from the left panel to the right panel), the power generally decreases, as indicated by the need for higher signal levels to achieve the same power. Moreover, the power for detecting sparse changes is much more affected than for dense changes as the tails become heavier. In \Cref{fig:subweibull}, when the noise is $\mathrm{GG}(2)$, there exists a range of sparsity levels where higher power can be achieved compared to dense changes at the same signal level. However, when the noise is $\mathrm{GG}(0.5)$, this advantage for sparse changes diminishes significantly. A similar pattern is seen in \Cref{fig:t-dist}. Moreover, as suggested by our theoretical results and now verified here, when the noise distribution has fewer than four finite moments, there is no sparsity level for which testing sparse changes is easier than testing dense changes.

\section{Extensions} \label{Sec:extensions}
In the previous sections, we characterised the minimax rates for testing a single mean change within independent high-dimensional data streams under two broad classes of heavy-tailed noise distributions. In this section, we consider several extensions, including (1) testing against multiple change points in \Cref{sec:multiple}, (2) accounting for temporal dependence among observations in \Cref{sec:temporal}, (3) addressing the case where noise matrix entries have fewer than two finite moments in \Cref{sec:<2moment}, and (4) {examining a modified testing problem where the potential change point is known to be some distance away from the endpoints in \Cref{Sec:awayfrombonudary}.} 

A summary of the problem setups and notation is provided in \Cref{Table:extension-summary}. Specifically, within each section, we propose testing procedures $\phi$ and control their worst case testing error, which takes the form
\[
\mathcal{R}_{\mathcal{Q}, \#}(\rho, \phi) :=  \sup_{P_e \in \mathcal{Q}}\sup_{\theta \in \mathrm{H}_0}\mathbb{E}_{\theta, P_e} (\phi) + \sup_{P_e \in \mathcal{Q}}\sup_{\theta \in \mathrm{H}_1}\mathbb{E}_{\theta, P_e} (1-\phi),
\]
where $\#$ denotes the problem-dependent quantities and $\rho$ is some measure of change strength defined clearly in each alternative parameter space. Lower bounds on the minimax testing rates will be provided where available by considering $\inf_{\phi}\mathcal{R}_{\mathcal{Q}, \#}(\rho, \phi)$ and, for brevity, we do not separately define the minimax testing rates for each problem. 

\begin{table}[ht!]
\centering
\def\arraystretch{1.8}
\begin{tabular}{|l|l|l|l|}
\hline
 & Alternative hypothesis $\mathrm{H}_1$ & Noise distribution $\mathcal{Q}$ & Notation  \\ \hline
 \Cref{sec:multiple} & $\theta \in \Theta_{\mathrm{multi}}(p,n,\rho)$ \eqref{Eq:H_0_H_2} & $\mathcal{P}_{\alpha, K}^{\otimes}$ & $\mathcal{R}_{\mathcal{P}, \mathrm{multi}}$ \\ \hline
 \Cref{sec:temporal}& $\theta \in \Theta(p,n,p,\rho)$ \eqref{Eq:H_0_H_1} & $\mathcal{P}_{\mathrm{Temp}}$ (\Cref{def-palphak_temp}) & $\mathcal{R}_{\mathcal{P}_{\mathrm{Temp}}}$ \\ \hline
 \Cref{sec:<2moment}& $\theta \in  \bigcup_{t_0 = 1}^{n-1}\Theta(p,n, \rho_{t_0})$ \eqref{eq:H1space_lessthan2moment} & $\mathcal{W}_{\alpha}^\otimes$ (\Cref{def:w-alpha-<2}) & $\mathcal{R}_{\mathcal{W}_{\alpha}^{\otimes}}$ \\ \hline
 \Cref{Sec:awayfrombonudary}& $\theta \in \Theta_{\mathrm{res}}(p,n,s,\rho)$ \eqref{eq:H1-restricted} & $\mathcal{P}_{\alpha, K}^{\otimes} \;\& \;\mathcal{G}_{\alpha, K}^{\otimes}$ & $\mathcal{R}_{\mathcal{Q}, {\mathrm{res}}}$ \\ \hline
\end{tabular}
\caption{Summary of the different settings in \Cref{Sec:extensions}.}
\label{Table:extension-summary}
\end{table}

\subsection{Testing against multiple change points}\label{sec:multiple}
We consider the following testing problem between no change point and at least one change point:
\begin{equation} \label{Eq:H_0_H_2}
\mathrm{H}_0: \theta \in \Theta_0(p,n) \quad \mathrm{vs.}\quad  \mathrm{H}_1: \theta \in \Theta_{\mathrm{multi}}(p,n,\rho) := \bigcup_{k \in \mathbb{Z}^+}\bigcup_{1\leq\tau_1<\ldots <\tau_k \leq n-1} \Theta^{(\tau_1,\ldots,\tau_k)}(p,n,\rho),
\end{equation}
where $\Theta_0(p,n)$ is the same as in~\eqref{Eq:null_space} and
    \begin{align*}
    \Theta^{(\tau_1, \ldots, \tau_k)}(p, n, \rho) := \biggl\{ &\theta : \theta_{\tau_{i-1} + 1} = \ldots = \theta_{\tau_i} = \mu_i \ \text{for all} \ i \in [k+1], \ \text{for some} \ \mu_1, \ldots, \mu_{k+1} \in \mathbb{R}^p, \\
    &\text{where there exists $i \in [k]$ s.t.~$\Delta_i \geq 4 \log(n)$ and $\kappa_i^2 \Delta_i \geq \rho^2$, with} \\ 
    &\Delta_i := \min\{\tau_i - \tau_{i-1}, \tau_{i+1} - \tau_i\}, \ \kappa_i := \|\mu_{i+1} - \mu_i\|_2 \ \text{for} \ i \in [k],\\
    &\tau_0 = 0 \ \text{and} \ \tau_{k+1} = n \mbox{ by convention} \biggr\}.
\end{align*}
The worst case testing error for a measurable test function $\phi$ is 
\[
\mathcal{R}_{\mathcal{Q}, \mathrm{multi}}(\rho, \phi) :=  \sup_{P_e \in \mathcal{Q}}\sup_{\theta \in \Theta_0(p,n)}\mathbb{E}_{\theta, P_e} (\phi) + \sup_{P_e \in \mathcal{Q}}\sup_{\theta \in \Theta_{\mathrm{multi}}(p,n,\rho)}\mathbb{E}_{\theta, P_e} (1-\phi).
\]
Throughout this subsection, we consider the class of distributions $\mathcal{Q} = \mathcal{P}_{\alpha, K}^\otimes$ with $\alpha \geq 4$ and $K < \infty$.
\begin{thm}\label{thm:h0_vs_h2} 
Assume $n \geq 50$ and $\alpha \geq 4$. For any $\varepsilon \in (0,1)$, there exists a test $\phi_{\mathcal{P},\mathrm{multi}}^{\mathrm{MoM}}$ that satisfies
\[
\mathcal{R}_{\mathcal{P}, \mathrm{multi}}(\rho, \phi_{\mathcal{P},\mathrm{multi}}^{\mathrm{MoM}}) \leq \varepsilon,
\]
as long as 
\[
\rho^2 \geq C \sqrt{p} \log(n),
\]
where $C>0$ is a constant depending only on $\alpha$, $K$ and $\varepsilon$.
\end{thm} 
The details of the test $\phi_{\mathcal{P},\mathrm{multi}}^{\mathrm{MoM}}$ is described in \Cref{sec:multiple-app}. At a high level, 
it exhibits a multi-scale nature by performing a collection of `local tests' at all possible change locations with a dyadic grid of scales. Such multi-scale statistics are commonly used for detecting multiple change points \citep[e.g.][]{frick2014smuce,pilliat2020optimalhighdim}. The robustness property is achieved by integrating the median-of-means methodology into the local tests.  Another ideal feature of our test is that it is adaptive to the unknown number of change points~$k$ under the alternative hypothesis, i.e.\ it does not take $k$ as an input.

In contrast to the single change point setting where we allow the potential change point to be arbitrarily close to the endpoints, here, we impose a mild minimum spacing condition $\Delta_i \geq 4 \log(n)$ for some change point. However, this condition is not required for testing under Gaussian or sub-Gaussian assumptions \citep{pilliat2020optimalhighdim}. In \Cref{sec:app_multiple}, we provide a heuristic example illustrating why the heavy-tailed nature of the data necessitates a minimum spacing condition to achieve a rate that has logarithmic dependence on~$n$. 

The following proposition shows that the test $\phi_{\mathcal{P},\mathrm{multi}}^{\mathrm{MoM}}$ is minimax rate-optimal up to a factor of $\sqrt{\log(n)}$ when $\alpha \geq 4$.
\begin{prop}\label{prop:h0_h2}
Let $n \geq 72$, $\alpha \geq 4$ and $K \geq \sqrt{\alpha+1}$. Then it holds that $\inf_{\phi}\mathcal{R}_{\mathcal{P}, \mathrm{multi}}(\rho, \phi) \geq 1/2$ whenever 
    \[
    \rho^2 \leq c\Big(\sqrt{p\log(n)}+\log(n)\Big),
    \]
    for some absolute constant $c > 0$.
\end{prop}

By comparing \Cref{thm:h0_vs_h2} and \Cref{prop:h0_h2} for testing against multiple change points, with Theorems~\ref{thm:finitemoment_upperbound_dense} and~\ref{thm:finitemoment-lowerbound-minimum} for the single change point problem in Section~\ref{Sec:robusttest}, we observe that the main difference is the presence of $\log(n)$ instead of $\log\log(n)$ in the rates. This distinction has been observed in the change point literature for univariate data under sub-Gaussian noise distributions \citep{verzelen2020optimal1d}. Our results reaffirm this phenomenon even in the context of high dimensional data with heavy-tailed distributions.

\subsection{Temporal dependence}\label{sec:temporal}
We now consider temporal dependence in heavy-tailed observations. To isolate its impact, we focus on the case where $s = p$ (eliminating the effect of sparsity) and assume that the error matrix~$E$ has independent rows. 

For simplicity, assume $n \geq 4$ to be an even number. The temporal dependence assumption to be introduced involves the interlaced $\alpha$-mixing coefficient of our noise sequence $\{E_i\}_{i \in [n]}$:
\begin{equation} \label{Eq:alpha*_def}
\alpha^*(i) := \sup_{\substack{S, T \subseteq [n]: \\ \min_{s \in S, t \in T} |s-t| \geq i}} \,\,  \sup_{\substack{A \in \sigma(E_j: j \in S), \\ B \in \sigma(E_j: j \in T)}} \bigl|\mathbb{P}(A \cap B) - \mathbb{P}(A)\mathbb{P}(B)\bigr|.
\end{equation}

Throughout this subsection, we consider the class of distributions $\mathcal{Q} = \mathcal{P}_{\mathrm{Temp}}$, defined as follows:
\begin{defn}[$\mathcal{P}_{\mathrm{Temp}}$ class of distributions]\label{def-palphak_temp}
    For $K > 0$, $\alpha \geq 2$ and $c_1, c_2 > 0$, let $\mathcal{P}_{\mathrm{Temp}}$  denote the class of all distributions of the error matrix $E$ satisfying the following conditions:
    \begin{enumerate}
    \item the matrix $E$ has independent rows;
     \item the marginal distribution of each entry in $E$ belongs to $\mathcal{P}_{\alpha, K}$; and
     \item the interlaced $\alpha$-mixing coefficient satisfies that
     \begin{equation} \label{Eq:alpha*_mixing}
         \alpha^*(i) \leq c_1 e^{-c_2 i},  \quad \text{for all } i \in [n-1].
     \end{equation}
\end{enumerate}
\end{defn}

\begin{thm}\label{thm:temporal}
Assume $\alpha > 4$. For any $\varepsilon \in (0,1)$, there exists a test $\phi_{\mathcal{P}_\mathrm{Temp}}$ that satisfies
\[
\mathcal{R}_{\mathcal{P}_{\mathrm{Temp}}}(\rho, \phi_{\mathcal{P}_\mathrm{Temp}}) \leq \varepsilon,
\]
as long as 
\[
\rho^2 \geq C\, p^{1/2} \{\log \log (8n)\} \{\log \log \log (64n)\}^2,
\]
where $C$ is a constant depending only on~$\alpha$,~$K$,~$c_1$,~$c_2$ and $\varepsilon$.
\end{thm}
 The test is essentially constructed in the same way as in \Cref{subSec:robust_dense} and we leave the detailed description to \Cref{sec:temporal-app}. We also note that a similar result can be obtained for the case of $2 < \alpha \leq 4$; see \Cref{rmk:temporal} in \Cref{sec:proof-temporal-app}.
By comparing \Cref{thm:temporal} with \Cref{thm:finitemoment_upperbound_dense}, we conclude that under certain weak dependence settings, described by the class $\mathcal{P}_{\mathrm{Temp}}$, we can still achieve nearly the same rate as in the independence setting, up to a factor of squared triple logarithm. 

To further interpret the main temporal dependence condition \eqref{Eq:alpha*_mixing} required for $\mathcal{P}_{\mathrm{Temp}}$, we note that \eqref{Eq:alpha*_mixing} on the interlaced $\alpha$-mixing coefficient \eqref{Eq:alpha*_def} of the noise sequence is stricter than imposing the same condition on the (usual) $\alpha$-mixing coefficient. The reason we choose this interlaced version of the coefficient is that our technique involves pairing observations (first with last, second with second last, etc.) to form $Z_i$'s before grouping and averaging. 

A sufficient condition for \eqref{Eq:alpha*_mixing} is shown in \citet[][Theorem~5.1(b) and Equation~(1.12)]{Bradley2005}. If each component series $\{E_i(j)\}_{i \in [n]}$ of the noise sequence satisfies
\begin{equation}
\label{Eq:rho*_mixing}
\sup_{\substack{S, T \subseteq [n]: \\ \min_{s\in S, t\in T}|s-t| \geq i}} \,\, \sup_{\substack{f \in L^2(\sigma(E_k(j): k \in S)) \\ g \in L^2(\sigma(E_k(j): k \in T))}} \mathrm{Corr}(f, g)
\leq c_1e^{-c_2i},
\end{equation}
then the $\alpha^*$-mixing rate condition \eqref{Eq:alpha*_mixing} is satisfied with constants $c_1$ and $c_2$. This fact implies that moving average processes essentially belong to the class $\mathcal{P}_{\mathrm{Temp}}$. Specifically, consider a data-generating mechanism in which $\{E_i(j)\}_{i \in [n]}$, the $j$-th component series of the noise, is a moving average process of order $q_j$, $j \in [p]$. Assuming that there exists a constant $q_{\max}$ such that $\max_{j \in [p]} q_j \leq  q_{\max}$, each univariate component series satisfies \eqref{Eq:rho*_mixing} with $c_1 = e^{q_{\max}}$ and $c_2 = 1$. We present a more detailed analysis of an MA(1) example in \Cref{sec:ma1}. For further discussion and examples satisfying \eqref{Eq:alpha*_mixing} or \eqref{Eq:rho*_mixing}, see \citet{Bradley1997}.

\subsection{Fewer than two finite moments}
\label{sec:<2moment}

In Section~\ref{Sec:robusttest}, we studied the minimax testing rate when $P_e \in \mathcal{P}_{\alpha, K}^\otimes$, for some constant $\alpha \geq 2$ and $K < \infty$. Now, we tackle the even more challenging case when each entry of the noise matrix does not necessarily have a finite variance. Given our main findings that when $\alpha \in [2,4)$, there is no sparse regime (c.f.\ Sections~\ref{Sec:main result} and \ref{subsec:discuss_gap_robust}), we do not need to consider sparse changes in this even more challenging setting. We first introduce the class of distributions $\mathcal{W}_{\alpha}$ that we work with.

\begin{defn}\label{def:w-alpha-<2} For $1 \leq \alpha \leq 2$, let $\mathcal{W}_{\alpha}$ denote the class of distributions on $\mathbb{R}^p$ such that for any $P \in \mathcal{W}_{\alpha}$ and random variable $W \sim P$, it holds that 
    \begin{equation}\label{eq:weakmoment}
            \mathbb{E}W = 0 \quad \mbox{and} \quad \mathbb{E}|\langle W, v\rangle|^{\alpha} \leq 1, \quad \forall v \in \mathbb{R}^p, \,\|v\|_2 = 1. 
    \end{equation}
\end{defn}
Let $\mathcal{W}_{\alpha}^\otimes$ denote the class of joint distributions of all entries in the error matrix $E$, where each column of $E$ independently follows a distribution on $\mathbb{R}^p$ that belongs to the class $\mathcal{W}_\alpha$. It follows by Jensen's inequality that $\mathcal{P}_{2, 1}^\otimes \subseteq \mathcal{W}_{\alpha}^\otimes \subseteq \mathcal{W}_{\alpha'}^\otimes$ for all $1 \leq \alpha' \leq \alpha \leq 2$.

We now specify the alternative parameter space.
For a given sequence of $\{\rho_{t_0}\}_{t_0 \in[n-1]}$, let
\begin{equation*}
 \mathrm{H}_1: \theta \in  \bigcup_{t_0 = 1}^{n-1}\Theta(p,n, \rho_{t_0}),
\end{equation*}
where 
\begin{align}
    \Theta(p,n, \rho_{t_0}):=\biggl\{&\theta: \theta_t = \mu_1\ \text{for}\ t = 1, \ldots, t_0, \ \theta_t = \mu_2 \ \text{for}\ t = t_0+1, \ldots, n, \nonumber \\
    &\text{for some}\ \mu_1, \mu_2 \in \mathbb{R}^p \ \text{s.t.}\ \|\mu_1-\mu_2\|_2^2 \geq \rho_{t_0}^2\biggr\}\label{eq:H1space_lessthan2moment}.
\end{align}
Letting $\rho_{t_0} = \rho\sqrt{t_0^{-1}(n-t_0)^{-1}n}$, we see that the alternative parameter space \eqref{eq:H1space_lessthan2moment} is a superset of $\Theta^{(t_0)}(p,n,p,\rho)$ defined in \eqref{Eq:alternative_space_original}.
When the noise variances are finite, the factor $t_0^{-1}(n-t_0)^{-1}n$ is the variance of the natural test statistic $t_0^{-1}\sum_{i=1}^{t_0}X_i - (n-t_0)^{-1}\sum_{i=t_0+1}^n X_i$, when $p=1$. Therefore, the normalising factor $t_0(n-t_0)n^{-1}$ in \eqref{Eq:alternative_space_original} ensures a variance-stable signal strength parameter $\rho$ across different change locations. In the absence of finite variance assumptions, as considered in this section, we opt out for such a specific normalising factor and use the general form $\rho_{t_0}$.

 For a given sequence $\{\rho_{t_0}\}_{t_0 \in[n-1]}$, 
 the worst case testing error for a given test $\phi$ is
\[
\mathcal{R}_{\mathcal{W}_{\alpha}^{\otimes}}(\{\rho_{t_0}\}_{t_0 \in[n-1]}, \phi) := \sup_{P_e \in \mathcal{W}_{\alpha}^{\otimes}}\sup_{\theta \in \Theta_0(p,n)}\mathbb{E}_{\theta, P_e} (\phi) + \sup_{P_e \in \mathcal{W}_{\alpha}^{\otimes}}\sup_{\theta \in  \bigcup_{t_0 = 1}^{n-1}\Theta(p,n, \rho_{t_0})}\mathbb{E}_{\theta, P_e} (1-\phi).
\]
To establish an upper bound on the minimax testing rates, we develop a test $\phi_{\mathcal{W}_{\alpha}}^{\mathrm{RM}}$, which has two components that separately target at whether the potential change point is sufficiently away from the boundary or not. When the potential change is away from the boundary, our test utilises on a robust mean estimator $\hat{\mu}^{\mathrm{RM}}$ from \citet[][c.f. Algorithm 1-7]{cherapanamjeri2022optimal}. Similar to the previous subsections in \Cref{Sec:extensions}, we directly present the result, with the details of the test $\phi_{\mathcal{W}_{\alpha}}^{\mathrm{RM}}$ deferred to \Cref{sec:test-2moment}.
\begin{thm}\label{thm:lessthan2}
Assume $1 \leq \alpha \leq 2$. For any $\varepsilon \in (0, 1)$, if
\begin{equation}\label{eq:rate-less-than2}
    \rho_{t_0} \geq C \biggl\{\sqrt{\frac{p}{t_0 \wedge (n-t_0)}}+\Big(\frac{p}{t_0 \wedge (n-t_0)}\Big)^{\frac{\alpha-1}{\alpha}} + \Big(\frac{\log\log(n)}{t_0 \wedge (n-t_0)}\Big)^{\frac{\alpha-1}{\alpha}}\biggr\},
\end{equation}
for all $t_0 \in [n-1]$, where $C>0$ is some constant that depends only on $\alpha$ and $\varepsilon$, then there exists a test $\phi_{\mathcal{W}_{\alpha}}^{\mathrm{RM}}$ such that 
\[
\mathcal{R}_{\mathcal{W}_{\alpha}^{\otimes}}(\{\rho_{t_0}\}_{t_0 \in[n-1]}, \phi_{\mathcal{W}_{\alpha}}^{\mathrm{RM}}) \leq \varepsilon.
\]
\end{thm}

To interpret the requirement on $\rho_{t_0}$ and compare it with the cases where at least two finite moments exist, we first note that for any $t_0 \in [n-1]$, 
\[
\frac{(n-t_0) \wedge t_0}{2}\leq \frac{t_0(n-t_0)}{n} \leq (n-t_0) \wedge t_0,
\]
i.e.\ the normalising factor in \eqref{Eq:alternative_space_original} is of the same order as $(n-t_0) \wedge t_0$. Since $\|\mu_1-\mu_2\|_2^2 \geq \rho^2_{t_0}$ under~\eqref{eq:H1space_lessthan2moment}, we can rewrite \eqref{eq:rate-less-than2} as
\begin{equation}\label{eq:convert-less2}
        \frac{t_0(n-t_0)}{n} \|\mu_1-\mu_2\|^2_2 \gtrsim {p} + p^{\frac{2\alpha-2}{\alpha}}\big\{t_0 \wedge (n-t_0)\big\}^{\frac{2-\alpha}{\alpha}} + \log^{\frac{2\alpha-2}{\alpha}}\log(n)\big\{t_0 \wedge (n-t_0)\big\}^{\frac{2-\alpha}{\alpha}}.
\end{equation}
We observe that in the special case when $\alpha = 2$, the right hand side of \eqref{eq:convert-less2} reduces to $p + \log\log(n)$. \Cref{thm:lessthan2} shows that when the noise distribution belongs to $\mathcal{W}_2^\otimes$, the worst case testing error is controlled at $\varepsilon$, under this condition on the normalised signal strength. Since $\mathcal{P}_{2,1}^{\otimes} \subseteq \mathcal{W}_{2}^{\otimes}$, we conclude that the minimax testing rate under $\mathcal{P}_{2,1}^\otimes$ satisfies $v^*_{\mathcal{P}_{2,1}^{\otimes}}(p,n,s) \lesssim p + \log\log(8n)$ for all $s \in [p]$, which improves upon the result in \Cref{thm:finitemoment_upperbound_dense} in the special case $\alpha =2$. Combining this with \Cref{thm:finitemoment-lowerbound-minimum} leads to an exact characterisation of the minimax testing rate under $\mathcal{P}_{2,1}^{\otimes}$, i.e.\ $v^*_{\mathcal{P}_{2,1}^{\otimes}} (p,n,s) \asymp p + \log\log(8n)$.

Investigating optimality under the more general alternative hypothesis \eqref{eq:H1space_lessthan2moment} is conceptually more challenging, as it actually necessitates lower bound constructions that are valid at every potential change location $t_0$ in order to derive the optimal form of $\{\rho_{t_0}\}_{t_0 \in [n-1]}$. When $p = 1$, we show in \Cref{prop:lowerbound-lessthan2} in \Cref{sec:<2proof} that $\inf_{\phi}\mathcal{R}_{\mathcal{W}_{\alpha}^{\otimes}}(\{\rho_{t_0}\}_{t_0 \in[n-1]},\phi) \geq 1/2$, if 
\[
\rho_{t_0} \leq \bigl\{t_0 \wedge (n-t_0) \bigr\}^{-(\frac{\alpha-1}{\alpha})},
\]
 for some $t_0 \in [n-1],$ which confirms the optimality of our upper bound \eqref{eq:rate-less-than2} up to logarithmic factors. Establishing optimality in general dimension $p$ is significantly more challenging and we  leave that for future investigation.

\subsection{Change away from boundary} \label{Sec:awayfrombonudary}
In this final extension, we aim to address whether it is possible to achieve the sub-Gaussian rate in the sparse regime, $s\log(ep/s)+\log\log(8n)$ (c.f.\ \Cref{table:testing_rate} with $\alpha = 2$), when it is known \emph{a priori} that the change location is at least some distance away from the endpoints, $1$ and $n$. 
We denote this minimum requirement on the distance from boundary by 
$t^{\mathrm{res}} \geq 1$, and modify the alternative space in \eqref{Eq:alternative_space_original} to include additional restrictions (we use sub/superscripts `res' to denote this):
\begin{equation}\label{eq:H1-restricted}
        \mathrm{H}_1: \theta \in \Theta_{\mathrm{res}}(p,n,s,\rho) := \bigcup_{t_0 = t^{\mathrm{res}}+1}^{n-t^{\mathrm{res}}}\Theta^{(t_0)}(p,n,s,\rho).
\end{equation}
We separately consider the cases $\mathcal{Q} = \mathcal{P}_{\alpha, K}^\otimes$ with $\alpha \geq 4$ and $\mathcal{Q} = \mathcal{G}_{\alpha, K}^\otimes$ for $0 < \alpha \leq 2$. 
Notice that in both cases, the rates in the dense regime are not affected by the level of heavy-tailedness $\alpha$, and they differ from the sub-Gaussian dense rate by at most a factor of $\log\log(n)$; see \Cref{table:testing_rate,table:testing_rate_finite}. Therefore, we are interested in whether the same phenomenon can happen in the sparse regime by considering a `simpler' problem with a restricted alternative parameter space specified in \eqref{eq:H1-restricted}. For a given test $\phi$, we denote the worst case testing error, with \eqref{eq:H1-restricted} as the alternative hypothesis, by~$\mathcal{R}_{\mathcal{Q}, {\mathrm{res}}}(\rho,  \phi)$.

For $\mathcal{Q} = \mathcal{P}_{\alpha, K}^\otimes$ with $\alpha \geq 4$, we modify the median-of-means-type test proposed in Section~\ref{subsubsec:mom-sparse} and provide the details of the new test $\phi_{\mathcal{P}, \mathrm{sparse}}^{\mathrm{MoM+res}}$ in \Cref{sec:test-awayfromboundary}. Our new test relies on a more robust coordinate selection step compared to \eqref{Eq:V_tga}, which is made possible due to the additional assumption that the change location is away from boundary.

\begin{thm} \label{thm:awayfromboundaryP}
    Let $t^{\mathrm{res}} = 32\bigl\{\log(e^2p/s) + s^{-1} \log \log (8n)\bigr\}$ and $\mathcal{Q} = \mathcal{P}_{\alpha, K}^\otimes$ with $\alpha \geq 4$. For any $\varepsilon \in (0, 1)$, there exists a test $\phi_{\mathcal{P}, \mathrm{sparse}}^{\mathrm{MoM+res}}$ such that
\[
\mathcal{R}_{\mathcal{Q}, {\mathrm{res}}}(\rho,  \phi_{\mathcal{P}, \mathrm{sparse}}^{\mathrm{MoM+res}}) \leq \varepsilon,
\]
as long as $\rho^2 \geq C v_{\mathcal{P}, \mathrm{sparse}}^{\mathrm{MoM+res}}$, where 
\[
v_{\mathcal{P}, \mathrm{sparse}}^{\mathrm{MoM+res}} := s\bigl\{\log(ep/s) + \log\log(8n)\bigr\},
\]
and $C$ is a constant that only depends on $\alpha,K$ and $\varepsilon$.
\end{thm}

The rate in \Cref{thm:awayfromboundaryP} offers a significant improvement over the original sparse rate, $s\big\{(p/s)^{2/\alpha} + \log \log(8n)\big\}$, achieved in \Cref{thm:finitemoment_upperbound_sparse}. In particular, the rate~$v_{\mathcal{P}, \mathrm{sparse}}^{\mathrm{MoM+res}} $ depends on the dimension~$p$ only through a logarithmic factor and is independent of $\alpha$---the number of finite moments assumed for the noise variables. We emphasise that this improvement is achieved by restricting the original alternative parameter space in \eqref{Eq:alternative_space_original} to \eqref{eq:H1-restricted} with $t^{\mathrm{res}} = 32\bigl\{\log(e^2p/s) + s^{-1} \log \log (8n)\bigr\}$, i.e.\ we only consider testing a potential change point that is not too close to the boundary. Finally, we observe that 
when $\log(ep/s) \geq \log\log(8n)$, with a very mild condition imposed by $t^{\mathrm{res}}$ that the change is away from the endpoints by at least the order of $\log(ep/s)$, the rate $v_{\mathcal{P}, \mathrm{sparse}}^{\mathrm{MoM+res}}$ matches the sparse sub-Gaussian rate for all $\alpha \geq 4$.

For $\mathcal{Q} = \mathcal{G}_{\alpha, K}^\otimes$ with $0 < \alpha < 2$, a similar result can be shown with the test construction deferred to \Cref{sec:test-awayfromboundary}. 
\begin{thm} \label{thm:awayfromboundaryG}
    Let $t^{\mathrm{res}} = 32\bigl\{\log(e^2p/s) + s^{-1} \log \log (8n)\bigr\}$ and $\mathcal{Q} = \mathcal{G}_{\alpha, K}^\otimes$ with $0 < \alpha < 2$. For any $\varepsilon \in (0, 1)$, there exists a test  $\phi_{\mathcal{G}, \mathrm{sparse}}^{\mathrm{res}}$ such that 
\[
\mathcal{R}_{\mathcal{Q}, {\mathrm{res}}}(\rho,  \phi_{\mathcal{G}, \mathrm{sparse}}^{\mathrm{res}}) \leq \varepsilon,
\]
as long as $\rho^2 \geq C v_{\mathcal{G}, \mathrm{sparse}}^{\mathrm{res}}$, where 
\[
v_{\mathcal{G}, \mathrm{sparse}}^{\mathrm{res}} := s\log(ep/s) + \log\log(8n),
\]
and $C$ is a constant that only depends on $\alpha,K$ and $\varepsilon$.
\end{thm}
\Cref{thm:awayfromboundaryG} parallels \Cref{thm:awayfromboundaryP} and demonstrates that $\phi_{\mathcal{G}, \mathrm{sparse}}^{\mathrm{res}}$ achieves the sparse sub-Gaussian rate when testing against change that is at least the order of $\log(ep/s)+s^{-1}\log\log(8n)$ away from the boundary. We note that this idea of obtaining sub-Gaussian rates by assuming change is away from the boundary is also briefly explored in \cite{yu2022robust}. Compared to their result, our result offers a significant improvement on the requirement of how far a change needs to be away from the endpoints, in order to achieve sub-Gaussian rates; see \Cref{sec:app-away} for detailed discussions. More generally, it would be interesting, from a lower bound perspective, to understand the smallest order of $t^{\mathrm{res}}$ such that it is possible to achieve sub-Gaussian rates. We leave this as an interesting future research direction. 

\section{Discussion}
In this paper, we have studied the problem of testing against a single mean change point for high-dimensional heavy-tailed data. We have characterised the minimax testing rates of this problem up to iterated logarithm factors in both the exponentially-decaying and polynomially-decaying tail cases. The transition boundary between the sparse and dense regimes occurs at $s^*_{\mathcal{G}} \asymp \sqrt{p}\log^{-2/\alpha}(ep)$ for $\mathcal{G}_{\alpha, K}$ with $\alpha \in (0, 2]$. For $\mathcal{P}_{\alpha, K}$, the transition happens at $s_{\mathcal{P}}^* = p^{1/2 - 1/(\alpha-2)}$ when $\alpha \geq 4$ and there is no sparse regime when $\alpha \in [2,4)$. Our results also quantify the costs of heavy-tailed distributions in this problem by comparing to the previous results under Gaussian error assumption \citep{liu2021minimax} and unveil a new phenomenon that the minimax testing rates of mean change point problem undergo a phase transition when the error distribution has finite fourth moment. There are several avenues for future research and we briefly discuss them below.

\bigskip
\noindent \textbf{Spatial dependence.} Throughout this paper, we have assumed independence across coordinates. To relax this independence assumption, one could allow for weak or strong coordinate-wise dependence via $\rho$-mixing, and employ alternative finite-sample analysis tools, as in, for example, \citet{jiang2023robust}. Alternatively, for allowing a general covariance matrix $\Sigma$, if we assume that $\Sigma^{-1/2} E$ has independent components with all eigenvalues of $\Sigma$ being of constant order, then at least in the dense case, all our theoretical results remain valid.  We leave a thorough investigation into these two generalisations for future endeavours.

\bigskip
\noindent \textbf{Adaptation to \texorpdfstring{$\alpha$}{alpha}.} All of our proposed testing procedures require the knowledge of $\alpha$, the tail decay index in the case of $\mathcal{G}_{\alpha,K}$ and the number of finite moments in the case of $\mathcal{P}_{\alpha, K}$, through the choices of parameters. Note that if we under-specify $\alpha$, all of our theoretical guarantees still hold, albeit non-optimal rates achieved by the procedures. On the other hand, an over-specification of $\alpha$ would invalidate our results. In practice, practitioners, based on domain knowledge, usually have a conservative idea on how heavy the tails may be. There have been some recent works on distinguishing between exponentially-decaying and polynomially-decaying tails \citep[e.g.][]{castillo2014methods,bhati2020test} and on estimating the tail index parameter for sub-Weibull distributions \citep{vladimirova2020sub}, which may be combined with our tests to obtain adaptivity. We leave this ambitious task for the future.

\section*{Acknowledgements}
The authors would like to thank Richard Samworth and Yining Chen for helpful discussions, an anonymous Associate Editor and two Referees for constructive comments. The research of YY and ML is (partially) supported by Engineering and Physical Sciences Research Council (EPSRC) grants EP/V013432/1 and EP/Z531327/1,  and that of TW and YC is supported by EPSRC grant EP/T02772X/2.

\bigskip

\bibliographystyle{custom2author}
\bibliography{robust_change}
\newpage
\appendix 

\section*{Appendices}
The proofs of all theoretical results are presented in the Appendices. \Cref{proof:upperbounds} contains proofs of upper bound results, including \Cref{thm:weibullupperbound_dense}, \Cref{thm:weibullupperbound_sparse}, \Cref{thm:finitemoment_upperbound_dense},
\Cref{thm:finitemoment_upperbound_sparse_improve}, \Cref{thm:finitemoment_upperbound_sparse} and \Cref{thm:finitemoment_combinerate}. \Cref{thm:adaptive_upperbound} regarding the adaptive test is proved in \Cref{proof:adaptation}. The lower bound results, \Cref{thm:subweibull-lowerbound-minimum} and \Cref{thm:finitemoment-lowerbound-minimum}, are proved in \Cref{proof:lowerbounds}. Technical details of \Cref{Sec:extensions} are collected in \Cref{sec:app-extensions}. \Cref{Sec:aux} contains auxiliary results. 

\addtocontents{toc}{\protect\setcounter{tocdepth}{4}}
\renewcommand{\contentsname}{Content of Appendices}
\tableofcontents

\section{Proofs of results in Sections~\ref{Sec:hidimtest},~\ref{Sec:robusttest} and~\ref{sec:adaptsparsity}} \label{Sec:proof_main}
\subsection{Proofs of upper bound results in Sections~\ref{Sec:hidimtest} and~\ref{Sec:robusttest} }\label{proof:upperbounds}
Throughout the proofs in this subsection, we fix $P_e \in \mathcal{G}_{\alpha, K}^{\otimes}$ (resp. $\mathcal{P}_{\alpha, K}^{\otimes}$) and write $\mathbb{E}_{\theta}$ in place of $\mathbb{E}_{\theta, P_e}$ for the ease of notation. In every proof, we desire to control the two terms $\sup_{\theta \in \Theta_0(p, n)} \mathbb{E}_\theta \phi$ (`\textbf{null term}') and  $\sup_{\theta \in \Theta_1(p, n, s, \rho)} \mathbb{E}_\theta (1-\phi)$ (`\textbf{alternative term}') respectively. The values of the constants $C_1, C_2, \dotsc$ vary from proof to proof. Note also that the order of the constants in each proof do not necessarily match that in the statement of the result, e.g.~$C_2$ in the proof of Theorem~\ref{thm:weibullupperbound_dense} below corresponds to $C_1$ in the statement of Theorem~\ref{thm:weibullupperbound_dense}.
\subsubsection{Proof of Theorem~\ref{thm:weibullupperbound_dense}}
\noindent \textbf{Null term.} For any $\theta \in \Theta_0(p,n)$, we can write
    \begin{align*}
        Y_{t} &= \frac{\sum_{i=1}^t (X_i - \theta_1) - \sum_{i=1}^t (X_{n+1-i}-\theta_1)}{\sqrt{2t}}.
    \end{align*}
    Observe that $Y_{t} = (Y_t(1), \dotsc, Y_t(p))^\top$ has independent components, each having mean $0$ and variance $1$. Moreover, each $X_i(j) - \theta_1(j)$ is a (centered) sub-Weibull random variable of order $\alpha$ belonging to the class $\mathcal{G}_{\alpha, K}$. Now, we consider the following block diagonal matrix $B \in \mathbb{R}^{2tp \times 2tp}$:
    \[
    B := \begin{pmatrix}
    B^{\mathrm{block}} & 0 & \cdots & 0 \\
    0 & B^{\mathrm{block}} & \cdots & 0 \\
    \vdots & \vdots & \ddots &  \vdots \\
    0 & 0 & \cdots & B^{\mathrm{block}}
    \end{pmatrix},
    \]
    where $B^{\mathrm{block}} = (b_{ij})_{i, j \in[2t]} \in \mathbb{R}^{2t \times 2t}$ is defined as follows:
    \[
    b_{ij} = \begin{cases}
    \frac{1}{2t} \quad &\text{if } i=j, \\
    \frac{1}{t} \quad &\text{if } 1 \leq i \neq j \leq t \text{ or } t < i \neq j \leq 2t, \\
    -\frac{1}{t} \quad &\text{if } 1 \leq i \leq t < j \leq 2t \text{ or } 1 \leq j \leq t < i \leq 2t.
    \end{cases}
    \]
    Let $U_i(j) := X_{i}(j) - \theta_1(j)$ for $i \in [n]$ and $j \in [p]$. Now, we can write 
    \[
    \sum_{j \in [p]} Y_{t}^2(j) = \sum_{j \in [p]} \frac{1}{2t}\bigg(\sum_{i=1}^t U_i(j) - \sum_{i=1}^t U_{n+1-i}(j) \bigg)^2 = \tilde{U}^{\top} B \tilde{U},
    \]
    where $\tilde{U} \in \mathbb{R}^{2tp}$ has its first $2t$ coordinates as $$ (U_1(1), U_2(1), \dotsc, U_t(1),U_{n+1-t}(1),\dotsc,U_n(1))^{\top},$$
    and the remaining entries take the same form but with the coordinate index changing from $1$ to $p$. 
   
    We calculate four different norms of matrix $B$:
    \begin{align*}
    \|B\|_{\mathrm{F}} &= \sqrt{2tp\biggl(\frac{1}{4t^2} + \frac{2t-1}{t^2}\biggr)} \leq 2\sqrt{p}, \\
    \|B\|_{2} &= \frac{1}{2t} + \frac{2t-1}{t} \leq 2, \\
    \|B\|_{2 \rightarrow \infty} &= \max_{i\in [2t]} \sqrt{\sum_{j\in [2t]} b^2_{ij}} \leq \sqrt{\frac{2}{t}}, \\
    \|B\|_{\max} &= 1/t. 
    \end{align*}
    For $\alpha \in [1,2]$, we observe by \Cref{lem:subweibull_ordering} that $\mathcal{G}_{\alpha, K} \subseteq \mathcal{G}_{1, K'}$ for some constant $K'>0$, depending only on $K$. Recall that $A_t = \sum_{j \in [p]}Y^2_t(j) - p$. Thus, for any $\alpha \in (0,2]$, by applying \Cref{prop:quadraticweibulltail}, we have
    \begin{align} \label{Eq:At>rnullbound}
        \mathbb{P}_\theta (A_t > r) &\leq \exp\Biggl\{1-\biggl(\frac{r}{C_1\sqrt{p}}\biggr)^2\Biggr\} + \exp\biggl\{1-\frac{r}{C_1}\biggr\}+\exp\Biggl\{1-\biggl(\frac{r\sqrt{t}}{C_1}\biggr)^{\frac{2\alpha}{2+\alpha} \wedge \frac{2}{3}} \Biggr\} \nonumber \\
        &\quad + \exp\Biggl\{1-\biggl(\frac{rt}{C_1}\biggr)^{\frac{\alpha}{2} \wedge \frac{1}{2} }  \Biggr\},
    \end{align}
    where $C_1 > 0$ is some constant depending only on $\alpha$ and $K$ from \Cref{prop:quadraticweibulltail}. Then, by union bounds (for all four terms) and \Cref{lemma:exp_decay_sum} (for the last two terms), we obtain that for any $\theta \in \Theta_0(p, n)$ and $r \geq C_1 \Bigl\{ \bigl(2^{\frac{\alpha}{2+\alpha} \wedge \frac{1}{3}}-1\bigr)^{-\bigl(\frac{2+\alpha}{\alpha} \vee 3\bigr)} \vee \bigl(2^{\frac{\alpha}{2}\wedge \frac{1}{2}}-1\bigr)^{-\bigl(\frac{2}{\alpha} \vee 2\bigr)}\Bigr\}$,
    \begin{align} \label{eq:dense_max_At_control}
        \mathbb{E}_\theta \phi_{\mathcal{G},\mathrm{dense}} = \mathbb{P}_\theta \bigl(\max_{t \in \mathcal{T} } A_{t,0} > r \bigr)  &\leq e\log_2(n) \exp\Biggl\{-\biggl(\frac{r}{C_1\sqrt{p}}\biggr)^2\Biggr\} + e\log_2(n) \exp\biggl\{-\frac{r}{C_1}\biggr\} \nonumber\\
        &\quad + e\sum_{t \in \mathcal{T}}\exp\Biggl\{-\biggl(\frac{r\sqrt{t}}{C_1}\biggr)^{\frac{2\alpha}{2+\alpha} \wedge \frac{2}{3}} \Biggr\} + e\sum_{t \in \mathcal{T}}\exp\Biggl\{-\biggl(\frac{rt}{C_1}\biggr)^{\frac{\alpha}{2}\wedge \frac{1}{2}} \Biggr\} \nonumber \\
        &\leq e\log_2(n) \exp\Biggl\{-\biggl(\frac{r}{C_1\sqrt{p}}\biggr)^2\Biggr\} + e\log_2(n) \exp\biggl\{-\frac{r}{C_1}\biggr\} \nonumber \\
        &\quad + 2e \exp\Biggl\{-\biggl(\frac{r}{C_1}\biggr)^{\frac{2\alpha}{2+\alpha} \wedge \frac{2}{3}} \Biggr\} + 2e \exp\Biggl\{-\biggl(\frac{r}{C_1}\biggr)^{\frac{\alpha}{2}\wedge \frac{1}{2}} \Biggr\},
    \end{align}
    Thus, when
    \begin{align*}
    r \geq \biggl\{ C_1 \sqrt{p\log(8e\varepsilon^{-1} \log_2(n))} \biggr\} &\vee \biggl\{ C_1 \log(8e\varepsilon^{-1} \log_2(n)) \biggr\} \\
    &\vee \biggl\{ C_1 \log^{\frac{2+\alpha}{2\alpha}\vee \frac{3}{2}}(16e\varepsilon^{-1}) \biggr\} \vee \biggl\{ C_1 \log^{\frac{2}{\alpha}\vee 2}(16e\varepsilon^{-1}) \biggr\},
    \end{align*}
    each of the four terms in~\eqref{eq:dense_max_At_control} can be upper bounded by $\varepsilon/8$. Equivalently, when
    \[
    r \geq C_2\bigl(\sqrt{p\log \log(8n)} + \log \log(8n)\bigr),
    \]
    for some constant $C_2 > 0$, depending only on $\alpha$, $K$ and $\varepsilon$, we have $\mathbb{E}_\theta \phi_{\mathcal{G},\mathrm{dense}} \leq \varepsilon/2$ for any $\theta \in \Theta_0(p,n)$.
    
    \medskip
    \noindent \textbf{Alternative term.} For any $\theta \in \Theta(p, n, s, \rho)$, there exists some $t_0 \in [n]$, such that the mean change happens at time $t_0$, with $\frac{t_0(n-t_0)}{n}\|\mu_1-\mu_2\|^2 \geq \rho^2$. We may assume without loss of generality that $t_0 \leq n/2$. By the definition of $\mathcal{T}$, there exists a unique $\tilde{t} \in \mathcal{T}$ such that $t_0/2 < \tilde{t} \leq t_0$. Note that then we can write
    \begin{align} \label{Eq:alternative_Yt_multivariate}
    Y_{\tilde{t}} &= \sqrt{\frac{\tilde{t}}{2}} (\mu_1 - \mu_2) + \frac{\sum_{i=1}^{\tilde{t}} (X_{i} - \mu_1) - \sum_{i=1}^{\tilde{t}} (X_{n+1-i} - \mu_2)}{\sqrt{2\tilde{t}}} =: \delta + Y'_{\tilde{t}},
    \end{align}
    where $\|\delta\|_2^2 \geq t_0\|\mu_1-\mu_2\|_2^2/4 \geq \rho^2/4$.
     Note also that for all $j \in [p]$, we have $\mathbb{E}_\theta [Y'_{\tilde{t}}(j)] = 0$ and $\mathbb{E}[(Y'_{\tilde{t}}(j))^2] = 1$. By \Cref{prop:subweibull-basic-prop}(b) and \Cref{Lemma:bounded2kmoments}(a), we have $\mathbb{E}[(Y'_{\tilde{t}}(j))^4]  \leq C_3$ for some constant $C_3 > 0$, depending only on $\alpha$ and $K$. When $\rho^2 \geq 8r\geq 8C_2\bigl(\sqrt{p\log \log(8n)} + \log \log(8n)\bigr)$, we have by Chebyshev's inequality that
    \begin{align} \label{Eq:weibull_dense_alt}
        \mathbb{E}_\theta(1-\phi_{\mathcal{G},\mathrm{dense}}) &\leq \mathbb{P}_\theta\biggl(\max_{t \in \mathcal{T}} \sum_{j = 1}^p Y_{t}(j)^2 - p \leq \rho^2/8 \biggr) \leq \mathbb{P}_\theta \biggl( \sum_{j=1}^p \Bigl( Y_{\tilde{t}}(j)^2 - \delta(j)^2 - 1 \Bigr) \leq -\|\delta\|_2^2/2 \biggr) \nonumber \\
        &\leq \frac{4\sum_{j = 1}^p \mathrm{Var}_\theta (Y_{\tilde{t}}(j)^2)}{\|\delta\|_2^4} = \frac{4\sum_{j = 1}^p \mathrm{Var}_\theta \bigl( Y'_{\tilde{t}}(j)^2 + 2\delta(j)Y'_{\tilde{t}}(j) \bigr)}{\|\delta\|_2^4} \nonumber \\
        &\leq \frac{4\sum_{j=1}^p \bigl\{2\mathrm{Var}_\theta (Y'_{\tilde{t}}(j)^2) + 8\delta(j)^2 \mathrm{Var}_\theta (Y'_{\tilde{t}}(j)) \bigr\} }{\|\delta\|_2^4}
        \leq  \frac{\sum_{j=1}^p \bigl\{ 8\mathbb{E}_\theta [Y'_{\tilde{t}}(j)^4] + 32\delta(j)^2 \bigr\} }{\|\delta\|_2^4} \nonumber \\
        &\leq  \frac{8C_3p + 32\|\delta\|_2^2   }{\|\delta\|_2^4} \leq 128\biggl(\frac{C_3p}{\rho^4} + \frac{1}{\rho^2}\biggr) \leq \frac{2C_3}{C_2^2 \log\log(8n)} + \frac{16}{C_2\sqrt{p\log\log(8n)}},
    \end{align}
    where we have used the fact that $\mathrm{Var}(X+Y) \leq 2(\mathrm{Var}(X) + \mathrm{Var}(Y))$ in the fourth inequality.
    Therefore, by having $C_2 > \max \bigl\{64/\varepsilon,\sqrt{8C_3/\varepsilon}\bigr\}$, we are guaranteed that $\mathbb{E}_\theta(1-\phi_{\mathcal{G},\mathrm{dense}}) \leq \varepsilon/2$ and the desired result follows.

    \subsubsection{Proof of Theorem~\ref{thm:weibullupperbound_sparse}}
    \noindent \textbf{Null term.} For any $\theta \in \Theta_0(p,n)$, we have by a union bound that
    \begin{align} \label{Eq:sparse_weibullnull_master}
    \mathbb{E}_\theta \phi_{\mathcal{G},\mathrm{sparse}} \leq \mathbb{P}_\theta(A_{1,a} > r_1) + \sum_{t\in \mathcal{T}\setminus\{1\}} \mathbb{P}_\theta(A_{t,a} > r).
    \end{align}
    We first control the second term in~\eqref{Eq:sparse_weibullnull_master}. Recall the definition of $Y_{t,1}$ and $Y_{t,2}$ from~\eqref{eq:samplespliting} and denote $\mathcal{J}_{t,a} := \{j \in [p]: |Y_{t,2}(j)| \geq a\}$ for $t \in \mathcal{T}$ and $a\geq 0$. Note that $\mathcal{J}_{t,a}$ is a random set. Then,
    \begin{align} \label{eq:sparse_weibullnull}
    \sum_{t\in \mathcal{T}\setminus\{1\}} \mathbb{P}_\theta(A_{t,a} > r) &\leq \sum_{t \in \mathcal{T}\setminus\{1\}}\mathbb{P}_\theta\Biggl(\sum_{j\in\mathcal{J}_{t,a}} \bigl(Y_{t,1}^2(j)-1\bigr) > r\Biggr) \nonumber \\
    &= \sum_{t \in \mathcal{T}\setminus\{1\}}\mathbb{E}_\theta \Biggl[ \mathbb{P}_\theta\Biggl(\sum_{j\in\mathcal{J}_{t,a}} \bigl(Y_{t,1}^2(j)-1\bigr) > r \Biggm| \mathcal{J}_{t,a}\Biggr)  \Biggr] \nonumber \\
    &= \sum_{t \in \mathcal{T}\setminus\{1\}} \sum_{J \subseteq [p]} \Biggl\{\mathbb{P}_\theta\Biggl(\sum_{j\in J} \bigl(Y_{t,1}^2(j)-1\bigr) > r \Biggr) \mathbb{P}_\theta (\mathcal{J}_{t,a}=J )\Biggr\} \nonumber \\
    &\leq \sum_{t \in \mathcal{T}\setminus\{1\}}\mathbb{P}_\theta (|\mathcal{J}_{t,a}| > s) + \sum_{t \in \mathcal{T}\setminus\{1\}} \sup_{J \subseteq [p]: |J| \leq s} \mathbb{P}_\theta\Biggl(\sum_{j\in J} \bigl(Y_{t,1}^2(j)-1\bigr) > r \Biggr),
    \end{align}
    where the third line follows from the independence of $Y_{t,1}$ and $Y_{t,2}$. We now control the two terms in~\eqref{eq:sparse_weibullnull} respectively.
    Using \Cref{prop:sumofweibulltail} with $u_i = t^{-1/2}$ for $i = 1,\dotsc, t/2$ and $u_i = -t^{-1/2}$ for $i=t/2+1,\dotsc, t$, we obtain that for any $t \in \mathcal{T}$, $j\in [p]$ and $x \geq 0$
    \begin{equation*}
        \mathbb{P}_\theta(|Y_{t,2}(j)| \geq x) \leq \exp\biggl\{1-\min\biggl\{ \biggl(\frac{x}{C_1}\biggr)^2, \biggl(\frac{x}{C_1\|u\|_{\beta(\alpha)}}\biggr)^\alpha  \biggr\} \biggr\},
    \end{equation*}
    for some constant $C_1 \geq 1$ depending only on $\alpha$ and $K$. For $\alpha \leq 1$, we have $\|u\|_{\beta(\alpha)} = \|u\|_\infty = t^{-1/2}$ and for $1<\alpha\leq 2$, we have $\|u\|_{\beta(\alpha)} = \|u\|_{\alpha/(\alpha-1)} = t^{1/2-1/\alpha}$. Thus
    \begin{equation} \label{Eq:prob_coordinate_null}
    q_{t,a} := \mathbb{P}_\theta(|Y_{t,2}(j)| \geq a) \leq \exp\biggl\{1-\min\biggl\{ \biggl(\frac{a}{C_1}\biggr)^2, \biggl(\frac{a}{C_1 t^{(-\frac{1}{2})\vee(\frac{1}{2}-\frac{1}{\alpha}) }}\biggr)^\alpha   \biggr\} \biggr\}.
    \end{equation}
    For $0 < \alpha < 2$, by a binomial tail bound \citep[][eq (2.1) in Theorem~1]{hoeffding1963}, we have
    \begin{align} \label{Eq:heoffding_bound_result}
    \mathbb{P}_\theta (|\mathcal{J}_{t,a}| > s) &\leq \biggl( \frac{q_{t,a}}{s/p} \biggr)^s \biggl( \frac{1 - q_{t,a}}{1-s/p} \biggr)^{p-s} = \biggl( \frac{pq_{t,a}}{s} \biggr)^s \biggl( 1 + \frac{s - pq_{t,a}}{p-s} \biggr)^{p-s} \nonumber \\
    &\leq \biggl( \frac{pq_{t,a}}{s} \biggr)^s e^{s-pq_{t,a}} \leq \biggl( \frac{epq_{t,a}}{s} \biggr)^s.
    \end{align}
    Combining this with~\eqref{Eq:prob_coordinate_null}, we have
    \begin{align} \label{Eq:use_Chernoff_Hoeffding_weibull}
    &\sum_{t \in\mathcal{T}\setminus\{1\}}\mathbb{P}_\theta (|\mathcal{J}_{t,a}| > s)  \leq \sum_{t \in\mathcal{T}\setminus\{1\}}\biggl( \frac{epq_{t,a}}{s} \biggr)^s \nonumber \\
    &\leq \log_2(n)\biggl(\frac{2e^2p}{s}\biggr)^s \exp\biggl\{-\frac{sa^2}{C_1^2}\biggr\} + \biggl(\frac{2e^2p}{s}\biggr)^s \sum_{t \in\mathcal{T}\setminus\{1\}} \exp\Biggl\{-s\Biggl(\frac{a^{\frac{2\alpha}{\alpha \wedge (2-\alpha)}}t}{C_1^{\frac{2\alpha}{\alpha \wedge (2-\alpha)}}}\Biggr)^{\frac{\alpha \wedge (2-\alpha)}{2}}\Biggr\} \nonumber \\
    &\leq \log_2(n)\biggl(\frac{2e^2p}{s}\biggr)^s \exp\biggl\{-\frac{sa^2}{C_1^2}\biggr\} + 2\biggl(\frac{2e^2p}{s}\biggr)^s \exp\biggl\{-\frac{sa^\alpha}{C_1^\alpha}\biggr\},
    \end{align}
    provided that $a \geq C_1 \bigl(2^{\frac{\alpha \wedge (2-\alpha)}{2}}-1\bigr)^{-1/\alpha}$, where we have used \Cref{lemma:exp_decay_sum} in the last inequality. In fact, for $\alpha=2$, by~\eqref{Eq:prob_coordinate_null}, the final bound in~\eqref{Eq:use_Chernoff_Hoeffding_weibull} remains valid for all $a \geq 0$. Now, the first term in the final bound above can be bounded by $\varepsilon/16$ when
    \[
    a \geq C_1 s^{-1/2} \log^{1/2}(16\varepsilon^{-1}\log_2(n)) + C_1 \log^{1/2}(2e^2p/s),
    \]
    and the second term can be bounded by $\varepsilon/16$ when
    \[
    a \geq C_1 s^{-1/\alpha} \log^{1/\alpha}(32\varepsilon^{-1}) + C_1 \log^{1/\alpha}(2e^2p/s).
    \]
    Thus, as long as we choose $a$ to satisfy
    \begin{equation*}
    a \geq C_2\bigl( \log^{1/\alpha} (ep/s) + s^{-1/2} \log^{1/2}(\varepsilon^{-1}\log(8n)) + s^{-1/\alpha} \log^{1/\alpha}(e\varepsilon^{-1})\bigr)
    \end{equation*} 
    for some large enough $C_2 > 0$, depending only on $\alpha$ and $K$,
    we are guaranteed that
    \begin{equation} \label{Eq:sparse_weibullnull_term1}
    \sum_{t \in\mathcal{T}\setminus\{1\}}\mathbb{P}_\theta (|\mathcal{J}_{t,a}| > s) \leq \varepsilon/8.
    \end{equation}
    We now begin to bound the second term in~\eqref{eq:sparse_weibullnull}. By replacing $p$ with $|J|$ in~\eqref{Eq:At>rnullbound}, we have
    \begin{align*}
    \mathbb{P}_\theta\Biggl(\sum_{j\in J} \bigl(Y_{t,1}^2(j)-1\bigr) > r \Biggr) \leq  \exp\Biggl\{1-\biggl(\frac{r}{C_3\sqrt{|J|}}\biggr)^2\Biggr\} &+ \exp\biggl\{1-\frac{r}{C_3}\biggr\} + \exp\Biggl\{1-\biggl(\frac{r\sqrt{t}}{C_3}\biggr)^{\frac{2\alpha}{2+\alpha}\wedge \frac{2}{3}} \Biggr\} \\
    &+ \exp\Biggl\{1-\biggl(\frac{rt}{C_3}\biggr)^{\frac{\alpha}{2}\wedge\frac{1}{2}} \Biggr\},
    \end{align*}
    where $C_3 > 0$ is some constant depending only on $\alpha$ and $K$ from \Cref{prop:quadraticweibulltail}. Then, by the same technique as used in~\eqref{eq:dense_max_At_control} in the proof of Theorem~\ref{thm:weibullupperbound_dense} (applying union bounds to the first two terms, and using Lemma~\ref{lemma:exp_decay_sum} for the last two terms), we obtain that for all $r \geq C_3 \Bigl\{ \bigl(2^{\frac{\alpha}{2+\alpha} \wedge \frac{1}{3}}-1\bigr)^{-\bigl(\frac{2+\alpha}{\alpha} \vee 3\bigr)} \vee \bigl(2^{\frac{\alpha}{2}\wedge \frac{1}{2}}-1\bigr)^{-\bigl(\frac{2}{\alpha} \vee 2\bigr)}\Bigr\}$
    \begin{align} \label{Eq:sparse_weibullnull_term2}
    &\sum_{t \in \mathcal{T}\setminus\{1\}} \sup_{J \subseteq [p]: |J| \leq s} \mathbb{P}_\theta\Biggl(\sum_{j\in J} \bigl(Y_{t,1}^2(j)-1\bigr) > r \Biggr) \nonumber \\
    &\leq \sum_{t \in \mathcal{T}\setminus\{1\}} \Biggl\{ \exp\Biggl\{1-\biggl(\frac{r}{C_3\sqrt{s}}\biggr)^2\Biggr\} + \exp\biggl\{1-\frac{r}{C_3}\biggr\} + \exp\Biggl\{1-\biggl(\frac{r\sqrt{t}}{C_3}\biggr)^{\frac{2\alpha}{2+\alpha}\wedge \frac{2}{3}} \Biggr\} \nonumber \\
    &\qquad \qquad \qquad  + \exp\Biggl\{1-\biggl(\frac{rt}{C_3}\biggr)^{\frac{\alpha}{2}\wedge\frac{1}{2}} \Biggr\} \Biggr\} \nonumber \\
    &\leq e\log_2(n) \exp\Biggl\{-\biggl(\frac{r}{C_3\sqrt{s}}\biggr)^2\Biggr\} + e\log_2(n) \exp\biggl\{-\frac{r}{C_3}\biggr\} + \sum_{t \in \mathcal{T}\setminus\{1\}}  \exp\Biggl\{1-\biggl(\frac{r\sqrt{t}}{C_3}\biggr)^{\frac{2\alpha}{2+\alpha}\wedge \frac{2}{3}} \Biggr\}  \nonumber \\
    &\quad  + \sum_{t \in \mathcal{T}\setminus\{1\}} \exp\Biggl\{1-\biggl(\frac{rt}{C_3}\biggr)^{\frac{\alpha}{2}\wedge\frac{1}{2}} \Biggr\} \nonumber \\
    &\leq e\log_2(n) \exp\Biggl\{-\biggl(\frac{r}{C_3\sqrt{s}}\biggr)^2\Biggr\} + e\log_2(n) \exp\biggl\{-\frac{r}{C_3}\biggr\} + 2e\exp\Biggl\{-\biggl(\frac{r}{C_3}\biggr)^{\frac{2\alpha}{2+\alpha}\wedge \frac{2}{3}}\Biggr\} \nonumber \\
    &\quad + 2e\exp\Biggl\{-\biggl(\frac{r}{C_3}\biggr)^{\frac{\alpha}{2}\wedge\frac{1}{2}}\Biggr\}.
    \end{align}
    Thus, when
    \begin{align*}
    r \geq \biggl\{ C_3 \sqrt{s\log(32e\varepsilon^{-1} \log_2(n))} \biggr\} &\vee \biggl\{ C_3 \log(32e\varepsilon^{-1} \log_2(n)) \biggr\} \\
    &\vee \biggl\{ C_3 \log^{\frac{2+\alpha}{2\alpha}\vee \frac{3}{2}}(64e\varepsilon^{-1}) \biggr\} \vee \biggl\{ C_1 \log^{\frac{2}{\alpha}\vee 2}(64e\varepsilon^{-1}) \biggr\},
    \end{align*}
    each of the four terms in~\eqref{Eq:sparse_weibullnull_term2} can be upper bounded by $\varepsilon/32$. Equivalently, when
    \[
    r \geq C_4\Bigl(\sqrt{s\log(\varepsilon^{-1}\log(8n))} + \log(\varepsilon^{-1}\log(8n)) + \log^{\frac{2}{\alpha}\vee 2}(e\varepsilon^{-1})\Bigr),
    \] 
   for some constant $C_4 > 0$, depending only on $\alpha$ and $K$, we are guaranteed
   \begin{equation} \label{Eq:sparse_weibullnull_term2_mod}
       \sum_{t \in \mathcal{T}\setminus\{1\}} \sup_{J \subseteq [p]: |J| \leq s} \mathbb{P}_\theta\Biggl(\sum_{j\in J} \bigl(Y_{t,1}^2(j)-1\bigr) > r \Biggr) \leq \varepsilon/8.
   \end{equation}
   Finally, for the first term in~\eqref{Eq:sparse_weibullnull_master}, by Proposition~\ref{prop:2sample}(a), whenever $r_1 \geq C'_4s\log^{2/\alpha}(ep/s)$ for some sufficiently large $C'_4 > 0$, depending on $\alpha, K$ and $\varepsilon$, we have
    \begin{equation} \label{Eq:subweibull_null_t=1}
    \mathbb{P}_\theta(A_{1,a} > r_1) \leq \varepsilon/4.
    \end{equation}
    By combining~\eqref{Eq:sparse_weibullnull_master},~\eqref{eq:sparse_weibullnull},~\eqref{Eq:sparse_weibullnull_term1},~\eqref{Eq:sparse_weibullnull_term2_mod} and~\eqref{Eq:subweibull_null_t=1}, we conclude that $\mathbb{E}_\theta \phi_{\mathcal{G},\mathrm{sparse}} \leq \varepsilon/2$ for all $\theta \in \Theta_0(p,n)$.
    
    \medskip
    \noindent \textbf{Alternative term.} We use the same argument as at the beginning of the alternative part of the proof of Theorem~\ref{thm:weibullupperbound_dense}. Recall that there exists a unique $\tilde{t} \in \mathcal{T}$ such that $t_0/2 < \tilde{t} \leq t_0$. We first consider the case $t_0 \geq 2$. This implies $\tilde{t} \geq 2$. Now, similar to~\eqref{Eq:alternative_Yt_multivariate}, we can write
    \begin{align*}
    Y_{\tilde{t}, 1} &= \frac{\sqrt{\tilde{t}}}{2} (\mu_1 - \mu_2) + \frac{\sum_{i=1}^{\tilde{t}/2} (X_{2i-1} - \mu_1) - \sum_{i=1}^{\tilde{t}/2} (X_{n-2i+1} - \mu_2)}{\sqrt{\tilde{t}}} =: \delta + Y'_{\tilde{t}, 1}, \nonumber \\
    Y_{\tilde{t}, 2} &= \frac{\sqrt{\tilde{t}}}{2} (\mu_1 - \mu_2) + \frac{\sum_{i=1}^{\tilde{t}/2} (X_{2i} - \mu_1) - \sum_{i=1}^{\tilde{t}/2} (X_{n-2i+2} - \mu_2)}{\sqrt{\tilde{t}}} =: \delta + Y'_{\tilde{t}, 2}.
    \end{align*}
    The quantity $\delta := \sqrt{\tilde{t}}(\mu_1-\mu_2)/2$ satisfies $\|\delta\|_2^2 \geq \rho^2/8$. Denote $\mathcal{S}_{\delta} := \{j \in [p]: \delta(j) \neq 0\}$ and $\mathcal{H}_{\delta,a} := \{j \in [p]: |\delta(j)| \geq 2a\}$. Note that these two sets are deterministic, while $\mathcal{J}_{\tilde{t},a} = \{j \in [p]: |Y_{\tilde{t},2}(j)| \geq a\}$ is random. Then, when $\rho^2 \geq 192(r+2s)\log(8/\varepsilon)$, we have
    \begin{align} \label{Eq:weibull_sparse_alter_ctrl}
    \mathbb{E}_\theta(1-\phi_{\mathcal{G},\mathrm{sparse}}) &\leq \mathbb{P}_\theta\biggl( \sum_{j \in \mathcal{J}_{\tilde{t}, a}} \bigl(Y_{\tilde{t},1}^2(j)-1\bigr) \leq r\biggr) \nonumber \\
    &= \mathbb{P}_\theta\biggl( \sum_{j \in \mathcal{J}_{\tilde{t}, a}\cap \mathcal{H}^c_{\delta, a}} \bigl(Y_{\tilde{t},1}^2(j)-1\bigr) + \sum_{j \in \mathcal{J}_{\tilde{t}, a}\cap \mathcal{H}_{\delta, a}} \bigl(Y_{\tilde{t},1}^2(j)-1\bigr) \leq r \biggr) \nonumber \\
    &\leq \mathbb{P}_\theta \bigl(|\mathcal{J}_{\tilde{t},a}| > 2s\bigr) + \mathbb{P}_\theta\biggl( \sum_{j \in \mathcal{J}_{\tilde{t}, a}\cap \mathcal{H}_{\delta, a}} \bigl(Y_{\tilde{t},1}^2(j)-1\bigr) \leq r+2s \biggr)  \nonumber \\
    &\leq \mathbb{P}_\theta \bigl(|\mathcal{J}_{\tilde{t},a}| > 2s\bigr) + \mathbb{P}_\theta\biggl(\sum_{j\in \mathcal{J}_{\tilde{t}, a} \cap \mathcal{H}_{\delta, a}} \delta(j)^2  < \frac{\|\delta\|_2^2}{12\log(8/\varepsilon)}  \biggr)  \nonumber \\
    &\quad + \mathbb{P}_\theta \biggl( \sum_{j\in \mathcal{J}_{\tilde{t}, a} \cap \mathcal{H}_{\delta, a}} \bigl(Y_{\tilde{t},1}^2(j)-\delta(j)^2-1\bigr) \leq - \frac{\|\delta\|_2^2}{24\log(8/\varepsilon)} \biggr).
    \end{align}
    We now control the three terms in~\eqref{Eq:weibull_sparse_alter_ctrl} respectively. By~\eqref{Eq:sparse_weibullnull_term1}, we have
    \begin{align} \label{Eq:concentration_binomial_alternative}
    \mathbb{P}_\theta \bigl(|\mathcal{J}_{\tilde{t},a}| > 2s\bigr) \leq \mathbb{P}_\theta \bigl(|\mathcal{J}_{\tilde{t},a} \cap \mathcal{S}^c_\delta| > s\bigr) \leq \varepsilon/8.
    \end{align}
    For the second term, we observe that for all $j \in \mathcal{H}_{\delta,a}$
    \begin{align} \label{Eq:prob_coordinate_alternative}
    \mathbb{P}_\theta(j \notin \mathcal{J}_{\tilde{t},a})=\mathbb{P}_\theta(|Y_{\tilde{t}, 2}(j)| < a) &= \mathbb{P}_\theta(|\delta(j) + Y'_{\tilde{t}, 2}(j)| < a) \leq \mathbb{P}_\theta(|Y'_{\tilde{t}, 2}(j)| > |\delta(j)| - a) \nonumber \\
    &\leq \exp\bigl\{1 - \bigl((|\delta(j)|-a)/C_1\bigr)^\alpha\bigr\} \leq \frac{1}{256\log(8/\varepsilon)},
    \end{align}
    where the penultimate inequality follows from~\eqref{Eq:prob_coordinate_null} and the last two inequalities follow from the choice $a \geq C_1\log^{1/\alpha}(700\log(8/\varepsilon))$. Consequently,
    \begin{align} \label{Eq:variance_bernoulli}
    \sum_{j\in\mathcal{H}_{\delta, a}} \mathrm{Var}_\theta\bigl( \delta(j)^2\mathbbm{1}_{ \{ j \in \mathcal{J}_{\tilde{t},a} \} }  \bigr) &\leq \sum_{j\in\mathcal{H}_{\delta, a}} \delta(j)^4 \mathbb{P}_\theta(j \notin \mathcal{J}_{\tilde{t},a}) \leq \frac{\sum_{j=1}^p \delta(j)^4}{256\log(8/\varepsilon)} \leq \frac{\|\delta\|_2^4}{256\log(8/\varepsilon)}.
    \end{align}
    Moreover, when $\rho^2 \geq 64a^2s$, we obtain
    \begin{equation} \label{Eq:signal_captured}
    \sum_{j\in\mathcal{H}_{\delta, a}} \delta(j)^2 \geq \|\delta\|_2^2 - s(2a)^2 \geq \|\delta\|_2^2/2.
    \end{equation}
    We first consider the case $\|\delta\|_2 \geq \sqrt{12\log(8/\varepsilon)}\|\delta\|_\infty$. Then, by combining~\eqref{Eq:prob_coordinate_alternative},~\eqref{Eq:variance_bernoulli},~\eqref{Eq:signal_captured} and Bernstein's inequality, we have
    \begin{align} \label{Eq:concentration_bernoulli}
    &\mathbb{P}_\theta\biggl(\sum_{j\in \mathcal{J}_{\tilde{t}, a} \cap \mathcal{H}_{\delta, a}} \delta(j)^2  < \|\delta\|_2^2/8 \biggr) = \mathbb{P}_\theta\biggl(\sum_{j\in\mathcal{H}_{\delta, a}} \delta(j)^2 \mathbbm{1}_{ \{ j \in \mathcal{J}_{\tilde{t},a} \} } < \|\delta\|_2^2/8 \biggr) \nonumber \\
    &\leq  \mathbb{P}_\theta\biggl(\sum_{j\in\mathcal{H}_{\delta, a}} \delta(j)^2 \bigl(\mathbbm{1}_{ \{ j \in \mathcal{J}_{\tilde{t},a} \} } - \mathbb{P}_\theta(j \in \mathcal{J}_{\tilde{t},a}) \bigr) < -\|\delta\|_2^2/8 \biggr) \nonumber\\
    &\leq \exp\biggl\{ -\frac{\|\delta\|_2^4/64}{2\sum_{j \in \mathcal{H}_{\delta, a}} \mathrm{Var}_\theta\bigl( \delta(j)^2\mathbbm{1}_{ \{ j \in \mathcal{J}_{\tilde{t},a} \} }  \bigr) + \|\delta\|_\infty^2\|\delta\|_2^2/12   } \biggr\} \nonumber \\
    &\leq \max\Biggl\{ \exp\{-\log(8/\varepsilon)\}, \exp\biggl\{ -\frac{\|\delta\|_2^2}{12\|\delta\|_\infty^2   } \biggr\}  \Biggr\} \leq \varepsilon/8.
    \end{align}
    If instead $\|\delta\|_\infty \leq \|\delta\|_2 < \sqrt{12\log(8/\varepsilon)}\|\delta\|_\infty$, we assume that $|\delta(j^*)| = \|\delta\|_\infty$ for some $j^* \in \mathcal{H}_{\delta, a}$. Note that when $\rho^2 \geq 384C_1^2\log^{\frac{\alpha+2}{\alpha}}(8e/\varepsilon)$, we have $|\delta(j^*)| \geq 2C_1\log^{1/\alpha}(8e/\varepsilon)$ and thus
    \begin{align} \label{Eq:concentration_bernoulli_single}
    &\mathbb{P}_\theta\biggl(\sum_{j\in \mathcal{J}_{\tilde{t}, a} \cap \mathcal{H}_{\delta, a}} \delta(j)^2 <  \frac{\|\delta\|_2^2}{12\log(8/\varepsilon)} \biggr)  \leq \mathbb{P}_\theta\biggl(\delta(j^*)^2 \mathbbm{1}_{ \{ j^* \in \mathcal{J}_{\tilde{t},a}\} } < \frac{\|\delta\|_2^2}{12\log(8/\varepsilon)} \biggr) \nonumber \\
    &\leq \mathbb{P}_\theta(|Y_{\tilde{t},2}(j^*)| < a) \leq \exp\{1-(|\delta(j^*)|/(2C_1))^\alpha\}  \leq \varepsilon/8.
    \end{align}
    For the third and final term in~\eqref{Eq:weibull_sparse_alter_ctrl}, we have by Chebyshev's inequality that
    \begin{align} \label{Eq:subweibull_sparse_chebyshev}
        &\mathbb{P}_\theta \biggl( \sum_{j\in \mathcal{J}_{\tilde{t}, a} \cap \mathcal{H}_{\delta, a}} \bigl(Y_{\tilde{t},1}^2(j)-\delta(j)^2-1\bigr) \leq - \frac{\|\delta\|_2^2}{24\log(8/\varepsilon)} \biggr) \nonumber \\
        &\leq \frac{\sum_{j\in\mathcal{H}_{\delta, a}} \mathrm{Var}_\theta \Bigl(\bigl(Y_{\tilde{t},1}^2(j)-\delta(j)^2-1\bigr)\mathbbm{1}_{ \{ |Y_{\tilde{t},2}(j)| \geq a \} } \Bigr)}{\|\delta\|_2^4/(576\log^2(8/\varepsilon))} \leq \frac{\sum_{j\in\mathcal{H}_{\delta, a}} \mathrm{Var}_\theta \bigl(Y_{\tilde{t},1}^2(j)\bigr) }{\|\delta\|_2^4/(576\log^2(8/\varepsilon))} \nonumber \\
        &\leq C_5\log^2(8/\varepsilon)\biggl(\frac{s}{\rho^4} + \frac{1}{\rho^2}\biggr),
    \end{align}
    where $C_5\geq 1$ is a constant depending on $\alpha$ and $K$ and the penultimate inequality follows from a similar argument to~\eqref{Eq:weibull_dense_alt}. Hence, when
    \[
    \rho^2 \geq C_5 \varepsilon^{-1}\max\Bigl\{192(r+2s)\log(8/\varepsilon), 64a^2s, 384C_1^2\log^{\frac{\alpha+2}{\alpha}}(8e/\varepsilon) \Bigr\},
    \]
    we have by combining~\eqref{Eq:weibull_sparse_alter_ctrl},~\eqref{Eq:concentration_binomial_alternative},~\eqref{Eq:concentration_bernoulli},~\eqref{Eq:concentration_bernoulli_single} and~\eqref{Eq:subweibull_sparse_chebyshev} that
    \begin{align*}
    &\mathbb{E}_\theta(1-\phi_{\mathcal{G},\mathrm{sparse}}) \leq \varepsilon/4 + C_5\log^2(8/\varepsilon)\biggl(\frac{s}{\rho^4} + \frac{1}{\rho^2}\biggr) \leq \varepsilon/2.
    \end{align*}
    Finally, We consider the case that the mean change happens at $t_0=1$ instead. Recall that in this case we have $\tilde{t}=1$. \eqref{Eq:weibull_sparse_alter_ctrl} remains true when $\rho^2 \geq 192(r_1+2s)\log(8/\varepsilon)$ if we redefine $\mathcal{J}_{\tilde{t}=1, a} := \{j \in [p]: |Y_1(j)|\geq a\}$. All three terms in~\eqref{Eq:weibull_sparse_alter_ctrl} can be controlled in the same way as when $t_0 \geq 2$ and this completes the proof.

\subsubsection{Proof of Theorem~\ref{thm:finitemoment_upperbound_dense}}
We first prove the result for $\alpha \geq 4$.\\
\noindent \textbf{Null term.} For any $\theta \in \Theta_0(p,n)$, we have $\mathbb{E}_\theta \overline{Z}_{t,g}(j) = 0$ and $\mathrm{Var}_\theta \overline{Z}_{t,g}(j) = G_t/t$ for every $t \in \mathcal{T}$, $g \in [G_t]$ and $j\in [p]$. Furthermore, from the class assumption $\mathbb{E}|E_i(j)|^\alpha \leq K^\alpha$, for all $i \in [n]$ and $j \in [p]$ and Jensen's inequality, we deduce $\mathbb{E} E_i(j)^4 \leq K^4$. We thus obtain, for all $i \leq n/2$ and $j \in [p]$
\begin{equation} \label{eq:4thmoment_C1}
\mathbb{E}_\theta Z_i^4(j) = \mathbb{E}_\theta \biggl[ \frac{X_i(j) - X_{n-i}(j)}{\sqrt{2}}\biggr]^4 = \frac{\mathbb{E}_\theta \bigl[E_i(j) - E_{n-i}(j)\bigr]^4}{4} \leq \frac{K^4+3}{2} =: C_1.
\end{equation}
Then, by Chebyshev's inequality (or, alternatively, Lemma~\ref{lemma:lowmoment-2}) and \Cref{Lemma:bounded2kmoments}(a), with $r_t = C_2 \sqrt{p} G_t$, we have for all $t \in \mathcal{T}$ and $g \in [G_t]$ that
\begin{align} \label{Eq:MoM_robust_null}
\mathbb{P}_\theta\biggl(t\sum_{j=1}^p V_{t,g}(j) > r_t\biggr) &=  \mathbb{P}_\theta\biggl(\sum_{j=1}^p \biggl(\overline{Z}^2_{t,g}(j) - \frac{G_t}{t}\biggr) > \frac{C_2\sqrt{p}G_t}{t}\biggr) \leq \frac{t^2\sum_{j=1}^p\mathbb{E}_\theta \overline{Z}^4_{t,g}(j)}{C^2_2pG_t^2} \nonumber \\
&\leq \frac{3pt^2 C_1 (G_t/t)^2}{C^2_2pG_t^2}  \leq \frac{3C_1}{C_2^2} \leq \varepsilon/36,
\end{align}
where $C_2$ is chosen to satisfies $C_2 \geq \sqrt{108C_1\varepsilon^{-1}}$. We denote
\[
\mathcal{B}_t := \biggl\{g \in [G_t]: t\sum_{j=1}^p V_{t,g}(j) > r_t\biggr\}.
\]
By~\eqref{Eq:MoM_robust_null} and the multiplicative Chernoff bound \cite[e.g.][Corollary 4.9]{mitzenmacher2017probability}, we have for $t \in \mathcal{T}$
\begin{align} \label{Eq:binomial_bound_robust}
&\mathbb{P}_\theta (A_t^{\mathrm{MoM}} > r_t) \leq  \mathbb{P}_\theta(|\mathcal{B}_t| \geq G_t/2) =\mathbb{P}_\theta\Biggl(|\mathcal{B}_t| \geq  \frac{\varepsilon G_t}{36}\biggl(1 + \Bigl( \frac{18}{\varepsilon}-1 \Bigr) \biggr)  \Biggr)  \nonumber \\
&\leq \exp\biggl\{-\frac{\varepsilon G_t}{36} \biggl( \frac{18}{\varepsilon }\log\Bigl(\frac{18}{\varepsilon}\Bigr) - \frac{18}{\varepsilon} + 1\biggr)  \biggr\} \leq \exp\biggl\{ -\frac{G_t}{2} \log\bigl(6/\varepsilon\bigr)\biggr\}.
\end{align}
Thus, by~\eqref{Eq:MoM_robust_null},~\eqref{Eq:binomial_bound_robust}, the choices of $G_t$ and $\Delta$ in~\eqref{Eq:robust_theorem_parameters} and a union bound, we conclude that
\begin{align} \label{Eq:MoM_robust_null_together}
\mathbb{E}_\theta \phi_{\mathcal{P},\mathrm{dense}} &\leq \mathbb{P}_\theta\biggl(\sum_{j=1}^p V_{t=1,1}(j) > r_{t=1}\biggr) + \sum_{t \in \mathcal{T}: \, 2 \leq t \leq \Delta} \mathbb{P}_\theta\bigl(A_t^{\mathrm{MoM}} > r_t \bigr) + \sum_{t \in \mathcal{T}: \,  t > \Delta} \mathbb{P}_\theta\bigl(A_t^{\mathrm{MoM}} > r_t \bigr)  \nonumber \\
&\leq \varepsilon/36 + \sum_{t \in \mathcal{T}: \, 2 \leq t \leq \Delta}  (6/\varepsilon)^{-t/2} + \sum_{t \in \mathcal{T}: \,  t > \Delta} (6/\varepsilon)^{-\Delta/2} \nonumber \\
&\leq \varepsilon/36 + \frac{(6/\varepsilon)^{-1}}{1-(6/\varepsilon)^{-1}} + \log_2(n/2)(6/\varepsilon)^{-\Delta/2} \leq \varepsilon/36 + \varepsilon/5 + \varepsilon/5  < \varepsilon/2,
\end{align}
for all $\theta \in \Theta_0(p,n)$.

\medskip
\noindent \textbf{Alternative term.} We again follow the argument in the first paragraph of the alternative term part of the proof of Theorem~\ref{thm:weibullupperbound_dense}. In particular, recall that there exists a unique $\tilde{t} \in \mathcal{T}$ such that $t_0/2 < \tilde{t} \leq t_0$, where $t_0$ (without loss of generality $t_0 \leq n/2$) is the true mean change location. For all $i \leq n/2$, we denote
\[
Z'_i := Z_i - \frac{\mu_1 - \mu_2}{\sqrt{2}} = \frac{(X_i - \mu_1) - (X_{n+1-i}-\mu_2)}{\sqrt{2}},
\]
and correspondingly $\overline{Z}'_{\tilde{t},g} := \overline{Z}_{\tilde{t},g} - (\mu_1 - \mu_2)/\sqrt{2}$, for $g \in [G_{\tilde{t}}]$. It follows from the null term part of the proof that $\mathbb{E}_\theta \overline{Z}'_{\tilde{t},g}(j) = 0$, $\mathrm{Var}_\theta \overline{Z}'_{\tilde{t},g}(j) = G_{\tilde{t}}/\tilde{t}$ and $\mathbb{E}_\theta (Z'_i(j))^4 \leq C_1$, where $C_1$ is as in~\eqref{eq:4thmoment_C1}. When $\rho^2 \geq 16C_2\sqrt{p}\Delta$, we have
\[
2\tilde{t}\|\mu_1-\mu_2\|^2 \geq \frac{t_0(n-t_0)}{n}\|\mu_1-\mu_2\|^2 \geq \rho^2 \geq 16C_2\sqrt{p}G_{\tilde{t}} = 16r_{\tilde{t}},
\]
since $G_{\tilde{t}} \leq \Delta$. Thus, for all $g \in [G_{\tilde{t}}]$, we have
\begin{align} \label{Eq:MoM_main_alternative}
    &\mathbb{P}_\theta\biggl(\tilde{t}\sum_{j=1}^p V_{\tilde{t},g}(j) \leq r_{\tilde{t}}\biggr) = \mathbb{P}_\theta\Biggl(\sum_{j=1}^p \biggl( \biggl(\overline{Z}'_{\tilde{t},g}(j) + \frac{\mu_1(j) - \mu_2(j)}{\sqrt{2}}\biggr)^2 - \frac{G_{\tilde{t}}}{\tilde{t}}\biggr) \leq \frac{r_{\tilde{t}}}{\tilde{t}} \Biggr)  \nonumber \\
    &= \mathbb{P}_\theta\Biggl(\sum_{j=1}^p \biggl( \bigl(\overline{Z}'_{\tilde{t},g}(j)\bigr)^2 - \frac{G_{\tilde{t}}}{\tilde{t}} + \sqrt{2}\bigl(\mu_1(j) - \mu_2(j)\bigr)\overline{Z}'_{\tilde{t},g}(j) \biggr) \leq \frac{r_{\tilde{t}}}{\tilde{t}}-\frac{\|\mu_1-\mu_2\|_2^2}{2} \Biggr) \nonumber \\
    &\leq \mathbb{P}_\theta\Biggl(\sum_{j=1}^p \biggl( \bigl(\overline{Z}'_{\tilde{t},g}(j)\bigr)^2 - \frac{G_{\tilde{t}}}{\tilde{t}} + \sqrt{2}\bigl(\mu_1(j) - \mu_2(j)\bigr)\overline{Z}'_{\tilde{t},g}(j) \biggr) \leq -\frac{\rho^2}{16\tilde{t}}-\frac{\|\mu_1-\mu_2\|_2^2}{4} \Biggr).
\end{align}
By Chebyshev's inequality and \Cref{Lemma:bounded2kmoments}(a), we obtain
\begin{equation} \label{Eq:robust_fourthmoment}
    \mathbb{P}_\theta\biggl(\sum_{j=1}^p \Bigl( \bigl(\overline{Z}'_{\tilde{t},g}(j)\bigr)^2 - G_{\tilde{t}}/\tilde{t} \Bigr) \leq -\frac{\rho^2}{16\tilde{t}} \biggr) \leq \frac{256(\tilde{t})^2 \sum_{j=1}^p\mathbb{E}_\theta \bigl(\overline{Z}'_{\tilde{t},g}(j)\bigr)^4 }{\rho^4} \leq \frac{768C_1pG^2_{\tilde{t}}}{\rho^4},
\end{equation}
and
\[
    \mathbb{P}_\theta\biggl(\sum_{j=1}^p \sqrt{2}\bigl(\mu_1(j) - \mu_2(j)\bigr)\overline{Z}'_{\tilde{t},g}(j) \leq -\frac{\|\mu_1-\mu_2\|_2^2}{4} \biggr) \leq \frac{32G_{\tilde{t}}\|\mu_1 - \mu_2\|_2^2/\tilde{t}}{\|\mu_1-\mu_2\|_2^4} \leq \frac{64G_{\tilde{t}}}{\rho^2}.
\]
Combining these with~\eqref{Eq:MoM_main_alternative}, as long as 
\[
\rho^2 \geq \max\biggl\{16C_2\sqrt{p}\Delta, 96\sqrt{\frac{2C_4}{\varepsilon}}\sqrt{p}\Delta, \frac{1536\Delta}{\varepsilon}\biggr\},
\]
we are guaranteed
\[
\mathbb{P}_\theta\biggl(\tilde{t}\sum_{j=1}^p V_{\tilde{t},g}(j) \leq r_{\tilde{t}}\biggr) \leq \varepsilon/12.
\]
If $\tilde{t}=1$, then $G_{\tilde{t}} =1$ and we immediately have
\[
\mathbb{E}_\theta(1-\phi_{\mathcal{P},\mathrm{dense}}) \leq \mathbb{P}_\theta \bigl(A_{\tilde{t}}^{\mathrm{MoM}} \leq r_{\tilde{t}}\bigr) = \mathbb{P}_\theta\biggl(\tilde{t}\sum_{j=1}^p V_{\tilde{t},1}(j) \leq r_{\tilde{t}}\biggr) \leq \varepsilon/12.
\]
If $\tilde{t} \geq 2$, then $G_{\tilde{t}} \geq 2$  and we use the same binomial tail bound argument as in~\eqref{Eq:binomial_bound_robust} to conclude that
\[
\mathbb{E}_\theta(1-\phi_{\mathcal{P},\mathrm{dense}}) \leq \mathbb{P}_\theta\bigl(A_{\tilde{t}}^{\mathrm{MoM}} \leq r_{\tilde{t}}\bigr) \leq \exp\biggl\{-\frac{\varepsilon G_{\tilde{t}}}{12} \biggl( \frac{6}{\varepsilon }\log\Bigl(\frac{6}{\varepsilon}\Bigr) - \frac{6}{\varepsilon} + 1\biggr)  \biggr\} 
\leq \Bigl(\frac{2}{\varepsilon}\Bigr)^{-1}.
\]
This completes the proof for $\alpha \geq 4$. We now consider the case $\alpha < 4$. The proof is similar to above and we essentially replace Chebyshev's inequality wherever used by \Cref{lemma:lowmoment-2}. We only highlight the difference for brevity.\\
\noindent \textbf{Null term.} Note that for all $t \in \mathcal{T}$ and $g \in [G_t]$, using \Cref{lemma:lowmoment-2} with $k = \alpha/2 < 2$ and $L = t/G_t$, we have with $r_t = C_2p^{2/\alpha} G_t$ that
\begin{equation}\label{eq:extreme_robust_null}
\mathbb{P}_\theta\biggl(t\sum_{j=1}^p V_{t,g}(j) > r_t\biggr) =  \mathbb{P}_\theta\biggl(\frac{t}{G_t}\sum_{j=1}^p \biggl(\overline{Z}^2_{t,g}(j) - \frac{G_t}{t}\biggr) > C_2p^{2/\alpha}\biggr) \leq \frac{\varepsilon}{36},
\end{equation} 
for $C_2 \geq C_{\alpha/2}(36/\varepsilon)^{2/\alpha}$, where $C_{\alpha/2}>0$ is the constant depending on $\alpha$ and $K$ from \Cref{lemma:lowmoment-2}. By substituting~\eqref{Eq:MoM_robust_null} with~\eqref{eq:extreme_robust_null} and following the rest of the argument in the above proof, we prove that $\mathbb{E}_\theta\phi_{\mathcal{P},\mathrm{dense}} \leq \varepsilon/2$ for all $\theta \in \Theta_0(p,n)$.

\medskip
\noindent \textbf{Alternative term.} For all $g \in [G_{\tilde{t}}]$, again using \Cref{lemma:lowmoment-2} with $k = \alpha/2 < 2$ and $L = \tilde{t}/G_{\tilde{t}}$, we have
\begin{equation} \label{Eq:extreme_robust_alternative}
    \mathbb{P}_\theta\biggl(\sum_{j=1}^p \Bigl( \bigl(\overline{Z}'_{\tilde{t},g}(j)\bigr)^2 - G_{\tilde{t}}/\tilde{t} \Bigr) \leq -\frac{\rho^2}{16\tilde{t}} \biggr)  = \mathbb{P}_\theta\biggl( \frac{\tilde{t}}{G_{\tilde{t}}}\sum_{j=1}^p \Bigl( \bigl(\overline{Z}'_{\tilde{t},g}(j)\bigr)^2 - G_{\tilde{t}}/\tilde{t} \Bigr) \leq -\frac{\rho^2}{16G_{\tilde{t}}} \biggr)  \leq \frac{\varepsilon}{24},
\end{equation}
for $\rho^2 \geq 24^{(2+\alpha)/\alpha}C_{\alpha/2} \varepsilon^{-2/\alpha}p^{2/\alpha}\Delta$, where $C_{\alpha/2}$ is, as above, a constant depending only on $\alpha$ and $K$. By substituting~\eqref{Eq:robust_fourthmoment} with~\eqref{Eq:extreme_robust_alternative} and following the rest of argument in the above proof, we prove that as long as
\[
\rho^2 \geq \max\biggl\{16C_2p^{2/\alpha}\Delta, 24^{(2+\alpha)/\alpha}C_{\alpha/2}\varepsilon^{-2/\alpha}p^{2/\alpha}\Delta, \frac{1536\Delta}{\varepsilon}\biggr\},
\]
we can control $\mathbb{E}_{\theta}(1-\phi_{\mathcal{P},\mathrm{dense}}) \leq \varepsilon/2$.

\subsubsection{Proof of Proposition~\ref{thm:finitemoment_upperbound_sparse}}
\noindent \textbf{Null term.} For any $\theta \in \Theta_0(p,n)$, we have by a union bound that
\begin{align} \label{Eq:sparse_MoMnull_master}
    \mathbb{E}_\theta \phi_{\mathcal{P},\mathrm{sparse}}^{\mathrm{MoM}} \leq \mathbb{P}_\theta(A_{1,a} > r_1) + \sum_{t\in \mathcal{T}\setminus\{1\}} \mathbb{P}_\theta\bigl(A_{t,a}^{\mathrm{MoM}} > r_t\bigr).
\end{align}
We first control the second term. Recall that $\mathcal{J}_{t,a} = \{j \in [p]: |Y_{t,2}(j)| \geq a\}$ for $t \in \mathcal{T}\setminus\{1\}$. For $J \subseteq [p]$, we denote
\[
A_{t,*,J}^{\mathrm{MoM}} := \frac{t}{2}\cdot \mathrm{median} \Biggl\{ \sum_{j\in J} \biggl( \overline{Z}^2_{t,g,1}(j) - \frac{2G_t}{t} \biggr): g \in [G_t]\Biggl\}.
\]
Note that $A_{t,a}^{\mathrm{MoM}} = A_{t,*, \mathcal{J}_{t,a}}^{\mathrm{MoM}}$. Using the same technique as~\eqref{eq:sparse_weibullnull} in the proof of Theorem~\ref{thm:weibullupperbound_sparse}, we have
\begin{align} \label{Eq:sparse_robustnull}
    \sum_{t \in \mathcal{T}\setminus\{1\}} \mathbb{P}_\theta(A_{t,a}^{\mathrm{MoM}} > r_t)  \leq \sum_{t \in \mathcal{T}\setminus\{1\}}\mathbb{P}_\theta (|\mathcal{J}_{t,a}| > s) + \sum_{t \in \mathcal{T}\setminus\{1\}} \sup_{J \subseteq [p]: |J| \leq s} \mathbb{P}_\theta(A_{t,*,J}^{\mathrm{MoM}} > r_t),
\end{align}
where $s$ is the sparsity. From the assumption that $\mathbb{E}|E_i(j)|^\alpha \leq K^\alpha$, for all $i \in [n]$ and $j \in [p]$ and Jensen's inequality, we deduce that
\[
\mathbb{E}_\theta |Z_i(j)|^\alpha  = \frac{\mathbb{E}_\theta \bigl|E_i(j) - E_{n-i}(j)\bigr|^\alpha}{2^{\alpha/2}} \leq \frac{\mathbb{E}_\theta \bigl(|E_i(j)| + |E_{n-i}(j)|\bigr)^\alpha}{2^{\alpha/2}} \leq 2^{\alpha/2}K^\alpha.
\]
Then, by Fuk--Nagaev inequality (\Cref{prop:fuknagaev}), we have
\begin{align} \label{Eq:robust_prob_coordinate_null}
q_{t,a} &= \mathbb{P}_\theta(|Y_{t,2}(j)| \geq a) \leq 2\Biggl( \frac{(\alpha+2)(K^\alpha 2^{\alpha/2}t/2)^{1/\alpha}}{\alpha a\sqrt{t/2}} \Biggr)^\alpha + 2\exp \biggl\{ -\frac{2a^2}{(\alpha+2)^2e^\alpha} \biggr\} \nonumber \\
&\leq \frac{K^\alpha}{(a/3)^\alpha t^{\alpha/2-1}} + \exp \biggl\{1 -\frac{a^2}{2\alpha^2e^\alpha} \biggr\},
\end{align}
where we have used $\alpha \geq 4$ in the last inequality. Similar to~\eqref{Eq:heoffding_bound_result} and~\eqref{Eq:use_Chernoff_Hoeffding_weibull}, by a binomial tail bound, we have
\begin{equation} \label{Eq:chernoff-hoeffding-robust}
    \sum_{t \in \mathcal{T}\setminus\{1\}}\mathbb{P}_\theta (|\mathcal{J}_{t,a}| > s) \leq \sum_{t \in \mathcal{T}\setminus\{1\}} \biggl( \frac{epq_{t,a}}{s} \biggr)^s \leq \biggl(\frac{2epK^\alpha}{s(a/3)^\alpha}\biggr)^s + \log_2(n)\biggl(\frac{2e^2p}{s}\biggr)^s \exp\biggl\{-\frac{sa^2}{2\alpha^2e^\alpha}\biggr\}.
\end{equation}
Thus, as long as we choose $a$ to satisfy
\begin{equation} \label{Eq:choice_a_finitemoment}
a \geq C_1\bigl( \varepsilon^{-1}(p/s)^{1/\alpha} + s^{-1/2} \log^{1/2}(\varepsilon^{-1}\log(8n))\bigr)
\end{equation} 
for some large enough $C_1 > 0$, depending only on $\alpha$ and $K$, we are guaranteed that
\[
\sum_{t \in\mathcal{T}\backslash \{1\}}\mathbb{P}_\theta (|\mathcal{J}_{t,a}| > s) \leq \frac{\varepsilon}{8}.
\]
 Furthermore, By setting $r_t = C_2\sqrt{s}G_t$ with a sufficently large $C_2 > 0$ and $\Delta = 2^{4 + \lceil\log_2 \log \log(8n)\rceil}$ and by following the argument from~\eqref{Eq:MoM_robust_null} to~\eqref{Eq:MoM_robust_null_together}, we can upper bound the second term in~\eqref{Eq:sparse_robustnull} at $\varepsilon/8$ as well. Finally, to control the first term in~\eqref{Eq:sparse_MoMnull_master}, by Proposition~\ref{prop:2sample}(b), whenever $r_1 \geq C'_1s(p/s)^{2/\alpha}$ for sufficiently large $C'_1 > 0$, depending on $\alpha, K$ and $\varepsilon$, we have $\mathbb{P}_\theta(A_{1,a} > r_1) \leq \varepsilon/4$. Hence, we conclude that $\mathbb{E}_\theta \phi_{\mathcal{P},\mathrm{sparse}}^{\mathrm{MoM}} \leq \varepsilon/2$ for all $\theta \in \Theta_0(p,n)$.

\medskip
\noindent \textbf{Alternative term.} Recall the definitions of $\delta$, $\mathcal{S}_{\delta}$ and $\mathcal{H}_{\delta,a}$ from the alternative term part of the proof of \Cref{thm:weibullupperbound_sparse}:
\[
\delta = \frac{\sqrt{\tilde{t}}}{2}(\mu_1-\mu_2), \quad \mathcal{S}_{\delta} = \{j \in [p]: \delta(j) \neq 0\}, \quad \mathcal{H}_{\delta,a} = \{j \in [p]: |\delta(j)| \geq 2a\},
\]
and the notation $\overline{Z}'_{\tilde{t},g} := \overline{Z}_{\tilde{t},g} - (\mu_1 - \mu_2)/\sqrt{2}$, for $g \in [G_{\tilde{t}}]$ introduced at the start of the alternative term part of the proof of Theorem~\ref{thm:finitemoment_upperbound_dense}. We first consider the case $t_0 \geq 2$, which implies $\tilde{t} \geq 2$. For $J \subseteq [p]$, we further denote
\[
A_{t,*,J}^{\mathrm{MoM}'} = \frac{t}{2} \cdot \mathrm{median}\Biggl\{\sum_{j\in J} \biggl(\overline{Z}^2_{t,g,1}(j) - \frac{(\mu_1(j)-\mu_2(j))^2}{2}-\frac{2G_{t}}{t}\biggr): g \in [G_t] \Biggr\}.
\]
Observe that for $g \in [G_{\tilde{t}}]$
\begin{align*}
&\sum_{j\in\mathcal{J}_{\tilde{t}, a} \cap \mathcal{H}_{\delta,a}} \biggl(\overline{Z}^2_{\tilde{t},g,1}(j) - \frac{(\mu_1(j)-\mu_2(j))^2}{2}-\frac{2G_{\tilde{t}}}{\tilde{t}}\biggr) - \sum_{j\in\mathcal{J}_{\tilde{t}, a}} V_{\tilde{t},g,a} (j) \\
&= -\frac{\sum_{j\in\mathcal{J}_{\tilde{t}, a} \cap \mathcal{H}_{\delta,a}} (\mu_1(j)-\mu_2(j))^2}{2} + \sum_{j\in\mathcal{J}_{\tilde{t}, a} \cap \mathcal{H}_{\delta,a}} V_{\tilde{t},g,a}(j)- \sum_{j\in\mathcal{J}_{\tilde{t}, a}} V_{\tilde{t},g,a}(j) \\
&\leq -\frac{\sum_{j\in\mathcal{J}_{\tilde{t}, a} \cap \mathcal{H}_{\delta,a}} (\mu_1(j)-\mu_2(j))^2}{2} + \frac{2G_{\tilde{t}}|\mathcal{J}_{\tilde{t}, a}|}{\tilde{t}}.
\end{align*}
Then, on the event $\{|\mathcal{J}_{\tilde{t},a}| \leq 2s\} \cap \Bigl\{\sum_{j\in\mathcal{J}_{\tilde{t}, a} \cap \mathcal{H}_{\delta,a}} \delta(j)^2 \geq \frac{\|\delta\|_2^2}{12\log(8/\varepsilon)}\Bigr\}$, by \Cref{Lemma:mediandifference}, we deduce 
\[
A_{\tilde{t},*, \mathcal{J}_{\tilde{t},a} \cap \mathcal{H}_{\delta,a}}^{\mathrm{MoM}'} \leq A_{\tilde{t},*, \mathcal{J}_{\tilde{t},a}}^{\mathrm{MoM}} - \frac{\|\delta\|_2^2}{12\log(8/\varepsilon)} + 2sG_{\tilde{t}},
\]
and consequently, when $\rho^2 \geq 192C_2s\Delta\log(8/\varepsilon)$, we have, with $C_2 \geq 2$, that
\[
\frac{\|\delta\|_2^2}{24\log(8/\varepsilon)} \geq C_2s\Delta \geq \max_{t \in \mathcal{T}\setminus\{1\}} \{r_t + 2sG_t\},
\]
where the first inequality is due to $\|\delta\|_2^2 \geq \rho^2/8$ and the second inequality is due to the choice of $G_t = (t \wedge \Delta)/2$.
Hence
\begin{align} \label{Eq:robust_sparse_alternative_all}
&\mathbb{E}_\theta(1-\phi_{\mathcal{P},\mathrm{sparse}}^{\mathrm{MoM}}) \leq \mathbb{P}_\theta\bigl(A_{\tilde{t},a}^{\mathrm{MoM}} \leq r_{\tilde{t}}\bigr) = \mathbb{P}_\theta\bigl(A_{\tilde{t},*, \mathcal{J}_{\tilde{t},a}}^{\mathrm{MoM}} \leq r_{\tilde{t}}\bigr) \nonumber \\
&\leq \mathbb{P}_\theta \bigl(|\mathcal{J}_{\tilde{t},a}| > 2s\bigr) + \mathbb{P}_\theta \biggl( \sum_{j\in\mathcal{J}_{\tilde{t}, a} \cap \mathcal{H}_{\delta,a}} \delta(j)^2 <\frac{\|\delta\|_2^2}{12\log(8/\varepsilon)} \biggr) + \mathbb{P}_\theta \biggl( A^{\mathrm{MoM}'}_{\tilde{t},*, \mathcal{J}_{\tilde{t},a} \cap \mathcal{H}_{\delta,a}} \leq  - \frac{\|\delta\|_2^2}{24\log(8/\varepsilon)} \biggr).
\end{align}
We control the three terms respectively. The arguments below mirror those made in the proof of \Cref{thm:weibullupperbound_sparse} between~\eqref{Eq:concentration_binomial_alternative} and~\eqref{Eq:subweibull_sparse_chebyshev} and we will omit details in places where the same reasoning is used in the last proof. First, it remains true that
\begin{equation*} \label{Eq:robust_sparse_alternative_term1}
    \mathbb{P}_\theta \bigl(|\mathcal{J}_{\tilde{t},a}| > 2s\bigr) \leq \varepsilon/8.
\end{equation*}
By~\eqref{Eq:robust_prob_coordinate_null} and the choice $a \geq \{3K(512\log(8/\varepsilon))^{1/\alpha}\} \vee \{2\alpha e^{\alpha/2}\log^{1/2}(700\log(8/\varepsilon))\}$, we have for all $j \in \mathcal{H}_{\delta, a}$ that
\begin{align*}
    \mathbb{P}_\theta(j \notin \mathcal{J}_{\tilde{t},a}) = \mathbb{P}_\theta(|Y_{\tilde{t}, 2}(j)| < a) &\leq \frac{K^\alpha}{\bigl((|\delta(j)|-a)/3\bigr)^\alpha} + \exp \biggl\{1 -\frac{(|\delta(j)|-a)^2}{2\alpha^2e^\alpha} \biggr\}\leq \frac{1}{256\log(8/\varepsilon)},
\end{align*}
and thus again 
\begin{equation*}
   \sum_{j\in\mathcal{H}_{\delta, a}} \mathrm{Var}_\theta\bigl( \delta(j)^2\mathbbm{1}_{ \{ j \in \mathcal{J}_{\tilde{t},a} \} }  \bigr) \leq \frac{\|\delta\|_2^4}{256\log(8/\varepsilon)}.
\end{equation*}
At this point, we consider
\begin{align*}
\rho^2 \geq C_3\max\Bigl\{192C_2s\Delta\log^2(8/\varepsilon), 64a^2s, 3456K^2(16/\varepsilon)^{2/\alpha}\log(8/\varepsilon),768\alpha^2e^\alpha\log^2(16e/\varepsilon)\Bigr\},
\end{align*}
with some $C_3 \geq 1$. Then, by repeating the argument in~\eqref{Eq:signal_captured},~\eqref{Eq:concentration_bernoulli} and~\eqref{Eq:concentration_bernoulli_single}, we obtain
\[
\mathbb{P}_\theta \biggl( \sum_{j\in\mathcal{J}_{\tilde{t}, a} \cap \mathcal{H}_{\delta,a}} \delta(j)^2 <\frac{\|\delta\|_2^2}{12\log(8/\varepsilon)} \biggr) \leq \varepsilon/8.
\]
We now bound the third and final term in~\eqref{Eq:robust_sparse_alternative_all}. By Chebyshev's inequality, we deduce that for $g \in [G_t]$
\begin{align*}
    &\mathbb{P}\Biggl( \frac{\tilde{t}}{2} \sum_{j\in\mathcal{J}_{\tilde{t}, a} \cap \mathcal{H}_{\delta,a}} \biggl(\overline{Z}^2_{\tilde{t},g,1}(j) - \frac{(\mu_1(j)-\mu_2(j))^2}{2}-\frac{2G_{\tilde{t}}}{\tilde{t}} \biggr) \leq - \frac{\|\delta\|_2^2}{24\log(8/\varepsilon)}  \Biggr)  \\
    &\leq \frac{\sum_{j\in\mathcal{H}_{\delta, a}} \mathrm{Var}_\theta \Bigl(\overline{Z}_{\tilde{t},g,1}^2(j)\Bigr) }{\|\delta\|_2^4/(144\tilde{t}^2\log^2(8/\varepsilon))} \leq \frac{\sum_{j\in\mathcal{H}_{\delta, a}} 2\tilde{t}^2\mathrm{Var}_\theta \Bigl(\bigl(\overline{Z}'_{\tilde{t},g,1}(j)\bigr)^2\Bigr) + \sum_{j\in\mathcal{H}_{\delta, a}} 16\tilde{t}\delta(j)^2 \mathrm{Var}_\theta \Bigl(\overline{Z}'_{\tilde{t},g,1}(j)\Bigr)}{\|\delta\|_2^4/(144\log^2(8/\varepsilon))} \\
    &\leq C_4\log^2(8/\varepsilon) \biggl(\frac{s\Delta^2}{\rho^4} + \frac{\Delta}{\rho^2} \biggr) \leq \varepsilon/48,
\end{align*}
when $C_3 \geq 1$ is sufficiently large. The third inequality above follows from~\eqref{Eq:robust_fourthmoment} and $C_4$ is a constant depending only on $\alpha$ and $K$. If $\tilde{t} = 2$, then $G_{\tilde{t}} =1$ and we immediately have
\[
\mathbb{P}_\theta \biggl( A^{\mathrm{MoM}'}_{\tilde{t},*, \mathcal{J}_{\tilde{t},a} \cap \mathcal{H}_{\delta,a}} \leq  - \frac{\|\delta\|_2^2}{24\log(8/\varepsilon)} \biggr) \leq \varepsilon/48.
\]
If $\tilde{t} > 2$, then $G_{\tilde{t}} \geq 2$ and we again use the binomial tail bound argument as in~\eqref{Eq:binomial_bound_robust} to obtain
\[
\mathbb{P}_\theta \biggl( A^{\mathrm{MoM}'}_{\tilde{t},*, \mathcal{J}_{\tilde{t},a} \cap \mathcal{H}_{\delta,a}} \leq  - \frac{\|\delta\|_2^2}{24\log(8/\varepsilon)} \biggr) \leq \exp\biggl\{-\frac{\varepsilon G_{\tilde{t}}}{48} \biggl( \frac{24}{\varepsilon }\log\Bigl(\frac{24}{\varepsilon}\Bigr) - \frac{24}{\varepsilon} + 1\biggr)  \biggr\} \leq \frac{\varepsilon}{8}.
\]
By~\eqref{Eq:robust_sparse_alternative_all}, we conclude $\mathbb{E}_\theta(1-\phi_{\mathcal{P},\mathrm{sparse}}^{\mathrm{MoM}}) \leq \varepsilon/2$. Finally, for the case that the mean change happens at $t_0=1$ instead, similar to the last paragraph of the proof of Theorem~\ref{thm:weibullupperbound_sparse}, we can still control the three terms in~\eqref{Eq:robust_sparse_alternative_all} in the same way respectively when we redefine $\mathcal{J}_{\tilde{t}=1, a} := \{j \in [p]: |Y_1(j)|\geq a\}$ instead.

\subsubsection{Proof of Theorem~\ref{thm:finitemoment_upperbound_sparse_improve}}\label{subsubsection:proof_rsm}
We actually prove a more general result. Any mean estimator that satisfies the following condition can be used in place of $\hat{\mu}_{n,s,\eta}^{\mathrm{RSM}}(\cdot) = \hat{\mu}_{n,s}^{\mathrm{RSM}}(\cdot; \eta)$ introduced in Section~\ref{subSec:robust_sparse} while \Cref{thm:finitemoment_upperbound_sparse_improve} still holds. 

\begin{condition} \label{Con:RSM}
Assume $\alpha \geq 4$. Let $W_1, \dotsc, W_n$ be independent random vectors in $\mathbb{R}^p$, each with mean $\mu_W$ and covariance matrix $I_p$. Assume $\|\mu_W\|_0 \leq s$ and $\mathbb{E}|W_i(j)-\mu_W(j)|^\alpha \leq (\sqrt{2}K)^\alpha$ for $i \in [n]$ and $j \in [p]$. Then there exist constants $C_1, C_2 \geq 1$, depending only on $\alpha$ and $K$ such that for any given $0 < \eta < 1$, when $n \geq C_1\bigl(s\log(ep/s) + \log(1/\eta)\bigr)$, then with probability at least $1-\eta$, we have
\[
\bigl\|\hat{\mu}_{n,s}^{\mathrm{RSM}}(W_1, \dotsc, W_n; \eta) - \mu_W\bigr\|_2 \leq \sqrt{C_2}\biggl(\sqrt{\frac{s\log(ep/s)}{n}} + \sqrt{\frac{\log(1/\eta)}{n}}\biggr).
\]
\end{condition}

\noindent In particular, the robust sparse mean estimator that we use from \citet{prasad2019unified} satisfies the condition above as shown in Corollary 11\footnote{Note that their result is under the assumption that for each vector $v$ with $\|v\|_2 = 1$, $\mathbb{E}(v^\top(W-\mu_W))^\alpha \leq C [\mathbb{E}(v^\top(W-\mu_W))^2]^{\alpha/2}$ for some absolute constant $C$, which is certainly satisfied by our assumption $\mathbb{E}|W(j) - \mu_W(j)|^\alpha \leq (\sqrt{2}K)^\alpha$ for $j \in [p]$ in \Cref{Con:RSM} with $C = (\sqrt{2}K)^\alpha$.(Rio09-(1.2))} therein.

In the rest of the proof, we denote $\tilde{\mathcal{T}}_1 := \{t \in \mathcal{T}: t \leq \tilde{\Delta}_1\}$, $\tilde{\mathcal{T}}_2 := \{t \in \mathcal{T}: \tilde{\Delta}_1 < t \leq \tilde{\Delta}_2\}$ and $\tilde{\mathcal{T}}_3 := \{t \in \mathcal{T}: t > \tilde{\Delta}_2\}$ and recall that $\mathcal{J}_{t,a} = \{j \in [p]: |Y_{t,2}(j)| \geq a\}$ for $t \in \mathcal{T}\setminus\{1\}$.

\medskip

\noindent \textbf{Null term.} For $\theta \in \Theta_0(p,n)$, we have
\begin{align} \label{Eq:robust_improve_null_main}
\mathbb{E}_\theta\phi^{\mathrm{RSM}}_{\mathcal{P},\mathrm{sparse}} &= \mathbb{P}_\theta(A_{1,a} > \tilde{r}_1) 
 + \sum_{t \in \tilde{\mathcal{T}}_1\backslash\{1\}} \mathbb{P}_\theta(A_{t,a} > \tilde{r}_t) \nonumber\\
 &\quad + \sum_{t \in \tilde{\mathcal{T}}_2} \mathbb{P}_\theta(A^{\mathrm{RSM}}_{t} > r^{\mathrm{RSM}}_t) + \sum_{t \in \tilde{\mathcal{T}}_3} \mathbb{P}_\theta(A^{\mathrm{RSM}}_{t} > r^{\mathrm{RSM}}_t).
\end{align}
For the first term, similar to the proof of Theorem~\ref{thm:weibullupperbound_sparse}, by Proposition~\ref{prop:2sample}(b), when $\tilde{r}_1 \geq C'_4 s(p/s)^{2/\alpha}$, for some large enough $C'_4 > 0$, depending only on $\alpha$, $K$ and $\varepsilon$, we have $\mathbb{P}_\theta(A_{1,a} > \tilde{r}_1) \leq \varepsilon/8$.
To control the second term in~\eqref{Eq:robust_improve_null_main}, we closely follow the arguments in the null term part of the proof of \Cref{thm:weibullupperbound_sparse} and \Cref{thm:finitemoment_upperbound_sparse}. By~\eqref{eq:sparse_weibullnull}, we have
\begin{equation} \label{Eq:robust_improve_null_term1}
\sum_{t \in \tilde{\mathcal{T}}_1\backslash\{1\}} \mathbb{P}_\theta(A_{t,a} > \tilde{r}_t)  \leq \sum_{t \in \tilde{\mathcal{T}}_1\backslash\{1\}}\mathbb{P}_\theta (|\mathcal{J}_{t,a}| > s) + \sum_{t \in \tilde{\mathcal{T}}_1\backslash\{1\}} \sup_{J \subseteq [p]: |J| \leq s} \mathbb{P}_\theta\Biggl(\sum_{j\in J} \bigl(Y_{t,1}^2(j)-1\bigr) > \tilde{r}_t\Biggr).
\end{equation}
For the first term on the right hand side, by~\eqref{Eq:chernoff-hoeffding-robust}, we obtain
\[
    \sum_{t \in \tilde{\mathcal{T}}_1\backslash\{1\}}\mathbb{P}_\theta (|\mathcal{J}_{t,a}| > s)  \leq \biggl(\frac{2epK}{s(a/3)^\alpha}\biggr)^s + \log_2(\tilde{\Delta}_1)\biggl(\frac{2e^2p}{s}\biggr)^s \exp\biggl\{-\frac{sa^2}{2\alpha^2e^\alpha}\biggr\}.
\]
The choice of $a$ in~\eqref{Eq:robust_theorem_parameters_sparse_rsm} with a large enough constant $C_3 > 0$ guarantees that
$\sum_{t \in \tilde{\mathcal{T}}_1\backslash\{1\}}\mathbb{P}_\theta (|\mathcal{J}_{t,a}| > s) \leq \varepsilon/16$. For the second term, we fix $J \subseteq [p]$ with $|J| \leq s$. By the same technique as in~\eqref{Eq:MoM_robust_null}, we obtain
\begin{align*}
\mathbb{P}_\theta\Biggl(\sum_{j\in J} \bigl(Y_{t,1}^2(j)-1\bigr) > \tilde{r}_t\Biggr) \leq 
\frac{\varepsilon}{16\log_2(\tilde{\Delta}_1)},
\end{align*}
when $\tilde{r}_t = C_4\sqrt{s\log \tilde{\Delta}_1 }$, for some large enough $C_4 > 0$, depending only on $\alpha$, $K$ and $\varepsilon$. We thus deduce that $\sum_{t \in \tilde{\mathcal{T}}_1\backslash\{1\}} \mathbb{P}_\theta(A_{t,a} > \tilde{r}_t) \leq \varepsilon/8$.

Now, we control the third and fourth terms in~\eqref{Eq:robust_improve_null_main}. For $t \in \tilde{\mathcal{T}}_2 \cup \tilde{\mathcal{T}}_3$, we observe that
\[
C_1\bigl(s\log(ep/s) + \log(1/\eta_t)\bigr) = \min(t, \tilde{\Delta}_2) \leq t.
\]
Since $Z_1, \dotsc, Z_t$ are independent and identically distributed random vectors with mean $0$ and covariance matrix $I_p$ and satisfy $\mathbb{E}|Z_i(j)|^\alpha \leq 2^{\alpha/2}K^\alpha$ for $i \in [t], j \in [p]$ under the null, by \Cref{Con:RSM}, we obtain
\[
\mathbb{P}_\theta(A^{\mathrm{RSM}}_{t} > r^{\mathrm{RSM}}_t) = \mathbb{P}_\theta\Bigl(t\bigl\|\hat{\mu}_{t,s,\eta_t}^{\mathrm{RSM}}\bigr\|^2_2 > 2C_2\bigl(s\log(ep/s) + \log(1/\eta_t)\bigr)\Bigr) \leq \eta_t,
\]
and therefore,
\begin{align} \label{Eq:robust_improve_null_term23}
\sum_{t \in \tilde{\mathcal{T}}_2} \mathbb{P}_\theta(A^{\mathrm{RSM}}_{t} > r^{\mathrm{RSM}}_t) + \sum_{t \in \tilde{\mathcal{T}}_3} \mathbb{P}_\theta(A^{\mathrm{RSM}}_{t} > r^{\mathrm{RSM}}_t) &\leq \sum_{t \in \tilde{\mathcal{T}}_2} \exp\biggl\{s\log(ep/s) - \frac{t}{C_1}\biggr\} + \sum_{t \in \tilde{\mathcal{T}}_3} \frac{\varepsilon}{16\log 2n} \nonumber \\
&\leq 2\exp\biggl\{s\log(ep/s) - \frac{\tilde{\Delta}_1}{C_1}\biggr\} + \frac{\varepsilon\log_2(n/2)}{16\log 2n} < \varepsilon/4,
\end{align}
where we use \Cref{lemma:exp_decay_sum} in the second inequality. 
Hence, we conclude that $\mathbb{E}_\theta\phi^{\mathrm{RSM}}_{\mathcal{P},\mathrm{sparse}} \leq \varepsilon/2$ for all $\theta \in \Theta_0(p,n)$.

\medskip
\noindent \textbf{Alternative term.} As in all previous proofs of alternative term, we consider the unique $\tilde{t} \in \mathcal{T}$, such that $t_0/2 \leq \tilde{t} \leq t_0$, where $t_0 (\leq n/2)$ is the true change point location. When $t_0 = 1$, we simply use the final paragraph of the proof of \Cref{thm:finitemoment_upperbound_sparse}. When $t_0 \geq 2$, we consider separately the two cases $\tilde{t} \in \tilde{\mathcal{T}}_1\backslash\{1\}$ and $\tilde{t} \in \tilde{\mathcal{T}}_2 \cup \tilde{\mathcal{T}}_3$. When $\tilde{t} \in \tilde{\mathcal{T}}_1\backslash\{1\}$, the arguments are again almost the same as those used in the alternative term part of the proof of \Cref{thm:finitemoment_upperbound_sparse}. We thus omit the details and directly state the conclusion: as long as
\[
\rho^2 \geq C_6 \max \bigl\{(\tilde{r} \mathbbm{1}_{\{t \neq 1\}}+2s)\log^2(8/\varepsilon), a^2s, (1/\varepsilon)^{2/\alpha}\log(8/\varepsilon) \bigr\},
\]
for some large enough $C_6 > 0$, depending only on $\alpha$ and $K$, we have $\mathbb{E}_\theta(1-\phi^{\mathrm{RSM}}_{\mathcal{P},\mathrm{sparse}}) \leq \varepsilon/2$. Note that if $\rho^2 \geq C_5 v_{\mathcal{P},\mathrm{sparse}}^{\mathrm{U}}$, for some large enough $C_5 > 0$, depending only on $\alpha$, $K$ and $\varepsilon$, then the above condition is satisfied.

If $\tilde{t} \in \tilde{\mathcal{T}}_2 \cup \tilde{\mathcal{T}}_3$ instead, then $Z_1, \dotsc, Z_t$ are independent and identically distributed random vectors with mean $(\mu_1-\mu_2)/\sqrt{2}$ and covariance matrix $I_p$ and satisfy $\mathbb{E}\bigl|Z_i(j) - \frac{\mu_1(j)-\mu_2(j)}{\sqrt{2}}\bigr|^\alpha \leq 2^{\alpha/2}K$ for $i \in [t], j \in [p]$. Recall that $\tilde{t}\|\mu_1-\mu_2\|^2_2 \geq \rho^2/2$. Hence, when $\rho^2 \geq 24C_2\bigl(s\log(ep/s)+\log(16\log(2n)/\varepsilon)\bigr)$, we have by \Cref{Con:RSM} that
\begin{align*}
&\mathbb{E}_\theta(1-\phi_{\mathcal{P},\mathrm{sparse}}^{\mathrm{RSM}}) = \mathbb{P}_\theta(A^{\mathrm{RSM}}_{\tilde{t},a} \leq r^{\mathrm{RSM}}_{\tilde{t}}) = \mathbb{P}_\theta\Bigl(\tilde{t}\bigl\|\hat{\mu}_{\tilde{t},s,\eta_{\tilde{t}}}^{\mathrm{RSM}} \bigr\|^2_2 \leq 2C_2\bigl(s\log(ep/s) + \log(1/\eta_{\tilde{t}})\bigr)\Bigr) \\
&\leq \mathbb{P}_\theta \biggl(  \sqrt{\tilde{t}} \biggl|\Bigl\| \frac{\mu_1-\mu_2}{\sqrt{2}} \Bigr\|_2 -  \Bigl\|\hat{\mu}_{\tilde{t},s,\eta_{\tilde{t}}}^{\mathrm{RSM}} - \frac{\mu_1-\mu_2}{\sqrt{2}} \Bigr\|_2  \biggr| \leq \sqrt{2C_2}\bigl(\sqrt{s\log(ep/s)} + \sqrt{\log(1/\eta_{\tilde{t}})}\bigr)\biggr) \\
&\leq \mathbb{P}_\theta \biggl( \sqrt{\tilde{t}} \Bigl\|\hat{\mu}_{\tilde{t},s,\eta_{\tilde{t}}}^{\mathrm{RSM}} - \frac{\mu_1-\mu_2}{\sqrt{2}} \Bigr\|_2 > \sqrt{C_2}\bigl(\sqrt{s\log(ep/s)} + \sqrt{\log(1/\eta_{\tilde{t}})}\bigr) \biggr) \leq \eta_{\tilde{t}} \\
&\leq \exp\biggl\{s\log(ep/s) - \frac{\min(t, \tilde{\Delta}_2)}{C_1}\biggr\} \leq \frac{\varepsilon}{16},
\end{align*}
as desired.

\subsubsection{Proof of Corollary~\ref{thm:finitemoment_combinerate}}
We first consider the statistical property of $\phi_{\mathcal{P},\mathrm{sparse}}$. By comparing the two rates
$v^{\mathrm{U, MoM}}_{\mathcal{P},\mathrm{sparse}}$ and $v^{\mathrm{U}}_{\mathcal{P},\mathrm{sparse}}$, we note that the improvement offered by $\phi_{\mathcal{P},\mathrm{sparse}}^{\mathrm{RSM}}$ over $\phi_{\mathcal{P},\mathrm{sparse}}^{\mathrm{MoM}}$ only exists when
\begin{equation*}
    (p/s)^{2/\alpha} < \log\log(8n),  
\end{equation*}
since otherwise $v^{\mathrm{U, MoM}}_{\mathcal{P},\mathrm{sparse}} = v^{\mathrm{U}}_{\mathcal{P},\mathrm{sparse}}$. Combining this with the fact that we are in the sparse regime $s < p^{(\alpha-4)/(2\alpha-4)}$, we deduce that $p < \log^{\alpha-2}(\log(8n))$. The desired result is then an immediate consequence of \Cref{thm:finitemoment_upperbound_sparse_improve} and \Cref{thm:finitemoment_upperbound_sparse}.

Now onto the computational complexity claim. For each $t \in \mathcal{T}$, computing the statistics $A^{\mathrm{MoM}}_{t,a}$ and $A_{t,a}$ in $\phi_{\mathcal{P},\mathrm{sparse}}^{\mathrm{MoM}}$ and $\phi_{\mathcal{P},\mathrm{sparse}}^{\mathrm{RSM}}$ take time polynomial in $n$ and $p$ since they only involve performing basic operations and finding the median of $G_t \leq 8\log\log(8n) $ quantities. The computationally demanding part lies in computing $A^{\mathrm{RSM}}_{t}$, or equivalently the robust sparse mean estimator $\hat{\mu}^{\mathrm{RSM}}_{t,s,\eta_t}$. Note that we are using this only when $p < \log^{\alpha-2}(\log(8n))$. For each fixed $t$, we claim that the computation/approximation of $\hat{\mu}^{\mathrm{RSM}}_{t,s,\eta_t}$ can be performed in time that is polynomial in $n$. We now show this by arguing that each component below has time complexity that is polynomial in $n$. In the rest of the proof, we omit the subscripts and adopt the notation $\hat{\mu}^{\mathrm{RSM}}$ for clarity.
\begin{enumerate}
    \item Each evaluation of the function $\mathrm{1DRobust}(\cdot)$ \citep[cf.][Algorithm 2]{prasad2019unified} of $t$ data point requires time of order $t\log t \leq n\log n$ (in order to find the shortest interval).
    \item The total number of projection $|\mathcal{N}^{1/2}_{2s}(\mathcal{S}^{p-1})|$ can be bounded by $|\mathcal{N}^{1/2}_{2s}(\mathcal{S}^{p-1})| \leq (6ep/s)^s \leq (6ep)^p \leq \exp(6ep^2) \leq \exp(C_\alpha\log(n)) = n^{C_{\alpha}}$ for some constant $C_{\alpha} >0$, depending only on $\alpha$. Denote
    \[
    g(\mu) := \max_{u \in \mathcal{N}^{1/2}_{2s}(\mathcal{S}^{p-1})} |u^{\top}\mu - \mathrm{1DRobust}(\{u^{\top}Z_i\}_{i=1}^t,\eta_t/ (6ep/s)^s)|.
    \]
    Thus for a fixed $\mu \in \mathbb{R}^p$, the computational complexity of evaluating $g(\mu)$ is polynomial in $n$.
    \item The optimisation problem defining $\hat{\mu}^{\mathrm{RSM}}$ can be written as
    \[
    \min_{\mu \in \mathcal{L}_s} g(\mu).
    \]
    We solve this by first considering each possible $s$-sparsity coordinate pattern individually before working out the minimum among these $\binom{p}{s} \leq n^{C_{\alpha}}$ minima.
    \item Fix $\mathcal{U} \subseteq \mathbb{R}^p$ with $|\mathcal{U}|=s$. We solve the optimisation problem 
    \[
    \min_{\mu \in \mathbb{R}^p: \mu(j)=0 \, \forall j\in \mathcal{U}^c} g(\mu)
    \]
    by subgradient descent. Denote the optimal value to be $g_{*,\mathcal{U}}$ and the $k$-th iterate to be $\mu_{\mathcal{U}}^{(k)}$. Note that $g(\mu)$ is $1$-Lipschitz and $\partial g(\mu) \subseteq \bigl\{\pm u: u \in \mathcal{N}^{1/2}_{2s}(\mathcal{S}^{p-1})\bigr\}$. Standard result on the convergence of subgradient descent \citep[e.g.][Theorem 3.2.2]{nesterov2003introductory} shows that $\bigl(\min_{k \in [K]} g\bigl(\mu_{\mathcal{U}}^{(k)}\bigr)\bigr) - g_{*,\mathcal{U}} \leq \upsilon$ in $K \asymp 1/\upsilon^2$ steps, where we choose $\upsilon = \sqrt{s\log(ep/s)t^{-1}}$. The computational complexity is again at most polynomial in $n$. Denote $\tilde{\mu}^{\mathrm{RSM}}_{\mathcal{U}}$ to be the update that attains the best objective value in $K$ iterations.
\end{enumerate} 
Write
\[
\tilde{\mu}^{\mathrm{RSM}} := \argmin_{\mu \in \bigl\{\tilde{\mu}^{\mathrm{RSM}}_{\mathcal{U}}: \,|\mathcal{U}|=s\bigr\}} g(\mu),
\]
as our final estimator (an approximation of $\hat{\mu}^{\mathrm{RSM}}$). We have now shown that $\tilde{\mu}^{\mathrm{RSM}}$ can be obtained in time that is polynomial in $n$. Finally, we prove that $\tilde{\mu}^{\mathrm{RSM}}$ still satisfies Condition~\ref{Con:RSM}. Indeed, following the proof of Lemma~4 and Corollary~12 in \citet{prasad2019unified}, we have
\begin{align*}
    \|\tilde{\mu}^{\mathrm{RSM}} - \mu_Z\|_2 &\leq \|\tilde{\mu}^{\mathrm{RSM}} - \hat{\mu}^{\mathrm{RSM}}\|_2 + \|\hat{\mu}^{\mathrm{RSM}} - \mu_Z\|_2 \leq g\bigl(\tilde{\mu}^{\mathrm{RSM}}\bigr) + g\bigl(\hat{\mu}^{\mathrm{RSM}} \bigr) + g\bigl(\hat{\mu}^{\mathrm{RSM}} \bigr) + g(\mu_Z) \\
    &\leq g\bigl(\hat{\mu}^{\mathrm{RSM}} \bigr) + \upsilon + 2g\bigl(\hat{\mu}^{\mathrm{RSM}} \bigr) + g(\mu_Z) \leq \upsilon + 4g(\mu_Z) \\
    &\leq \sqrt{C}\biggl(\sqrt{\frac{s\log(ep/s)}{t}} + \sqrt{\frac{\log(1/\eta_t)}{t}}\biggr),
\end{align*}
for some $C \geq 1$, where $\mu_Z = \mathbb{E}Z_1$.

\subsection{Proof of the adaptation result in Section~\ref{sec:adaptsparsity}}
\subsubsection{Adaptive testing procedure} \label{sec:adapt-test}
We focus on the case when $P_e \in \mathcal{P}_{\alpha, K}^{\otimes}$ for $\alpha > 4$. We introduce an adaptive testing procedure based on these two tests:
\begin{equation} \label{Eq:adaptive_test}
\begin{aligned} 
\phi_{\mathcal{P}, \mathrm{adaptive}} &:= \phi_{\mathcal{P},\mathrm{dense}} \vee \max_{s \in \mathcal{K}} {\phi}_{\mathcal{P}, \mathrm{sparse}, s}  \\
&\,= \begin{cases}
    \phi_{\mathcal{P},\mathrm{dense}} \vee \max_{s \in \mathcal{K}} {\phi}^{\mathrm{MoM}}_{\mathcal{P}, \mathrm{sparse}, s}, &\text{if} \; p > \log^{\alpha-2}(\log(8n)), \\
    \phi_{\mathcal{P},\mathrm{dense}} \vee \max_{s \in \mathcal{K}} {\phi}^{\mathrm{RSM}}_{\mathcal{P}, \mathrm{sparse}, s}, &\text{if} \; p \leq \log^{\alpha-2}(\log(8n)),
\end{cases}
\end{aligned}
\end{equation}
\sloppy where the dependence of $\phi_{\mathcal{P},\mathrm{sparse}}$ on $s$ is made explicit by writing it as ${\phi}_{\mathcal{P}, \mathrm{sparse}, s}$, and the set $\mathcal{K} := \bigl\{1, 2, 4, \dotsc, 2^{\lceil \log_2(p) \rceil - 1}\bigr\}$ is a dyadic grid. Recall that $\phi_{\mathcal{P},\mathrm{dense}}$ does not require the knowledge of~$s$, and we keep its original parameter choices as in~\eqref{Eq:robust_theorem_parameters}, with a potentially larger constant $C_1$ in~$r_t$: 
\begin{equation} \label{Eq:adaptive_param_1}
    r_t =  C_1p^{(1/2)\vee(2/\alpha)}G_t, \quad G_t = t \wedge \Delta \quad \mbox{and} \quad \Delta = 2^{ 3 + \lceil\log_2 \log \log(8n)\rceil}.
\end{equation}
For $\phi_{\mathcal{P},\mathrm{sparse}, s}^{\mathrm{MoM}}$, we modify the original parameter choices~\eqref{Eq:robust_theorem_parameters_sparse} as follows:
\begin{equation}\label{Eq:adaptive_param_2}
\begin{aligned} 
    a_s = C_2\bigl((p/s)^{1/\alpha} + s^{-1/2} \log^{1/2}(\log(8n))\bigr), &\quad r_{t,s} =  C_3\bigl(s(p/s)^{2/\alpha}\mathbbm{1}_{\{t=1\}} + s^{3/4}G_t\mathbbm{1}_{\{t>1\}}\bigr),  \\
 \quad  G_t = (t \wedge \Delta)/2 \quad &\mbox{and} \quad  \Delta = 2^{ 4 + \lceil\log_2 \log \log(8n)\rceil}.
\end{aligned}
\end{equation}
Comparing with~\eqref{Eq:robust_theorem_parameters_sparse}, we use the same $a_s$ (but potentially larger constants) and modify $r_{t,s}$. Finally, for $\phi_{\mathcal{P},\mathrm{sparse}, s}^{\mathrm{RSM}}$, we modify its original parameter choices \eqref{Eq:robust_theorem_parameters_sparse_rsm} to be:
\begin{equation}\label{Eq:adaptive_param_3}
\begin{aligned}
a_s = C_4\bigl((p/s)^{1/\alpha} + s^{-1/2} \log^{1/2}(\log \tilde{\Delta}_{1,s})\bigr), &\quad \tilde{r}_{t,s}  = C_5 \Bigl( s(p/s)^{2/\alpha}\mathbbm{1}_{\{t=1\}}+ s^{3/4}\sqrt{\log \tilde{\Delta}_1}\mathbbm{1}_{\{t>1\}}\Bigr), \\
\eta_{t,s} = \exp\biggl\{s\log(ep/s) - \frac{t\wedge \tilde{\Delta}_2}{C_6}\biggr\}, &\quad r^{\mathrm{RSM}}_{t,s} = C_7(t \wedge \tilde{\Delta}_{2,s}), \\
\tilde{\Delta}_{1,s} =  C_6\bigl(s\log(ep/s) + \log(80s/\varepsilon)\bigr)\quad &\mbox{and} \quad \tilde{\Delta}_{2,s} = C_6\bigl(s\log(ep/s) + \log(80s\log(2n)/\varepsilon)\bigr).
\end{aligned}
\end{equation}

\subsubsection{Proof of Theorem~\ref{thm:adaptive_upperbound}}\label{proof:adaptation}
We prove the following theorem on the theoretical guarantee of the test $\phi_{\mathcal{P}, \mathrm{adaptive}}$, fully constructed and specified above in \Cref{sec:adapt-test}. Theorem~\ref{thm:adaptive_upperbound} then follows as an immediate consequence.

\begin{thm} \label{thm:adaptive_upperbound_appendix}
Assume $\alpha \geq 4$. For any $\varepsilon \in (0,1)$, there exist $C_1, \dotsc, C_8 > 0$ depending only on~$\alpha$,~$K$ and $\varepsilon$, such that the test $\phi_{\mathcal{P}, \mathrm{adaptive}}$ defined in~\eqref{Eq:adaptive_test} with its parameters specified in~\eqref{Eq:adaptive_param_1},~\eqref{Eq:adaptive_param_2} and~\eqref{Eq:adaptive_param_3} satisfies 
\[
\mathcal{R}_{\mathcal{P}}(\rho,  \phi_{\mathcal{P},\mathrm{adaptive}} ) \leq \varepsilon,
\]
as long as $\rho^2 \geq C_8\bigl(v_{\mathcal{P}, \mathrm{dense}}^{\mathrm{U}} \wedge v_{\mathcal{P}, \mathrm{sparse}}^{\mathrm{U}}\bigr)$.
\end{thm}

\begin{proof}
This proof is based on the proofs of \Cref{thm:finitemoment_upperbound_dense},
\Cref{thm:finitemoment_upperbound_sparse_improve} and
\Cref{thm:finitemoment_upperbound_sparse}. For brevity, we only highlight the main steps and differences.

\noindent \textbf{Null term.} By a union bound,~\eqref{Eq:sparse_MoMnull_master} and~\eqref{Eq:robust_improve_null_main}, we have
\begin{align} \label{Eq:adaptive_null_master}
&\mathbb{E}_\theta \phi_{\mathcal{P}, \mathrm{adaptive}} \nonumber \\
&\leq \mathbb{E}_\theta \phi_{\mathcal{P}, \mathrm{dense}} + \mathbb{E}_\theta \biggl[\max_{s \in \mathcal{K}} \phi_{\mathcal{P},\mathrm{sparse},s}^{\mathrm{MoM}} \biggr] \mathbbm{1}_{\{ p > \log^{\alpha-2}(\log(8n)) \}} + \mathbb{E}_\theta \biggl[\max_{s \in \mathcal{K}} \phi_{\mathcal{P},\mathrm{sparse},s}^{\mathrm{RSM}} \biggr]\mathbbm{1}_{\{ p \leq \log^{\alpha-2}(\log(8n)) \}} \nonumber\\
&\leq \mathbb{E}_\theta \phi_{\mathcal{P}, \mathrm{dense}} + \Biggl( \mathbb{P}_\theta\biggl(\max_{s \in \mathcal{K}} 
\frac{A_{1,a_s,s}}{r_{1,s}} > 1\biggr) + \sum_{s \in \mathcal{K}} \sum_{t\in \mathcal{T}\setminus\{1\}} \mathbb{P}_\theta\bigl(A_{t,a_s,s}^{\mathrm{MoM}} > r_{t,s}\bigr) \Biggr) \mathbbm{1}_{\{ p > \log^{\alpha-2}(\log(8n)) \}} \nonumber\\
&\quad + \Biggl( \mathbb{P}_\theta\biggl(\max_{s \in \mathcal{K}} \frac{A_{1,a_s,s}}{\tilde{r}_{1,s}} > 1\biggr) + \sum_{s \in \mathcal{K}} \sum_{t \in \tilde{\mathcal{T}}_{1,s}\backslash\{1\}} \mathbb{P}_\theta(A_{t,a_s,s} > \tilde{r}_{t,s}) \nonumber\\
&\qquad \qquad + \sum_{s \in \mathcal{K}} \sum_{t \in \tilde{\mathcal{T}}_{2,s}} \mathbb{P}_\theta(A^{\mathrm{RSM}}_{t,s} > r^{\mathrm{RSM}}_{t,s}) + \sum_{s \in \mathcal{K}} \sum_{t \in \tilde{\mathcal{T}}_{3,s}} \mathbb{P}_\theta(A^{\mathrm{RSM}}_{t,s} > r^{\mathrm{RSM}}_{t,s})\Biggr) \mathbbm{1}_{\{ p \leq \log^{\alpha-2}(\log(8n)) \}} \nonumber\\
&\leq \mathbb{E}_\theta \phi_{\mathcal{P}, \mathrm{dense}} +  \mathbb{P}_\theta\biggl(\max_{s \in \mathcal{K}} \frac{A_{1,a_s,s}}{r_{1,s} \wedge \tilde{r}_{1,s}} > 1\biggr) + \sum_{s \in \mathcal{K}} \sum_{t\in \mathcal{T}\setminus\{1\}} \mathbb{P}_\theta\bigl(A_{t,a_s,s}^{\mathrm{MoM}} > r_{t,s}\bigr) \nonumber \\
&\quad + \sum_{s \in \mathcal{K}}\sum_{t \in \tilde{\mathcal{T}}_{1,s}\backslash\{1\}} \mathbb{P}_\theta(A_{t,a_s,s} > \tilde{r}_{t,s}) + \sum_{s \in \mathcal{K}}\Biggl(\sum_{t \in \tilde{\mathcal{T}}_{2,s}} \mathbb{P}_\theta(A^{\mathrm{RSM}}_{t,s} > r^{\mathrm{RSM}}_{t,s}) + \sum_{t \in \tilde{\mathcal{T}}_{3,s}} \mathbb{P}_\theta(A^{\mathrm{RSM}}_{t,s} > r^{\mathrm{RSM}}_{t,s})\Biggr),
\end{align}
where we denote $\tilde{\mathcal{T}}_{1,s} := \{t \in \mathcal{T}: t \leq \tilde{\Delta}_{1,s}\}$, $\tilde{\mathcal{T}}_{2,s} := \{t \in \mathcal{T}: \tilde{\Delta}_{1,s} < t \leq \tilde{\Delta}_{2,s}\}$ and $\tilde{\mathcal{T}}_{3,s} := \{t \in \mathcal{T}: t > \tilde{\Delta}_{2,s}\}$. In the following, we bound each of the five terms in~\eqref{Eq:adaptive_null_master} by $\varepsilon/10$.\\
\noindent \textbf{Term 1.} By closely following the null term part of the proof of Theorem~\ref{thm:finitemoment_upperbound_dense}, with a sufficiently large constant $C_1$, we deduce, similar to~\eqref{Eq:MoM_robust_null_together}, that
\begin{align*}
   \mathbb{E}_\theta\phi_{\mathcal{P},\mathrm{dense}} &\leq \varepsilon/180 + \frac{(32/\varepsilon)^{-1}}{1-(32/\varepsilon)^{-1}} + \log_2(n/2)(32/\varepsilon)^{-\Delta/2} \leq \varepsilon/180 + \varepsilon/31 + \varepsilon/31  < \varepsilon/10.
\end{align*}
\textbf{Term 2.} By having $C_3$ and $C_5$ sufficiently large, by Proposition~\ref{prop:2sample}(c), we can control this term at level $\varepsilon/10$. \\
\textbf{Term 3.} For this, we follow the null term part of the proof of \Cref{thm:finitemoment_upperbound_sparse}. The key step in that proof was to bound both terms in~\eqref{Eq:sparse_robustnull}. The first term can be controlled via~\eqref{Eq:chernoff-hoeffding-robust}. A careful inspection reveals that the condition on $a$ (same as $a_s$ here) given in~\eqref{Eq:choice_a_finitemoment} with a possibly larger value of the leading constant can guarantee the control of both terms in~\eqref{Eq:chernoff-hoeffding-robust} at $\varepsilon/(160s^{1/2})$. Bounding the second term in~\eqref{Eq:sparse_robustnull} required the argument from~\eqref{Eq:MoM_robust_null} to~\eqref{Eq:MoM_robust_null_together}, within which the dimension $p$ was replaced by $s$. Our new choice of $r_{t,s} = C_3s^{3/4}G_t$ for $t > 1$ with a sufficiently large $C_3$ allows us to have $\varepsilon/(1100s^{1/2})$
as the RHS bound in~\eqref{Eq:MoM_robust_null} (dimension being $s$). Correspondingly, the RHS of~\eqref{Eq:binomial_bound_robust} now becomes
\begin{align*}
&\exp\biggl\{-\frac{\varepsilon G_t}{1100\sqrt{s}} \biggl( \frac{550\sqrt{s}}{\varepsilon }\log\Bigl(\frac{550\sqrt{s}}{\varepsilon}\Bigr) - \frac{550\sqrt{s}}{\varepsilon} + 1\biggr)  \biggr\} \leq \exp\biggl\{ -\frac{G_t}{2} \log\bigl(200\sqrt{s}/\varepsilon\bigr)\biggr\}.
\end{align*}
Thus, the second term in~\eqref{Eq:sparse_robustnull} can now be bounded instead by
\begin{align*}
&\frac{\varepsilon}{1100\sqrt{s}} + \frac{(200\sqrt{s}/\varepsilon)^{-1}}{1-(200\sqrt{s}/\varepsilon)^{-1}} + \log_2(n/2)(200\sqrt{s}/\varepsilon)^{-\Delta/4} \leq  \frac{\varepsilon}{1100\sqrt{s}} + \frac{\varepsilon}{199\sqrt{s}} + \frac{\varepsilon}{199\sqrt{s}} < \frac{\varepsilon}{80\sqrt{s}}.
\end{align*}
Putting everything together, we conclude that
\[
\sum_{s \in \mathcal{K}} \sum_{t\in \mathcal{T}\setminus\{1\}} \mathbb{P}_\theta\bigl(A_{t,a_s,s}^{\mathrm{MoM}} > r_{t,s}\bigr) \leq \sum_{s \in \mathcal{K}} \frac{\varepsilon}{40\sqrt{s}} < \frac{\varepsilon}{10}.
\]
\noindent \textbf{Term 4.} We follow the null term part of the proof of \Cref{thm:finitemoment_upperbound_sparse_improve}. More specifically, this term can be split into two terms according to~\eqref{Eq:robust_improve_null_term1}. Similar to the argument made for the second term above, with $C_4$ sufficiently large, the first term in~\eqref{Eq:robust_improve_null_term1} can be guaranteed to be at most $\varepsilon/(80s^{1/2})$. The second term, with the new choice of $\tilde{r}_{t,s}$ and its leading constant $C_5$ being sufficiently large, can also be bounded above by
\[
\frac{\varepsilon |\tilde{\mathcal{T}}_{1,s}\backslash\{1\}|}{80\sqrt{s}\log_2(\tilde{\Delta}_{1,s})} \leq \varepsilon/(80s^{1/2}).
\]
Therefore, we can again control the fourth term at level $\varepsilon/10$.\\
\noindent \textbf{Term 5.}
We again follow the null term part of the proof of \Cref{thm:finitemoment_upperbound_sparse_improve}. By Condition~\ref{Con:RSM} and similar to~\eqref{Eq:robust_improve_null_term23}, we can now bound
\begin{align*} 
&\sum_{s \in \mathcal{K}}\Biggl(\sum_{t \in \tilde{\mathcal{T}}_{2,s}} \mathbb{P}_\theta(A^{\mathrm{RSM}}_{t,s} > r^{\mathrm{RSM}}_{t,s}) + \sum_{t \in \tilde{\mathcal{T}}_{3,s}} \mathbb{P}_\theta(A^{\mathrm{RSM}}_{t,s} > r^{\mathrm{RSM}}_{t,s})\Biggr) \\
&\leq \sum_{s \in \mathcal{K}} \Biggl( \sum_{t \in \tilde{\mathcal{T}}_{2,s}} \exp\biggl\{s\log(ep/s)  - \frac{t}{C_6}\biggr\} + \sum_{t \in \tilde{\mathcal{T}}_{3,s}} \exp\biggl\{s\log(ep/s)  - \frac{\tilde{\Delta}_{2,s}}{C_6}\biggr\} \Biggr) \\
&\leq \sum_{s \in \mathcal{K}} \Biggl( 2\exp\biggl\{s\log(ep/s) - \frac{\tilde{\Delta}_{1,s}}{C_6}\biggr\} + \log_2(n/2) \frac{\varepsilon}{80s\log n} \Biggr) < \sum_{s \in \mathcal{K}}\frac{\varepsilon}{20s} \leq \frac{\varepsilon}{10},
\end{align*}
as desired.

\medskip
\noindent \textbf{Alternative term.} First, let $s_1$ satisfy $s_1(p/s_1)^{2/\alpha} + \log\log(8n) = \sqrt{p}\log\log(8n)$ and $s_2$ satisfy $s_2\bigl((p/s_2)^{2/\alpha} + \log\log(8n)\bigr) = \sqrt{p}\log \log(8n)$. Note that $s_1 \geq s_2$. For $\theta \in \Theta(p,n,s,\rho)$, we consider all four possible $(p,s)$ regimes below. \\
(1) $p \geq \log^{\alpha-2}(\log(8n))$ and $s \geq s_2/2$. We have $\mathbb{E}_{\theta}(1-\phi_{\mathcal{P}, \mathrm{adaptive}}) \leq \mathbb{E}_{\theta}(1-{\phi}_{\mathcal{P}, \mathrm{dense}})$. By the alternative term part of the proof of Theorem~\ref{thm:finitemoment_upperbound_dense}, we can bound the above quantity by $\varepsilon/2$ as long as $\rho^2 \geq C' \sqrt{p} \log\log(8n)$ with a sufficiently large $C'$. We also note that when $s_2/2 \leq s < s_2$, we have
\[
\frac{1}{2}\sqrt{p} \log \log(8n)\leq s\bigl((p/s)^{2/\alpha} + \log\log(8n)\bigr) \leq \sqrt{p} \log \log(8n).
\]
(2) $p \geq \log^{\alpha-2}(\log(8n))$ and $s < s_2/2$. By the definition of $\mathcal{K}$, there exists an $\tilde{s} \in \mathcal{K}$ such that $s \leq \tilde{s} < 2s$. We have $\mathbb{E}_{\theta}(1-\phi_{\mathcal{P}, \mathrm{adaptive}}) \leq \mathbb{E}_{\theta}(1-{\phi}_{\mathcal{P}, \mathrm{sparse}, \tilde{s}}^{\mathrm{MoM}})$. Now, by carefully inspecting the alternative term part of the proof of \Cref{thm:finitemoment_upperbound_sparse}, we can still deduce $\mathbb{E}_{\theta}(1-{\phi}_{\mathcal{P}, \mathrm{sparse}, \tilde{s}}^{\mathrm{MoM}}) \leq \varepsilon/2$
as long as $\rho$ satisfies
\begin{align} \label{Eq:adaptive_1}
\rho^2 &\geq C'\Bigl(s\bigl((p/s)^{2/\alpha} + \log\log(8n)\bigr)\Bigr) \geq \frac{C'}{2}\Bigl(\tilde{s}\bigl((p/\tilde{s})^{2/\alpha} + \log\log(8n)\bigr)\Bigr) \nonumber \\ 
&\geq C'' \max\Bigl\{\log^2(8/\varepsilon) \max_{t\in\mathcal{T}\setminus\{1\}} (r_{t,\tilde{s}} + 2\tilde{s} G_t) , (r_{1,\tilde{s}}+2\tilde{s})\log^2(8/\varepsilon), a^2_{\tilde{s}} \tilde{s}\Bigr\},
\end{align}
for sufficiently large $C'$ and $C''$, where the final inequality in~\eqref{Eq:adaptive_1} remains true with our modified choice of $r_{t,s}$. \\
(3) $p < \log^{\alpha-2}(\log(8n))$ and $s \geq s_1/2$. We use the same argument as in (1) to obtain the same condition $\rho^2 \geq C'\sqrt{p} \log\log(8n)$. Similarly, we also note that when $s_1/2\leq s < s_1$, we have
\[
\frac{1}{2}\sqrt{p}\log \log(8n) \leq s(p/s)^{2/\alpha} + \log\log(8n) \leq \sqrt{p} \log\log(8n).
\]
(4) $p < \log^{\alpha-2}(\log(8n))$ and $s < s_1/2$. Similar to (2), we have $\mathbb{E}_{\theta}(1-\phi_{\mathcal{P}, \mathrm{adaptive}}) \leq \mathbb{E}_{\theta}(1-{\phi}_{\mathcal{P}, \mathrm{sparse}, \tilde{s}}^{\mathrm{RSM}})$. By carefully examining the alternative term part of the proof of \Cref{thm:finitemoment_upperbound_sparse_improve}, we can obtain $\mathbb{E}_{\theta}(1-{\phi}_{\mathcal{P}, \mathrm{sparse}, \tilde{s}}^{\mathrm{RSM}}) \leq \varepsilon/2$ as long as
\begin{align} \label{Eq:adaptive_3}
\rho^2 &\geq C'\bigl(s(p/s)^{2/\alpha} + \log\log(8n) \bigr) \geq \frac{C'}{2}\bigl(\tilde{s}(p/\tilde{s})^{2/\alpha} + \log\log(8n) \bigr) \nonumber \\
&\geq C''\max\bigl\{(\tilde{r}_{t\neq 1, \tilde{s}} + 2\tilde{s})\log^2(8/\varepsilon), (\tilde{r}_{1,\tilde{s}}+2\tilde{s})\log^2(8/\varepsilon), a^2_{\tilde{s}}\tilde{s}, \tilde{s}\log(ep/\tilde{s}) + \log\log(8n)\bigr\},
\end{align}
for sufficiently large $C'$ and $C''$, where the final inequality in~\eqref{Eq:adaptive_3} remains true with our new choices of $a_s$ and $\tilde{r}_{t,s}$.

The desired result then follows from \Cref{thm:finitemoment_combinerate} and the first part of its proof.
\end{proof}

\subsection{Proofs of lower bound results in Sections~\ref{Sec:hidimtest} and~\ref{Sec:robusttest} }\label{proof:lowerbounds}
For $\mathcal{Q} = \mathcal{G}_{\alpha, K}^\otimes$, to prove \Cref{thm:subweibull-lowerbound-minimum}, we establish the lower bound separately for the sparse and dense regimes. In the sparse regime $s < \sqrt{p} \log^{-2/\alpha}(ep)$, we have  
\[
s\log^{2/\alpha}(ep/s) \leq \sqrt{p} = \sqrt{p (\log \log(8n))^{\omega_1}},
\]
and we shall prove the lower bound $s\log^{2/\alpha}(ep/s) + \log\log(8n)$, as stated in \Cref{prop:weibulldlowerbound_sparse} below. In the dense regime, we first consider when $\sqrt{p} \log^{-2/\alpha}(ep) \leq s \leq \sqrt{p\log \log (8n)}$, we have  
   \[
   s\log^{2/\alpha}(ep/s) \gtrsim \sqrt{p} = \sqrt{p (\log \log(8n))^{\omega_1}}.
   \]
When $s > \sqrt{p\log \log (8n)}$, we still have  
   \[
   s\log^{2/\alpha}(ep/s) \geq s \geq \sqrt{p\log\log(8n)}.
   \]
Thus, in the dense regime, it suffices to prove the lower bound $\sqrt{p (\log \log(8n))^{\omega_1}} + \log \log(8n)$, as stated in \Cref{prop:weibulldlowerbound_dense} below.

Similarly for $\mathcal{Q} = \mathcal{P}_{\alpha, K}^\otimes$, to prove \Cref{thm:finitemoment-lowerbound-minimum}, it suffices to establish \Cref{prop:finitemoment_lowerbound_sparse} for the sparse regime and \Cref{prop:finitemoment_lowerbound_dense} for the dense regime.

\begin{prop} \label{prop:weibulldlowerbound_sparse}
For $\mathcal{Q} = \mathcal{G}_{\alpha, K}^\otimes$ with $0 < \alpha \leq 2$ and $K \geq K_\alpha$, for some constant $K_\alpha > 0$ depending only on $\alpha$. Assume $c \leq s < \sqrt{p}\log^{-2/\alpha}(ep)$, for some absolute constant $c \geq 1$. There exists some constant $c' > 0$ depending only on $\alpha$ and $K$, such that $\mathcal{R}_{\mathcal{G}}(\rho) \geq 1/2$
whenever $\rho^2 \leq c' v_{\mathcal{G}, \mathrm{sparse}}^{\mathrm{L}}$, where 
\[
v_{\mathcal{G}, \mathrm{sparse}}^{\mathrm{L}} := s\log^{2/\alpha}(ep/s) + \log\log(8n).
\]
\end{prop}

\begin{prop} \label{prop:weibulldlowerbound_dense}
For $\mathcal{Q} = \mathcal{G}_{\alpha, K}^\otimes$ with $0 < \alpha \leq 2$ and $K \geq K_\alpha$, for some constant $K_\alpha > 0$ depending only on $\alpha$. Assume $s \geq \sqrt{p}\log^{-2/\alpha}(ep) \vee c$, for some absolute constant $c \geq 1$. There exists some constant $c' > 0$ depending only on $\alpha$ and $K$, such that $\mathcal{R}_{\mathcal{G}}(\rho) \geq 1/2$
whenever $\rho^2 \leq c' v_{\mathcal{G}, \mathrm{dense}}^{\mathrm{L}}$, where 
\[
v_{\mathcal{G}, \mathrm{dense}}^{\mathrm{L}} := \sqrt{p (\log \log(8n))^{\omega_1}} + \log \log(8n)
\]
and $\omega_1 = \mathbbm{1}_{\bigl\{s \geq \sqrt{p \log \log(8n)}\bigr\}}$.
\end{prop}

\begin{prop} \label{prop:finitemoment_lowerbound_sparse}
For $\mathcal{Q} = \mathcal{P}_{\alpha, K}^\otimes$ with $\alpha \geq 2$ and $K \geq K_\alpha$, for some constant $K_\alpha > 0$ depending only on $\alpha$. Assume $c \leq s < p^{\frac{1}{2} - (\frac{1}{\alpha-2} \wedge \frac{1}{2})}$, for some absolute constant $c \geq 1$. There exists some constant $c' > 0$ depending only on $\alpha$ and $K$, such that $\mathcal{R}_{\mathcal{P}}(\rho) \geq 1/2$
whenever $\rho^2 \leq c' v_{\mathcal{P}, \mathrm{sparse}}^{\mathrm{L}}$, where 
\[
v_{\mathcal{P}, \mathrm{sparse}}^{\mathrm{L}} := s(p/s)^{2/\alpha} + \log \log(8n).
\]
\end{prop}

\begin{prop} \label{prop:finitemoment_lowerbound_dense}
For $\mathcal{Q} = \mathcal{P}_{\alpha, K}^\otimes$ with $\alpha \geq 2$ and $K \geq K_\alpha$, for some constant $K_\alpha > 0$ depending only on $\alpha$. Assume $s \geq p^{\frac{1}{2} - (\frac{1}{\alpha-2} \wedge \frac{1}{2})} \vee c$, for some absolute constant $c \geq 1$. There exists some constant $c' > 0$ depending only on $\alpha$ and $K$, such that $\mathcal{R}_{\mathcal{P}}(\rho) \geq 1/2$
whenever $\rho^2 \leq c' v_{\mathcal{P}, \mathrm{dense}}^{\mathrm{L}}$, where 
\[
v_{\mathcal{P}, \mathrm{dense}}^{\mathrm{L}} := p^{(2/\alpha) \vee (1/2)}  (\log \log(8n))^{\omega_2} + \log \log(8n)
\]
and $\omega_2 =(1/2) \mathbbm{1}_{\bigl\{s > \sqrt{p \log \log(8n)}\bigr\} \cap \{\alpha \geq 4\} }$.
\end{prop}

We now prove all these four lower bound results. Throughout the proof, we use $P_{\theta, \Xi}$ to denote the probability distribution of $X \in \mathbb{R}^{p\times n}$ that satisfies $X - \theta \sim \Xi$, and $E_{\theta, \Xi}$ the corresponding expectation under this distribution. It suffices to prove the five lower bound rate claims below, as they then immediately imply Propositions~\ref{prop:weibulldlowerbound_sparse},~\ref{prop:weibulldlowerbound_dense},~\ref{prop:finitemoment_lowerbound_sparse} and~\ref{prop:finitemoment_lowerbound_dense}.
\begin{enumerate}[label=(\roman*).]
\item $\log\log(8n)$, for $\mathcal{G}_{\alpha,K}^{\otimes}$ with $0 < \alpha \leq 2$ and $K \geq 2^{1+2/\alpha}$ and for $\mathcal{P}_{\alpha,K}^{\otimes}$ with $\alpha > 2$ and $K\geq \sqrt{\alpha+1}$ or $\alpha=2$ and $K \geq 1$;
\item $\sqrt{p\log\log(8n)}$ when $s \geq \sqrt{p\log\log(8n)}$, for $\mathcal{G}_{\alpha,K}^{\otimes}$ with $0 < \alpha \leq 2$ and $K \geq 2^{1+2/\alpha}$ and for $\mathcal{P}_{\alpha,K}^{\otimes}$ with $\alpha > 2$ and $K\geq \sqrt{\alpha+1}$ or $\alpha=2$ and $K \geq 1$;
\item $p^{2/\alpha}$ when $s \geq 30$, for $\mathcal{P}_{\alpha,K}^{\otimes}$ with $\alpha > 2$ and $K \geq \sqrt{2}$ or $\alpha=2$ and $K \geq 1$;
\item $s(p/s)^{2/\alpha}$ when $30 \leq s \leq p^{\frac{\alpha-4}{2\alpha-4}}$, for $\mathcal{P}_{\alpha,K}^{\otimes}$ with $\alpha \geq 4$ and $K \geq \sqrt{2}$;
\item $s\log^{2/\alpha}(ep/s)$ when $30 \leq s \leq \sqrt{p}\log^{-2/\alpha}(ep)$, for $\mathcal{G}_{\alpha,K}^{\otimes}$ with $0 < \alpha \leq 2$ and $K \geq 2^{1+2/\alpha}$.
\end{enumerate}

\medskip

\noindent \textit{Proof of claim (i).} We first consider that each entry of the noise matrix $E$ follows an independent standard normal distribution. Then for $0 < \alpha \leq 2$, $i \in [n]$, $j \in [p]$ and $x \geq 2^{1+2/\alpha}$, we have
\begin{align*}
\mathbb{E}\Biggl[\exp\biggl\{ \biggl(\frac{|E_i(j)|}{x} \biggr)^\alpha\biggr\}\Biggr]  &= \mathbb{E}\Biggl[\exp\biggl\{ \biggl(\frac{|E_i(j)|}{x} \biggr)^\alpha\biggr\} \mathbbm{1}_{\{|E_i(j)| \geq 2\}}\Biggr] + \mathbb{E}\Biggl[\exp\biggl\{ \biggl(\frac{|E_i(j)|}{x} \biggr)^\alpha\biggr\} \mathbbm{1}_{\{|E_i(j)| < 2\}}\Biggr] \\
&\leq \mathbb{E}\Biggl[\exp\biggl\{ \biggl(\frac{|E_i(j)|}{2} \biggr)^2\biggr\} \mathbbm{1}_{\{|E_i(j)| \geq 2\}}\Biggr] + \exp\bigl\{ (2/x)^\alpha\bigr\} \\
&= \sqrt{2} - \mathbb{E}\Biggl[\exp\biggl\{ \biggl(\frac{|E_i(j)|}{2} \biggr)^2\biggr\} \mathbbm{1}_{\{|E_i(j)| < 2\}}\Biggr] + \exp\bigl\{ (2/x)^\alpha\bigr\} \\
&\leq \sqrt{2} - \bigl(1-\exp(-2)\bigr) + \exp\bigl\{ (2/x)^\alpha\bigr\} < 2,
\end{align*}
where the penultimate inequality follows from the standard Gaussian tail bound. Thus, for any $K \geq 2^{1+2/\alpha}$, we have $\|E_i(j)\|_{\psi_\alpha} \leq K$. Furthermore, by Jensen's inequality, we obtain for $\alpha > 2$
\[
\mathbb{E}|E_i(j)|^\alpha \leq \biggl\{\mathbb{E}|E_i(j)|^{2\lceil\alpha/2\rceil} \biggr\}^{\frac{\alpha/2}{\lceil\alpha/2\rceil}} = \biggl\{ \prod_{i=1}^{\lceil\alpha/2\rceil} (2i-1)\biggr\}^{\frac{\alpha/2}{\lceil\alpha/2\rceil}} \leq \bigl(2\lceil \alpha/2\rceil-1 \bigr)^{\alpha/2} \leq (\alpha+1)^{\alpha/2}.
\]
Therefore $P_e \in \mathcal{G}_{\alpha, K}^{\otimes}$ for all $0 < \alpha \leq 2$ and $K \geq 2^{1+\alpha/2}$ and $P_e \in \mathcal{P}_{\alpha, K}^{\otimes}$ for all $\alpha \geq 2$ and $K \geq \sqrt{\alpha+1}$ or $\alpha=2$ and $K \geq 1$. For the mean vectors $\mu_1$ and $\mu_2$ in the definition of $\Theta^{(t_0)}(p,n,s,\rho)$, we restrict them to be equal in all coordinates except perhaps the first. Then under this setting, the lower bound $\log\log(8n)$ of the detection rate is established in \citet[][Proposition~4.2]{Gao2020isotonic}. Note that this lower bound holds for all $1 \leq s \leq p$.

\medskip

\noindent \textit{Proof of claim (ii).} When $s \geq \sqrt{p\log\log(8n)}$, we again consider the independent standard normal noise structure. The lower bound $\sqrt{p\log\log(8n)}$ is shown in \citet[][Proposition~3]{liu2021minimax}.

\medskip

We now use a unified approach to establish the three remaining rates. Let $\xi$ and $\tilde{\xi}$ be two independent random variables on $\mathbb{R}$, whose distributions are to be specified later; let $\tilde{\mathcal{\omega}}$ be an discrete random variable (independent of $\xi, \tilde{\xi}$), taking values
\begin{equation} \label{Eq:omega_tilde}
\tilde{\omega} = \begin{cases}
    +1  &\text{ w.p. } \frac{s}{4p}\bigl(1+\frac{\gamma^2 s}{2p}\bigr)^{-1}\\
    -1 &\text{ w.p. } \frac{s}{4p}\bigl(1+\frac{\gamma^2 s}{2p}\bigr)^{-1}\\
    0 &\text{ otherwise,}
\end{cases}
\end{equation}
where $\gamma > 0$ is also to be specified later; let $\tilde{\pi} := \tilde{\xi} + \gamma \tilde{\omega}$. We remark that $\tilde{\omega}$ can be viewed as a Rademacher random variable being multiplied by a Bernoulli random variable. Denote $\underline{\xi} := (\xi(1), \dotsc, \xi(p))^\top \in \mathbb{R}^p$, where the coordinates are i.i.d. copies of $\xi$ and we use similar notations $\underline{\tilde{\xi}}, \underline{\tilde{\omega}}, \underline{\tilde{\pi}}$. Let $\nu$ denote the distribution of $\gamma \tilde{\underline{\omega}} \in \mathbb{R}^p$, and $\bar{\nu}$ the distribution restricted to $\mathcal{V}_s := \{v\in \mathbb{R}^p: s/6 \leq \|v\|_0 \leq s\}$, i.e.\ $\bar{\nu}(A) = \frac{\nu(A \cap \mathcal{V}_s) }{ \nu(\mathcal{V}_s ) }$ for any Borel set $A \subseteq \mathbb{R}^p$. Consequently, the support of this restricted measure satisfies
\begin{equation} \label{Eq:restricted_measure}
    \mathrm{supp}(\bar{\nu}) \subseteq \bigl\{v \in \mathbb{R}^p: \|v\|_0 \leq s, \|v\|^2_2 \geq s\gamma^2/6 \bigr\}.
\end{equation}
We also have
\begin{equation} \label{Eq:measure_close}
    - \nu(\mathcal{V}_s^c) = -\Bigl(\frac{1}{\nu(\mathcal{V}_s)}-1\Bigr) \nu(\mathcal{V}_s)\leq \nu(A) - \bar{\nu}(A) = \nu(A \cap \mathcal{V}_s^c) - \Bigl(\frac{1}{\nu(\mathcal{V}_s)}-1\Bigr) \nu(A \cap \mathcal{V}_s) \leq \nu(\mathcal{V}_s^c).
\end{equation}
for any Borel set $A$. Denote $\Xi^*$ to be the distribution of $(\underline{\xi},R_2, \dotsc, R_n) \in \mathbb{R}^{p\times n}$, $\tilde{\Xi}^*$ the distribution of $(\tilde{\underline{\xi}}, R_2, \dotsc, R_n)$, and $\tilde{\Pi}$ the distribution of $(\tilde{\underline{\pi}}, R_2, \dotsc, R_n)$, where $(R_i(j))_{i\in \{2,\dotsc,n\}, j \in [p]}$ are i.i.d. Rademacher random variables, independent of $\underline{\xi},\tilde{\underline{\xi}}, \tilde{\underline{\pi}}$. Now we consider the following mixture measures:
\[
\bar{\mathbf{P}}^* := \int P_{\theta^{(1)}, {\Xi}^*} \, \bar{\nu}(d\theta_1), \quad \mathbf{P}^* := \int P_{\theta^{(1)}, {\Xi}^*} \, \nu(d\theta_1), \quad \text{and} \quad \tilde{\mathbf{P}}^* := \int P_{\theta^{(1)}, \tilde{\Xi}^*} \, \nu(d\theta_1),
\]
where $\theta^{(1)} := (\theta_1,0, \dotsc, 0) \in \mathbb{R}^{p \times n}$. Observe that $\tilde{\mathbf{P}}^* = P_{0, \tilde{\Pi}}$, as both sides represent the distribution of $(\tilde{\underline{\pi}},R_2, \dotsc, R_n)$. We first provide an upper bound on the total variation distance between ${\mathbf{P}}^*$ and $\bar{\mathbf{P}}^*$. By~\eqref{Eq:measure_close}, we have
\begin{equation} \label{Eq:TV_1}
\mathrm{TV}({\mathbf{P}}^*, \bar{\mathbf{P}}^*) \leq \mathrm{TV}(\nu, \bar{\nu}) = \sup_{A} |\nu(A) - \bar{\nu}(A)| \leq \nu(\mathcal{V}_s^c) = \mathbb{P}\bigl(\| \underline{\tilde{ \omega} } \|_0 > s\bigr) + \mathbb{P}\bigl(\| \underline{\tilde{ \omega}} \|_0 < s/6\bigr).
\end{equation}
Suppose $\gamma$ is chosen to satisfy $\gamma \leq \sqrt{p/s}$. Then from~\eqref{Eq:omega_tilde}, we deduce $\frac{s}{3p}\leq \mathbb{P}(\tilde{\omega}(1) \neq 0) < \frac{s}{2p}$. By Chernoff bounds, we have
\begin{align}
    \mathbb{P}\bigl(\| \underline{\tilde{\omega}} \|_0 > s\bigr) &\leq \frac{\mathbb{E}\bigl[e^{\| \underline{\tilde{\omega}} \|_0 \log 2  }\bigr]}{e^{s\log 2}} \leq \frac{\bigl(1+s/(2p)\bigr)^p}{e^{s\log 2}} \leq e^{-s/6}, \nonumber \\
    \mathbb{P}\bigl(\| \underline{\tilde{\omega}} \|_0 < s/6\bigr) &\leq \frac{\mathbb{E}\bigl[e^{-\| \underline{\tilde{\omega}} \|_0 \log 2  }\bigr]}{e^{-(s\log 2)/6}} \leq \frac{\bigl(1-s/(6p)\bigr)^p}{e^{-(s\log 2)/6}} \leq e^{-s/20}. \label{Eq:measure_sparsity}
\end{align}
The key step of the proof is to carefully construct two random variables $\xi$ and $\tilde{\xi}$ such that the following three conditions are satisfied:
\begin{align}
    \Xi^* \in \mathcal{G}_{\alpha, K} \  &(\text{resp. } \mathcal{P}_{\alpha, K}), \label{Eq:LB_con1} \\
    \tilde{\Pi} \in \mathcal{G}_{\alpha, K} \  &(\text{resp. } \mathcal{P}_{\alpha, K}), \label{Eq:LB_con2} \\
    \mathrm{H}^2(P_\xi, P_{\tilde{\xi}}) &\leq \frac{1}{16p}, \label{Eq:LB_con3}
\end{align}
where, in a slight abuse of notation, we denote $P_{\xi}$ and $P_{\tilde{\xi}}$ to be the distribution of $\xi$ and $\tilde{\xi}$ respectively. Then, by data processing inequality as well as some basic properties of the total variation distance and the Hellinger distance, we obtain
\begin{align} 
    \mathrm{TV}(\tilde{\mathbf{P}}^*, {\mathbf{P}}^*) &\leq \mathrm{TV}(P_{0, \tilde{\Xi}^*}, P_{0, {\Xi}^*}) \leq \mathrm{TV}\bigl(P_{\tilde{\underline{\xi}}}, P_{\underline{\xi}}\bigr) \leq \mathrm{H}\bigl(P_{\tilde{\underline{\xi}}}, P_{\underline{\xi}}\bigr) = \sqrt{2\bigl(1-(1-\mathrm{H}^2(P_{\tilde{\xi}}, P_\xi)/2)^p\bigr)} \nonumber \\
    &\leq \sqrt{p\mathrm{H}^2(P_\xi, P_{\tilde{\xi}})} \leq 1/4, \label{Eq:TV_2}
\end{align}
where the penultimate inequality follows from the fact that $(1-x)^p \geq 1-px$ for all $0 \leq x \leq 1$ and $p \geq 1$. Combining~\eqref{Eq:restricted_measure},~\eqref{Eq:TV_1},~\eqref{Eq:measure_sparsity}, and~\eqref{Eq:TV_2}, when $s \geq 30$, for all $\rho^2 \leq s\gamma^2/12$, we have
\begin{align*}
\mathcal{R}_{\mathcal{Q}}(\rho) &= \inf_{\phi \in \Phi} \biggl\{ \sup_{P_e \in \mathcal{Q}}\sup_{\theta \in \Theta_0(p,n)}\mathbb{E}_{\theta, P_e} \phi + \sup_{P_e \in \mathcal{Q}}\sup_{\theta \in \Theta(p,n,s,\rho)}\mathbb{E}_{\theta, P_e} (1-\phi) \biggr\} \\
&\geq 1-\mathrm{TV}(P_{0, \tilde{\Pi}}, \bar{\mathbf{P}}^*) = 1-\mathrm{TV}(\tilde{\mathbf{P}}^*, \bar{\mathbf{P}}^*) \geq 1-\mathrm{TV}(\tilde{\mathbf{P}}^*, {\mathbf{P}}^*)-\mathrm{TV}({\mathbf{P}}^*, \bar{\mathbf{P}}^*) \\
&\geq 3/4 - e^{-s/6} - e^{-s/20} \geq 1/2,  
\end{align*}
where the class $\mathcal{Q}$ is either $\mathcal{G}_{\alpha, K}^{\otimes}$ or $\mathcal{P}_{\alpha, K}^{\otimes}$. Below, we give three constructions of $\xi$ and $\tilde{\xi}$ that satisfy conditions~\eqref{Eq:LB_con1},~\eqref{Eq:LB_con2} and~\eqref{Eq:LB_con3}, and specify the corresponding choices of $\gamma$. Each construction corresponds to a rate given at the beginning of the proof.

\medskip

\noindent \textit{Proof of claim (iii).} We work within the noise distribution class $\mathcal{P}_{\alpha, K}^{\otimes}$ with $\alpha > 2$ and $K \geq \sqrt{2}$ or $\alpha=2$ and $K \geq 1$ and we only consider $s = 30$ (a constant) in this construction. Let $\xi$ and $\tilde{\xi}$ be two independent discrete random variables such that
\[
\tilde{\xi} = 
\begin{cases}
\bigl(1 + \frac{\gamma^2 s}{2p}\bigr)^{-1/2}  &\text{ w.p. } 1/2 \\
-\bigl(1 + \frac{\gamma^2 s}{2p}\bigr)^{-1/2}  &\text{ w.p. } 1/2
\end{cases}
\qquad \text{and} \qquad
\xi = 
\begin{cases}
\bigl(1 + \frac{\gamma^2 s}{2p}\bigr)^{-1/2}  &\text{ w.p. } \frac{t_0^2-1}{2\Bigl(t_0^2 - \bigl(1 + \frac{\gamma^2 s}{2p}\bigr)^{-1}\Bigr)}\\
-\bigl(1 + \frac{\gamma^2 s}{2p}\bigr)^{-1/2}  &\text{ w.p. } \frac{t_0^2-1}{2\Bigl(t_0^2 - \bigl(1 + \frac{\gamma^2 s}{2p}\bigr)^{-1}\Bigr)}\\
t_0 &\text{ w.p. } \frac{1-\bigl(1 + \frac{\gamma^2 s}{2p}\bigr)^{-1}}{2\Bigl(t_0^2 - \bigl(1 + \frac{\gamma^2 s}{2p}\bigr)^{-1}\Bigr)} \\
-t_0 &\text{ w.p.  } \frac{1-\bigl(1 + \frac{\gamma^2 s}{2p}\bigr)^{-1}}{2\Bigl(t_0^2 - \bigl(1 + \frac{\gamma^2 s}{2p}\bigr)^{-1}\Bigr)}.
\end{cases}
\]
Direct calculations show that both $\xi$ and $\tilde{\xi} + \gamma \tilde{\omega}$ have mean 0 and variance 1. Choose
\[
\gamma = \max \biggl\{ -1 + \frac{\bigl\{(K^\alpha-1)p/s\bigr\}^{1/\alpha}}{\max\{32, K\}}, \frac{\sqrt{2}}{32}\biggr\}
 \quad \text{and} \quad t_0 = 32\gamma \geq \sqrt{2}.
\]
Note that we have $\gamma \leq \sqrt{p/s}$. Now, to check~\eqref{Eq:LB_con1} and~\eqref{Eq:LB_con2}, it suffices to only verify that $\mathbb{E}|\xi|^\alpha \leq K^\alpha$ and that $\mathbb{E}\bigl|\tilde{\xi} + \gamma \tilde{\omega} \bigr|^\alpha \leq K^\alpha$ respectively. Indeed, as $\alpha > 2$ and $K \geq \sqrt{2}$, we have
\begin{align*}
\mathbb{E}|\xi|^\alpha &\leq 1 + t_0^\alpha \frac{1-\bigl(1 + \frac{\gamma^2 s}{2p}\bigr)^{-1}}{t_0^2 - \bigl(1 + \frac{\gamma^2 s}{2p}\bigr)^{-1}} \leq 1 + t_0^\alpha \frac{\gamma^2 s/(2p)}{t_0^2-1} = 1 + \frac{\gamma^2 s t_0^{\alpha-2} }{p} \leq 1 + \frac{32^{\alpha-2} \gamma^\alpha s}{p} \\
&\leq 1 + \max\biggl\{ K^\alpha - 1, 2^{\alpha/2 - 10} \biggr\} \leq K^\alpha,
\end{align*}
and
\begin{align*}
\mathbb{E} \bigl| \tilde{\xi} + \gamma \tilde{\omega} \bigr|^\alpha &\leq 1 + (1+\gamma)^\alpha \cdot \mathbb{P}(\tilde{\omega} \neq 0) \leq 1 + \frac{(1+\gamma)^\alpha s}{2p} \leq 1 + \max\biggl\{ K^\alpha - 1, \frac{(17/16)^\alpha}{2} \biggr\} \leq K^\alpha.
\end{align*}
We also verify~\eqref{Eq:LB_con3}:
\begin{align*}
    \mathrm{H}^2(P_\xi, P_{\tilde{\xi}}) &= \Biggl( 1 - \sqrt{\frac{t_0^2-1}{t_0^2 - \bigl(1 + \frac{\gamma^2 s}{2p}\bigr)^{-1}}}\Biggr)^2 + \frac{1-\bigl(1 + \frac{\gamma^2 s}{2p}\bigr)^{-1}}{t_0^2 - \bigl(1 + \frac{\gamma^2 s}{2p}\bigr)^{-1}} \leq \frac{2 \Bigl(1-\bigl(1 + \frac{\gamma^2 s}{2p}\bigr)^{-1} \bigr)}{t_0^2 - \bigl(1 + \frac{\gamma^2 s}{2p}\bigr)^{-1}} \\
    &\leq \frac{2\gamma^2 s}{pt_0^2} = \frac{60\gamma^2}{(32\gamma)^2 p}\leq \frac{1}{16p}.
\end{align*}
We thus conclude that under the noise distribution class $\mathcal{P}_{\alpha, K}^{\otimes}$ with $\alpha > 2$ and $K \geq \sqrt{2}$, whenever $s \geq 30$ and
\[
\rho^2 \leq \frac{30}{12} \Biggl( \max \biggl\{ -1 + \frac{\bigl\{(K^\alpha-1)p/30\bigr\}^{1/\alpha}}{\max\{32, K\}}, \frac{\sqrt{2}}{32}\biggr\} \Biggr)^2 \leq c\cdot p^{2/\alpha},
\]
for some $c > 0$ depending only on $\alpha$ and $K$, we have
$\mathcal{R}_{\mathcal{P}}(\rho) \geq 1/2$. When $\alpha=2$, we can simply set $\gamma = \sqrt{p/s}$ and $t_0=32\gamma$ and reach the same conclusion.

\medskip

\noindent \textit{Proof of claim (iv).} We work within the noise distribution class $\mathcal{P}_{\alpha, K}^{\otimes}$ with $\alpha \geq 4$ and $K \geq \sqrt{2}$. We first define an auxiliary random variable $\xi_{\mathrm{aux}}$ and with the following density elsewhere:
\begin{equation*}
f_{\xi_{\mathrm{aux}}}(x) = \begin{cases}
    1000(x- \mathrm{sgn}(x) \cdot 0.9)^2 &\qquad 0.9 \leq |x| < 0.95 \\
    5-1000(x-\mathrm{sgn}(x))^2 &\qquad 0.95 \leq |x| \leq 1.05 \\
    1000(x-\mathrm{sgn}(x) \cdot 1.1)^2 &\qquad 1.05< |x| \leq 1.1 \\
    0 &\qquad \mathrm{otherwise}.
\end{cases}
\end{equation*}
Observe that $\mathbb{E}\xi_{\mathrm{aux}} = 0$ and $\sigma_{\mathrm{aux}}^2 := \mathbb{E}\xi_{\mathrm{aux}}^2 \in (1, 1.01)$. Now let $\xi$ and $\tilde{\xi}$ be independent random variables such that $\xi \stackrel{d}{=} \sigma_{\mathrm{aux}}^{-1} \xi_{\mathrm{aux}}$ and $\tilde{\xi} \stackrel{d}{=} \bigl(1 + \frac{\gamma^2 s}{2p}\bigr)^{-1/2} \sigma_{\mathrm{aux}}^{-1} \xi_{\mathrm{aux}}$. Again, direct calculations show that both $\xi$ and $\tilde{\xi} + \gamma \tilde{\omega}$ have mean 0 and variance 1. For condition~\eqref{Eq:LB_con1}, since $|\xi| < 1.1$ holds with probability one,  we have $\Xi^* \in \mathcal{P}_{\alpha, K}$ for all $\alpha \geq 4$ and $K \geq \sqrt{2}$. We choose
\[
\gamma = \frac{1}{12}(p/s)^{1/\alpha}.
\]
Note that $\gamma \leq \sqrt{p/s}$. We verify~\eqref{Eq:LB_con2} as follows:
\begin{align} \label{Eq:verify_con2}
\mathbb{E} \bigl| \tilde{\xi} + \gamma \tilde{\omega} \bigr|^\alpha &\leq 1.1^\alpha + (1.1+\gamma)^\alpha \cdot \mathbb{P}(\tilde{\omega} \neq 0) \leq 1.1^\alpha + \frac{(1.1+\gamma)^\alpha s}{2p} \leq 1.1^\alpha + \frac{\max\{1.2, 12\gamma\}^\alpha s}{2p} \nonumber \\
&\leq \max\biggl\{1.1^\alpha + \frac{1.2^\alpha}{2}, 1.1^\alpha + \frac{1}{2}\biggr\} \leq 2^{\alpha/2} \leq K^\alpha.
\end{align}
Finally, by \citet[Theorem~7.6]{ibragimov2013statistical}, we have when $s \leq p^{\frac{\alpha-4}{2\alpha-4}}$
\begin{align} \label{Eq:verify_con3}
\mathrm{H}^2(P_\xi, P_{\tilde{\xi}}) &\leq \frac{\Bigl(\sigma_{\mathrm{aux}}^{-1}-\bigl(1 + \frac{\gamma^2 s}{2p}\bigr)^{-1/2} \sigma_{\mathrm{aux}}^{-1} \Bigr)^2}{4} \sup_{u \in \bigl[\bigl(1 + \frac{\gamma^2 s}{2p}\bigr)^{-1/2} \sigma_{\mathrm{aux}}^{-1}, \sigma_{\mathrm{aux}}^{-1}\bigr]} \frac{\int_{\mathrm{supp}(\xi_{\mathrm{aux}})} \bigl(f'_{\xi_{\mathrm{aux}}}(x)\bigr)^2/f_{\xi_{\mathrm{aux}}}(x)\, dx}{u^2} \nonumber \\
&\leq \frac{\Bigl(\bigl(1 + \frac{\gamma^2 s}{2p}\bigr)^{1/2}-1\Bigr)^2}{4}\int_{\mathrm{supp}(\xi_{\mathrm{aux}})} \bigl(f'_{\xi_{\mathrm{aux}}}(x)\bigr)^2/f_{\xi_{\mathrm{aux}}}(x)\, dx \nonumber \\
&\leq \frac{\gamma^4s^2}{64p^2} \cdot 4\biggl(\int_{0}^{0.05} \frac{(-2000x)^2}{5-1000x^2} \, dx + \int_{0.05}^{0.1} \frac{(2000(x-0.1))^2}{1000(x-0.1)^2} \, dx \biggr) \leq \frac{25\gamma^4 s^2}{p^2} \leq \frac{1}{16p},
\end{align}
and this verifies~\eqref{Eq:LB_con3}. Therefore, under the noise distribution class $\mathcal{P}_{\alpha, K}^{\otimes}$ with $\alpha \geq 4$ and $K \geq \sqrt{2}$, whenever $30 \leq s \leq p^{\frac{\alpha-4}{2\alpha-4}}$ and $\rho^2 \leq s(p/s)^{2/\alpha}/1728$, we have
$\mathcal{R}_{\mathcal{P}}(\rho) \geq 1/2$.

\medskip

\noindent \textit{Proof of claim (v).} We work within the noise distribution class $\mathcal{G}_{\alpha, K}^{\otimes}$ with $0 < \alpha \leq 2$ and $K \geq 2^{1+2/\alpha}$. We use the same construction as in (iv), but now choose instead
\[
\gamma = \frac{1}{3\cdot(8/\alpha)^{1/\alpha}}\log^{1/\alpha}(ep/s).
\]
Since $\log x \leq \frac{2}{e\alpha} x^{\alpha/2}$ for all $x \geq e$, we can verify that $\gamma \leq \sqrt{p/s}$. Again, for condition~\eqref{Eq:LB_con1}, since $|\xi| < 1.1$ holds with probability one, we have $\Xi^* \in \mathcal{G}_{\alpha, K}$ for all $\alpha \geq 4$ and $K \geq 2^{1+2/\alpha}$, as $\exp\{(1.1/K)^\alpha\} \leq e^{1/4}<2$. We now verify~\eqref{Eq:LB_con2} using the technique in~\eqref{Eq:verify_con2}:
\begin{align*}
\mathbb{E}\Biggl[\exp\biggl\{ \biggl(\frac{|\tilde{\xi} + \gamma \tilde{\omega}|}{K} \biggr)^\alpha\biggr\}\Biggr] &\leq \exp\biggl\{\biggl(\frac{1.1}{K}\biggr)^\alpha\biggr\} + \frac{s}{2p}\exp\biggl\{\biggl(\frac{\max\{2, 3\gamma\}}{K}\biggr)^\alpha\biggr\}  \\
&\leq e^{1/4} +  \max\Biggl\{e^{1/4}\frac{s}{2p}, \Bigl(\frac{s}{2p}\Bigr)^{1 - \frac{2\alpha}{8K^\alpha}}  \Biggr\} \\
&\leq e^{1/4} + \max\bigl\{e^{1/4}/2, \sqrt{1/2}\bigr\} < 2.
\end{align*}
We then follow~\eqref{Eq:verify_con3} to verify~\eqref{Eq:LB_con3} as well:
\[
\mathrm{H}^2(P_\xi, P_{\tilde{\xi}}) \leq \frac{25\gamma^4 s^2}{p^2} \leq \frac{1}{16p},
\]
when $s \leq \sqrt{p}\log^{-2/\alpha}(ep)$. Therefore, under the noise distribution class $\mathcal{G}_{\alpha, K}^{\otimes}$ with $\alpha \leq 2$ and $K \geq 2^{1+2/\alpha}$, whenever $30 \leq s \leq \sqrt{p}\log^{-2/\alpha}(ep)$ and $\rho^2 \leq \frac{s\log^{2/\alpha}(ep/s)}{36\cdot (8/\alpha)^{1/\alpha}}$, we have
$\mathcal{R}_{\mathcal{G}}(\rho) \geq 1/2$.

\section{Technical details of Section~\ref{Sec:extensions} }\label{sec:app-extensions}
\subsection{Multiple change points}
\label{sec:app_multiple}
\subsubsection{Testing procedure}\label{sec:multiple-app}

To describe our testing procedure for multiple change points, we first denote
\[
J := \Bigl\{(\ell,t): t = 2^{1 + \lceil \log_2(\log(n)) \rceil}, 2^{2+ \lceil \log_2(\log(n)) \rceil}, \dotsc, 2^{\lfloor \log_2(n) \rfloor - 1}, \ell=t, \dotsc,n-t \Bigr\}.
\]
Recall that in Section~\ref{subSec:robust_dense}, we use a median-of-means-type statistic $A_t^{\mathrm{MoM}}$ in \eqref{eq:A_mom} to determine whether there is a single change at or near each $t \in \mathcal{T}$ based on data $X_1, \ldots, X_t, X_{n-t+1}, \ldots X_n$. Here, we compute the same statistic using $X_{\ell-t+1},\ldots, X_{\ell}, X_{\ell+1}, \ldots, X_{\ell+t}$ for each pair $(\ell, t) \in J$,  i.e.\ $Z_{i,\ell,t} := (X_{\ell-t+i} - X_{\ell+t+1-i})/\sqrt{2}$ for $i \in [t]$.  We then split $\{Z_{1,\ell,t}, Z_{2, \ell, t}, \dotsc,Z_{t,\ell,t}\}$ into $G$ groups of equal size, with $G$ specified later in \Cref{thm:h0_vs_h2-app}.  Let $V_{g,\ell,t} \in \mathbb{R}^p$ be with $j$-th coordinate $V_{g,\ell,t}(j) := \overline{Z}_{g,\ell,t}^2(j) - G/t$, $j \in [p]$ and $g \in [G]$, where $\overline{Z}_{g,\ell,t}$ is the sample mean of the $g$-th group. We then have that
\begin{equation*}
    T_{\ell,t}:= \mathbbm{1}_{\left\{ A_{\ell, t}^{\mathrm{MoM}} > r \right\}},
\end{equation*}
where
\begin{equation*}
    A_{\ell, t}^{\mathrm{MoM}}:= t \cdot \mathrm{median} \Biggl(  \sum_{j=1}^p V_{1, \ell, t}(j),  \sum_{j=1}^p V_{2, \ell, t}(j), \dotsc, \sum_{j=1}^p V_{G_t, \ell, t}(j)\Biggl),
\end{equation*}
and the threshold $r$ is specified in \Cref{thm:h0_vs_h2-app}. Our test for the multiple change points case is
\begin{equation}\label{eq:h0h2_test}
    \phi_{\mathcal{P},\mathrm{multi}}^{\mathrm{MoM}} := \max_{(\ell, t) \in J} T_{\ell, t}.
\end{equation}

\subsubsection{Proofs of Theorem~\ref{thm:h0_vs_h2} and Proposition~\ref{prop:h0_h2}}

We prove the following theorem on the theoretical guarantee of the test $\phi_{\mathcal{P},\mathrm{multi}}^{\mathrm{MoM}}$, constructed above in \Cref{sec:multiple-app}. Theorem~\ref{thm:h0_vs_h2} then follows as an immediate consequence.

\begin{thm}\label{thm:h0_vs_h2-app} 
Assume $n \geq 50$ and $\alpha \geq 4$. For any $\varepsilon \in (0,1)$, there exist $C_1, C_2> 0$ depending only on $\alpha$, $K$ and $\varepsilon$, such that the test $\phi_{\mathcal{P},\mathrm{multi}}^{\mathrm{MoM}}$ defined in~\eqref{eq:h0h2_test} with
\begin{align*}
 r =  C_1\sqrt{p} G \quad \text{and} \quad G = 2^{1 + \lceil \log_2(\log(n)) \rceil}
\end{align*}
satisfies that
\[
\mathcal{R}_{\mathcal{P}, \mathrm{multi}}(\rho, \phi_{\mathcal{P},\mathrm{multi}}^{\mathrm{MoM}}) \leq \varepsilon,
\]
as long as $\rho^2 \geq C_2 v_{\mathcal{P}}^{\mathrm{U}}$, where 
\[
v_{\mathcal{P}}^{\mathrm{U}} := \sqrt{p} \log(n).
\]
\end{thm}

\begin{proof}
    \textbf{Null term.} For any $\theta \in \Theta_{0}(p,n)$, we have, by a union bound
    \begin{align*}
    \mathbb{E}_{\theta}\phi_{\mathcal{P},\mathrm{multi}}^{\mathrm{MoM}} &= \mathbb{P}_{\theta} \Bigl( \max_{(\ell, t) \in J} T_{\ell, t} = 1 \Bigr)\leq \sum_{(\ell,t) \in J } \mathbb{P}_\theta(T_{\ell,t} = 1). 
    \end{align*}
    Note that, for a fixed pair $(\ell, t)$, the test variable $T_{\ell,t}$ is constructed in the exact same way as the MoM test $\phi_{\mathcal{P}, \mathrm{dense}}$ we studied in Section~\ref{subSec:robust_dense}. Therefore, by following the proof of \Cref{thm:finitemoment_upperbound_dense} and, in particular,~\eqref{Eq:binomial_bound_robust}, we have 
    \[
    \mathbb{P}_\theta(T_{\ell,t} = 1) = \mathbb{P}_\theta \bigl(A_{\ell, t}^{\mathrm{MoM}} > r \bigr)  \leq \exp\biggl\{-\frac{G}{2}\log(6/\varepsilon)\biggr\},
    \]
    when we choose $r = C_1 \sqrt{p}G$ for some sufficiently large constant $C_1$ that depends on $\varepsilon$. Now, with $G = 2^{1 + \lceil \log_2(\log(n)) \rceil} \geq 2\log(n)$ and $n \geq 50$, we have for all $(\ell, t) \in J$ that
    \[
    \mathbb{P}_\theta(T_{\ell,t} = 1) \leq \exp\biggl\{-\frac{G}{2}\log(6/\varepsilon)\biggr\} \leq \exp\Bigl\{-\frac{4}{3}\log(n) - \log(6/\varepsilon)\Bigr\} = \frac{\varepsilon}{6n^{4/3}},
    \]
    where the second inequality is derived from the calculation $\bigl(\log(n)-1\bigr)\bigl(\log(6/\varepsilon)-4/3\bigr) \geq 4/3$. Therefore, as long as $n \geq 50$, we conclude
    \[
    \sum_{(\ell,t) \in J } \mathbb{P}_\theta(T_{\ell,t} = 1) \leq \frac{\varepsilon}{6n^{4/3}}|J| \leq \frac{\varepsilon}{6n^{4/3}} n\log_2(n/2) \leq  \frac{\varepsilon}{4}.  
    \]
    \textbf{Alternative term.} For any $\theta \in \Theta_{\mathrm{multi}}(p,n,\rho)$, by definition, there exists an $i^* \in \mathbb{Z}^+$ such that $\Delta_{i^*} \geq 4\log(n)$ and $\kappa_{i^*}^2 \Delta_{i^*} \geq \rho^2$. Consequently, there exists a corresponding $(\tau_{i^*}, t^*) \in J$ such that $\Delta_{i^*}/2 \leq t^* \leq \Delta_{i^*}$. Then, we have
    \begin{align*}
    \mathbb{E}_{\theta}\bigl(1-\phi_{\mathcal{P},\mathrm{multi}}^{\mathrm{MoM}}\bigr) &= \mathbb{P}_{\theta} \Bigl( \max_{(\ell, t) \in J} T_{\ell, t} = 0 \Bigr)  \leq \mathbb{P}_{\theta}\Bigl(T_{\tau_{i^*}, t^*} = 0\Bigr).
    \end{align*}
    Since we use the same type of test statistic as $\phi_{\mathcal{P}, \mathrm{dense}}$, and that $\{Z_{1,\tau_{i^*}, t^*}, \dotsc, Z_{t^*, \tau_{i^*}, t^*}\}$ are independent random vectors each with mean $(\mu_1 - \mu_2)/\sqrt{2}$, we can again follow the proof of \Cref{thm:finitemoment_upperbound_dense} to obtain 
    \[
    \mathbb{P}_{\theta} \Big(T_{\tau_{i^*}, t^*} = 0\Big) \leq \exp\biggl\{-\frac{\varepsilon G}{12} \biggl( \frac{6}{\varepsilon }\log\Bigl(\frac{6}{\varepsilon}\Bigr) - \frac{6}{\varepsilon} + 1\biggr)  \biggr\} \leq \frac{\varepsilon}{4},
    \]
    provided that $t^*\|\mu_1-\mu_2\|_2^2/2 \geq \rho^2/4 \geq \frac{C_2}{4}\sqrt{p}\log(n)$ for some sufficiently large constant $C_2$ that depends on $\varepsilon$ and this completes the proof.
\end{proof}

\begin{proof}[Proof of \Cref{prop:h0_h2}]

Throughout this proof, we take $P_e = N^\otimes (0,1)$. Recall that $N^\otimes (0,1)$ denotes the joint distribution of all $pn$ independent $N(0,1)$ entries in $E$. Following the calculation in the lower bound proof of claim (i) in Section~\ref{proof:lowerbounds}, we confirm that $P_e \in \mathcal{P}_{\alpha,K}^\otimes$ for all $\alpha \geq 2$ and $K \geq \sqrt{\alpha+1}$. We write $P_{\theta,E}$ to denote the distribution of $\theta + E$ where $\theta \in \mathbb{R}^{p \times n}$, and $E \in \mathbb{R}^{p \times n} $ is a matrix with entries being i.i.d standard normal random variables. We also write $P_{\theta \sim \pi, E}$ to denote the distribution of $\theta +E$ when $\theta \sim \pi$.

This result essentially follows from Theorem 2 in \cite{pilliat2020optimalhighdim}. 
    To lower bound the minimax testing error, we have
    \begin{align*}
        \mathcal{R}_{\mathcal{\mathcal{P}}, \mathrm{multi}}(\rho) &= \inf_{\phi \in \Phi} \biggl\{ \sup_{P_e \in \mathcal{P}_{\alpha,K}^\otimes}\sup_{\theta \in \Theta_0(p,n)}\mathbb{E}_{\theta, P_e} (\phi) + \sup_{P_e \in \mathcal{P}_{\alpha,K}^\otimes}\sup_{\theta \in \Theta_{\mathrm{multi}}(p,n,\rho)}\mathbb{E}_{\theta, P_e} (1-\phi) \biggr\} \\
        & \geq 1-\mathrm{TV}(P_{0,E}, P_{\theta \sim \pi, E}),
    \end{align*}
    for any $\pi$ that has support only on $\theta$ with two change points such that $\min\{\Delta_1, \Delta_2\} \geq 4\log(n)$. The constructions of $\tilde{\Theta}_{(2)}$ and $\tilde{\Theta}_{(3)}$ in Case 2 and Case 3 in the proof of Theorem 2 in \cite{pilliat2020optimalhighdim} with $r = 4\log(n)$ both have support  on $\theta$ with two change points such that $\min\{\Delta_1, \Delta_2\} \geq 4\log(n)$. The remaining calculation therein shows that when $n/4 \geq \lceil 4\log(n) \rceil$ and
    \[
    \rho^2 \leq c\big\{\sqrt{p\log(n)}+\log(n)\big\},
    \]
    for some sufficiently small constant $c>0$, we have $\mathrm{TV}(P_{0,E}, P_{\theta \sim \pi, E}) \leq 1/2$.
\end{proof}

\paragraph*{A heuristic example for requiring minimum spacing}

We provide a heuristic example illustrating why the heavy-tailed nature of the data necessitates a minimum spacing condition to achieve a rate with logarithmic dependence on~$n$.

Consider independent univariate random variables $U_1,\ldots, U_n \in \mathbb{R}$ with $\mathbb{E}U_i = \mu_i$ for $i \in [n]$. For $\alpha \geq 2$ and $K < \infty$, we assume that $\mathbb{E}\bigl[ \bigl(|U_i-\mu_i|/K\bigr)^\alpha\bigr] \leq 1$ for all $i$. Now, consider the testing problem
\[
\mathrm{H}_0: \mu_1 = \cdots = \mu_n = 0 \quad \text{vs.} \quad
\mathrm{H}_1: \exists j \in [n] \text{ s.t. } \mu_i = 0 \text{ for all } i \neq j \text{ and } |\mu_j| \geq \rho.
\]
This can be viewed as testing for a single outlier caused by a mean shift at an unknown time, which corresponds to two change points with \textit{no minimum spacing requirement}.

A natural test to consider is  
\[
\max_{i \in [n]} \mathbbm{1}_{\{|U_i| \geq r\}},
\]
for some threshold $r$. Suppose we want to control the Type I error probability at some $\varepsilon > 0$. Under the null, applying a union bound and Chebyshev's inequality, we obtain the upper bound  
\begin{align*}
    \text{Type I error prob.} \leq \sum_{i=1}^n \mathbb{P}(|U_i| \geq r) \leq \frac{n}{(r/K)^\alpha}.
\end{align*}
This is at most $\varepsilon$ when $r \geq K(n/\varepsilon)^{1/\alpha}$. To distinguish the null from the alternative, the signal strength $\rho$ must exceed this threshold, implying that $\rho$ must be at least of order $n^{1/\alpha}$,~i.e. polynomial in $n$.

\subsection{Temporal dependence}\label{sec:proof-temporal}
\subsubsection{Testing procedure}\label{sec:temporal-app}
Denote $Z_i := \bigl(X_i - X_{n+1-i}\bigr)/\sqrt{2}$ for $i \in [n/2]$. For $t\in \mathcal{T}$, we split $\{Z_1, \dotsc, Z_{t}\}$ into $G_t$ groups of equal size that
\begin{align*}
    \mathcal{Z}_{t,1} := \Bigl\{Z_1, \ldots, Z_{\frac{t}{G_t}}\Bigr\}, \quad  \mathcal{Z}_{t,2} := \Bigl\{Z_{\frac{t}{G_t}+1}, \ldots, Z_{\frac{2t}{G_t}}\Bigr\},\quad \dotsc, \quad  \mathcal{Z}_{t,G_t} := \Bigl\{Z_{\frac{(G_t-1)t}{G_t}+1}, \ldots, Z_{t}\Bigr\}.
\end{align*}
The procedure to form the test remains the same as in~\eqref{eq:V_tg(j)} and~\eqref{eq:A_mom} from Section~\ref{subSec:robust_dense}, except that we replace $G_t/t$ with $\mathbb{E}\overline{Z}^2_{t,g}(j)$ in $V_{t,g}(j)$.
Our test is
\begin{equation}\label{eq:finitemoment_dense_test_temp}
\phi_{\mathcal{P}_\mathrm{Temp}} := \mathbbm{1}_{\left\{\max_{t \in \mathcal{T}} A_{t}^{\mathrm{MoM}}/r_t^{\mathrm{Temp}} > 1 \right\}},
\end{equation}
with the detection threshold $r_t^{\mathrm{Temp}}$ specified in \Cref{thm:temporal-app}. We assume $\mathbb{E}\overline{Z}^2_{t,g}(j)$ to be known, though this can be relaxed if this quantity can be estimated reasonably well. We discuss this aspect in detail and provide a specific example where the noise is generated by a moving average process in \Cref{sec:ma1}.

\subsubsection{Proof of Theorem~\ref{thm:temporal}} \label{sec:proof-temporal-app}

We prove the following theorem on the theoretical guarantee of the test $\phi_{\mathcal{P}_\mathrm{Temp}}$, constructed above in \Cref{sec:temporal-app}. Theorem~\ref{thm:temporal} then follows as an immediate consequence.

\begin{thm}\label{thm:temporal-app}
Assume $\alpha > 4$. For any $\varepsilon \in (0,1)$, there exist $C_1, C_2, C_3 > 0$ depending only on~$\alpha$,~$K$,~$c_1$,~$c_2$ and $\varepsilon$, such that the test $\phi_{\mathcal{P}_\mathrm{Temp}}$ defined in~\eqref{eq:finitemoment_dense_test_temp} with
\begin{align*}
 r_t^{\mathrm{Temp}} =  C_1p^{1/2}G_t, \quad G_t = t \wedge \Delta \quad  \mbox{and} \quad \Delta =  2^{ \lceil \log_2 (C_2 (\log \log (8n)) (\log \log \log (16n))^2 ) \rceil},
\end{align*}
satisfies that
\[
\mathcal{R}_{\mathcal{P}_{\mathrm{Temp}}}(\rho, \phi_{\mathcal{P}_\mathrm{Temp}}) \leq \varepsilon,
\]
as long as $\rho^2 \geq C_3 v_{\mathcal{P}_{\mathrm{Temp}}}^{\mathrm{U}}$, where 
\[
v_{\mathcal{P}_{\mathrm{Temp}}}^{\mathrm{U}} := p^{1/2} (\log \log (8n)) (\log \log \log (64n))^2.
\]
\end{thm}
\begin{proof}
If we define $\{\tilde{E}_i\}_{i \in [n]}$ as the reordered sequence of $\{E_i\}_{i \in [n]}$ with
\begin{equation} \label{Eq:noise_reorder}
\tilde{E}_i = \begin{cases}
E_{(i+1)/2} \quad \text{for odd $i$}, \\
E_{n+1-i/2} \quad \text{for even $i$},
\end{cases}
\end{equation}
then we verify that the (usual) $\alpha$-mixing coefficient of this reordered process $\{\tilde{E}_i\}_{i \in [n]}$ satisfies
\begin{align} \label{Eq:alpha_mix_reorder}
    \alpha_{\mathrm{pa}}(i) &:= \sup_{\ell \in [n-i]} \sup_{A \in \sigma(\tilde{E}_j: 1 \leq j \leq \ell), B \in \sigma(\tilde{E}_j: \ell+i \leq j \leq n)} \bigl|\mathbb{P}(A \cap B ) - \mathbb{P}(A)\mathbb{P}(B)\bigr|  \nonumber \\
    &\leq \alpha^*(\lfloor i/2 \rfloor) \leq c_1e^{-c_2i/3},
\end{align}
for all $i \in [n-1]$. 

\noindent \textbf{Null term.} Similar to the proof of Theorem~\ref{thm:finitemoment_upperbound_dense}, we still have for all $t \in \mathcal{T}$ and $g \in [G_t]$ that
\begin{align} \label{Eq:MoM_robust_null_temp}
\mathbb{P}_\theta\biggl(t\sum_{j=1}^p V_{t,g}(j) > r_t^{\mathrm{Temp}}\biggr) &=  \mathbb{P}_\theta\biggl(\sum_{j=1}^p \biggl(\overline{Z}^2_{t,g}(j) - \mathbb{E}\overline{Z}^2_{t,g}(j) \biggr) > \frac{C_1\sqrt{p}G_t}{t}\biggr) \leq \frac{t^2\sum_{j=1}^p\mathbb{E}_\theta \overline{Z}^4_{t,g}(j)}{C_1^2pG_t^2} \nonumber \\
&\leq \frac{pt^2 C_4 (G_t/t)^2}{C_1^2pG_t^2}  = \frac{C_4}{C_1^2}.
\end{align}
for some $C_4>0$, depending on $\alpha, K$ and $c$. The second inequality above now follows from \citet[][Theorem~4.1]{ShaoYu1996}. According to the reordered sequence $\{\tilde{E}_i\}_{i\in [n]}$ defined in \eqref{Eq:noise_reorder}, we have, for $g \in [G_t]$
\begin{align*}
\mathbbm{1}_{\bigl\{ t\sum_{j=1}^p V_{t,g}(j) > r_t^{\mathrm{Temp}} \bigr\}} \in \sigma\Bigl(\tilde{E}_i: \frac{2(g-1)t}{G_t}+1 \leq i \leq \frac{2gt}{G_t}\Bigr).
\end{align*}
Thus, as long as $C_4/C_1^2 \leq 1/4$ in \eqref{Eq:MoM_robust_null_temp}, by \eqref{Eq:alpha_mix_reorder} and \citet[][Theorem~1]{merlevede2009}\footnote{Theorem 1 in the cited work assumes that the $\alpha$-mixing coefficient is bounded by $\alpha(i) \leq e^{-ci}$ for some $c > 0$. A closer examination of the proof reveals that this result remains valid even if the bound is relaxed to $c'e^{-ci}$ for some constants $c, c' > 0$.}, we have for all $t \in \mathcal{T}$ with $G_t \geq 4$
\begin{align} \label{Eq:binomial_bound_robust_temp}
&\mathbb{P}_\theta \Bigl(A_t^{\mathrm{MoM}} > r_t^{\mathrm{Temp}}\Bigr) \leq  \mathbb{P}_\theta\Biggl( \sum_{g=1}^{G_t} \mathbbm{1}_{\bigl\{ t\sum_{j=1}^p V_{t,g}(j) > r_t^{\mathrm{Temp}} \bigr\}}  \geq G_t/2\Biggr) \leq \exp\biggl\{-\frac{C_5G_t}{\log G_t \log \log G_t}\biggr\},
\end{align}
for some $C_5 > 0$, depending only on $c_1$ and $c_2$, the constants in the interlaced $\alpha$-mixing condition \eqref{Eq:alpha*_mixing}. To this end, we observe that there exists sufficiently large $C_2 > 0$, depending only on $C_5$ and $\varepsilon$, such that for all $t \geq C_2/12$ and for our choice $\Delta \geq  C_2 (\log \log(8n)) (\log \log \log(64n))^2 > C_2/12$, we have
\begin{align} 
\label{Eq:bound_middle_t_temp}
\exp\biggl\{-\frac{C_5t}{\log t \log \log t}\biggr\} \leq \frac{\varepsilon}{10}e^{-\sqrt{t}},
\end{align}
and
\begin{align}
\label{Eq:bound_Delta_temp}
\exp\biggl\{-\frac{C_5\Delta}{\log \Delta \log \log \Delta}\biggr\} \leq \frac{\varepsilon}{10\log_2(n/2)},
\end{align}
Thus, by choosing $C_1 > \sqrt{C_2^2C_4/\varepsilon}$ and combining~\eqref{Eq:MoM_robust_null_temp},~\eqref{Eq:binomial_bound_robust_temp},~\eqref{Eq:bound_middle_t_temp},~\eqref{Eq:bound_Delta_temp} and a union bound, we conclude, when $C_2$ is sufficiently large, that
\begin{align*}
\mathbb{E}_\theta \phi_{\mathcal{P}_{\mathrm{Temp}}} &\leq \sum_{t \in \mathcal{T}: \, t<C_2/12} \sum_{g \in [t]}\mathbb{P}_\theta\biggl(t\sum_{j=1}^p V_{t,g}(j) > r_t^{\mathrm{Temp}}\biggr)  + \sum_{t \in \mathcal{T}: \, C_2/12 \leq t \leq \Delta} \mathbb{P}_\theta\bigl(A_t^{\mathrm{MoM}} > r_t^{\mathrm{Temp}} \bigr) \\
&\qquad + \sum_{t \in \mathcal{T}: \,  t > \Delta}   \mathbb{P}_\theta\bigl(A_t^{\mathrm{MoM}} > r_t^{\mathrm{Temp}} \bigr)  \\
&\leq \frac{C_2^2C_4}{12C_1^2} +  \sum_{t \in \mathcal{T}: \, C_2/12 \leq t \leq \Delta} \frac{\varepsilon}{10} e^{-\sqrt{t}} + \sum_{t \in \mathcal{T}: \,  t > \Delta} \frac{\varepsilon}{10\log_2(n/2)} 
\leq \varepsilon/12 + \varepsilon/10 + \varepsilon/10 < \varepsilon/2.
\end{align*}
for all $\theta \in \Theta_0(p,n)$.

\medskip
\noindent \textbf{Alternative term.} We again follow the proof of Theorem~\ref{thm:finitemoment_upperbound_dense} and reach
\begin{align} \label{Eq:MoM_main_alternative_temp}
&\mathbb{P}_\theta\biggl(\tilde{t}\sum_{j=1}^p V_{\tilde{t},g}(j) \nonumber \leq r_{\tilde{t}}^{\mathrm{Temp}}\biggr)  \\
&\leq \mathbb{P}_\theta\Biggl(\sum_{j=1}^p \biggl( \bigl(\overline{Z}'_{\tilde{t},g}(j)\bigr)^2 - \mathbb{E}\bigl( \overline{Z}'_{\tilde{t},g}(j)\bigr)^2 + \sqrt{2}\bigl(\mu_1(j) - \mu_2(j)\bigr)\overline{Z}'_{\tilde{t},g}(j) \biggr) \leq -\frac{\rho^2}{16\tilde{t}}-\frac{\|\mu_1-\mu_2\|_2^2}{4} \Biggr).
\end{align}
Again by \citet[][Theorem~4.1]{ShaoYu1996}, we have
\begin{align*}
    \mathbb{P}_\theta\biggl(\sum_{j=1}^p \Bigl( \bigl(\overline{Z}'_{\tilde{t},g}(j)\bigr)^2 - \mathbb{E}\bigl( \overline{Z}'_{\tilde{t},g}(j)\bigr)^2 \Bigr) \leq -\frac{\rho^2}{16\tilde{t}} \biggr) \leq \frac{256(\tilde{t})^2 \sum_{j=1}^p\mathbb{E}_\theta \bigl(\overline{Z}'_{\tilde{t},g}(j)\bigr)^4 }{\rho^4} \leq \frac{C_6pG^2_{\tilde{t}}}{\rho^4},
\end{align*}
and
\begin{align*}
    \mathbb{P}_\theta\biggl(\sum_{j=1}^p \sqrt{2}\bigl(\mu_1(j) - \mu_2(j)\bigr)\overline{Z}'_{\tilde{t},g}(j) \leq -\frac{\|\mu_1-\mu_2\|_2^2}{4} \biggr) &\leq \frac{32 \sum_{j=1}^p (\mu_1(j) - \mu_2(j))^2 \mathbb{E}_\theta \bigl(\overline{Z}'_{\tilde{t},g}(j)\bigr)^2 }{\|\mu_1-\mu_2\|_2^4} \\
    &\leq \frac{C_6G_{\tilde{t}}}{\rho^2},
\end{align*}
for some $C_6>0$, depending only on $\alpha$, $K$ and $c$. Combining these with~\eqref{Eq:MoM_main_alternative_temp}, as long as
\[
\rho^2 \geq C_3 p^{1/2} (\log \log (8n)) (\log \log \log (64n))^2,
\]
for sufficiently large $C_3$, we are guaranteed that
\[
\mathbb{P}_\theta\biggl(\tilde{t}\sum_{j=1}^p V_{\tilde{t},g}(j) \leq r_{\tilde{t}}^{\mathrm{Temp}}\biggr) \leq \frac{C_4}{C_1^2} < \frac{\varepsilon}{12}.
\]
Note that 
\[
\mathbb{E}_\theta(1-\phi_{\mathcal{P}_{\mathrm{Temp}}}) \leq \mathbb{P}_\theta \bigl(A_{\tilde{t}}^{\mathrm{MoM}} \leq r_{\tilde{t}}^{\mathrm{Temp}}\bigr) \leq \mathbb{P}_\theta\Biggl( \sum_{g=1}^{G_{\tilde{t}}} \mathbbm{1}_{\bigl\{ \tilde{t}\sum_{j=1}^p V_{\tilde{t},g}(j) \leq r_{\tilde{t}}^{\mathrm{Temp}} \bigr\}}  \geq G_{\tilde{t}}/2\Biggr).
\]
We consider three cases separately: (i) $\tilde{t} < C_\varepsilon$, (ii)  $C_\varepsilon \leq \tilde{t} < \Delta$, and (iii) $\tilde{t} \geq \Delta$. Using the same argument as in controlling the null term, we can show that $\mathbb{E}_\theta(1-\phi_{\mathcal{P}_{\mathrm{Temp}}}) \leq \varepsilon/2$ and this completes the proof for $\alpha > 4$.
\end{proof}

\begin{remark}\label{rmk:temporal}
A similar result can be obtained for $2 < \alpha \leq 4$. In this case, \Cref{thm:temporal-app} holds with $r_t^{\mathrm{Temp}} =  C_1p^{2/\alpha'}G_t$ and $v_{\mathcal{P}_{\mathrm{Temp}}}^{\mathrm{U}} = p^{2/\alpha'} (\log \log (8n)) (\log \log \log (64n))^2$ for any $\alpha' < \alpha$, with $G_t$ and $\Delta$ remaining unchanged.  
\end{remark}

\subsubsection{An MA(1) example}
\label{sec:ma1}

When forming the test $\phi_{\mathcal{P}_{\mathrm{Temp}}}$, we assume $\mathbb{E}\overline{Z}^2_{t,g}(j)$ is known for all $t \in \mathcal{T}$, $g \in [G_t]$, and $j \in [p]$. This term also appears in the setting with independent observations, where it simplifies to $G_t/t$, given the assumption that the variance of each error term is $1$; see Definition~\ref{def-palphak}. A close examination of the proof of Theorem~\ref{thm:temporal-app} reveals that if, with high probability, $\sum_{j=1}^p \mathbb{E}\overline{Z}^2_{t,g}(j)$ can be accurately estimated with an error of $O(\sqrt{p}G_t /t)$ for all $t \in \mathcal{T}$, $g \in [G_t]$, and $j \in [p]$, then the conclusion of Theorem~\ref{thm:temporal-app} still holds for the modified test $\phi_{\mathcal{P}_{\mathrm{Temp}}}$, where the exact expectations $\mathbb{E}\overline{Z}^2_{t,g}(j)$ are replaced by their estimators.

We consider an example of a specific temporal dependence model. Assume that $E$ has independent component series, and for each $j \in [p]$, the $j$-th component series $\{E_i(j)\}_{i \in [n]}$ follows a moving average process of order 1 (MA(1))
\begin{equation} \label{Eq:MA1_mechanism}
    E_i(j) = \omega_i(j) + \pi_{\mathrm{ma}}\omega_{i-1}(j),
\end{equation}
where $\{\omega_i(j)\}_{i=0,1,\ldots}$ is an independent white noise sequence satisfying $\mathbb{E}[\omega_i(j)^2] = (1 + \pi^2_{\mathrm{ma}})^{-1}$ and $\mathbb{E}\bigl|\omega_i(j) + \pi_{\mathrm{ma}}\omega_{i-1}(j)\bigr|^\alpha \leq K^\alpha$ for all $i$. The lag-1 autocorrelation is given by $r_1 := \pi_{\mathrm{ma}}/(1+ \pi^2_{\mathrm{ma}})$. As discussed in \Cref{sec:temporal}, the interlaced $\alpha$-mixing coefficient of our noise sequence $\{E_i\}_{i \in [n]}$ satisfies $\alpha^*(i) \leq e^{1-i}$ for all $i \in [n-1]$. Now, if we can estimate $r_1$ well, then a plug-in estimator for $\mathbb{E}\overline{Z}^2_{t,g}(j)$ can be used. We formalise this in the following corollary.

\begin{cor} \label{cor:MAresult}
Assume $\alpha > 4$ and let $\varepsilon \in (0,1)$. Consider the MA(1) data-generating mechanism for the noise sequence described by~\eqref{Eq:MA1_mechanism} and in the last paragraph. Let $\hat{r}_1$ be any estimator for $r_1$ satisfying
\begin{equation} \label{Eq:con_on_lag1_est}
    \mathbb{P}\bigl(| \hat{r}_1 - r_1| > cp^{-1/2}\bigr) \leq \varepsilon/2,
\end{equation}
for some $c > 0$, depending on $\varepsilon$. Then, after modifying $V_{t, g}$ to become
\[
V_{t,g}(j) :=   \overline{Z}^2_{t,g}(j) - \frac{(t/G_t) + 2\bigl\{(t/G_t)-1\bigr\} \hat{r}_1 }{(t/G_t)^2},
\]
the theoretical guarantee on $\phi_{\mathcal{P}_{\mathrm{Temp}}}$ in Theorem~\ref{thm:temporal-app} remains valid, with possibly increased values of $C_1, C_2$ and $C_3$.
\end{cor}
Suppose we have a (historical) dataset $X_1^{(h)}, \ldots, X_m^{(h)}$ of size $m \in \mathbb{Z}^+$, within which no change point is present. Each observation can be written as $X_i^{(h)} = \mu_0  +E_i^{(h)}$, where the noise follows the data-generating mechanism described in the previous paragraph. The lag-1 autocorrelation can be estimated by
\[
\hat{r}_1 :=\frac{1}{(m-1)p} \sum_{j=1}^p \sum_{i=1}^{m-1} \biggl(X_i^{(h)}(j) - \frac{1}{m}\sum_{k=1}^m X_k^{(h)}(j) \biggr) \biggl(X_{i+1}^{(h)}(j) - \frac{1}{m}\sum_{k=1}^m X_k^{(h)}(j) \biggr).
\]
This estimator can be shown to satisfy condition~\eqref{Eq:con_on_lag1_est} when, informally speaking, the sample size $m$ is significantly larger than $\sqrt{p}$.

Note that in \eqref{Eq:MA1_mechanism}, all components share the same lag-1 coefficient. If this assumption does not hold, then the requirement on the estimators becomes $\mathbb{P}\bigl( \sum_{j=1}^p | \hat{r}_1(j) - r_1(j)| > cp^{-1/2}\bigr) \leq \varepsilon/2$ in place of \eqref{Eq:con_on_lag1_est} in Corollary~\ref{cor:MAresult}, where $r_1(j)$ denotes the lag-1 autocorrelation for the $j$-th component series and $\hat{r}_1(j)$ its estimator.

\subsection{Fewer than two finite moments} \label{sec:proof-<2moments}
\subsubsection{Testing procedure}
\label{sec:test-2moment}

As mentioned in \Cref{sec:<2moment} of the main text, our test has two components. One component utilises a robust mean estimator $\hat{\mu}^{\mathrm{RM}}$ from \citet[][c.f. Algorithm 1-7]{cherapanamjeri2022optimal}, developed specifically for distributions satisfying \eqref{eq:weakmoment}.

\begin{prop}[\cite{cherapanamjeri2022optimal}]\label{prop:weakmoment}
    Let $1 \leq \alpha \leq 2$.  For $t \in \mathbb{Z}^+$, let $X_1, \dotsc, X_t$ be independent random vectors in $\mathbb{R}^p$ with mean $\mu$. Assume that the distribution of $W_i := X_i - \mu$ belongs to $\mathcal{W}_\alpha$ for each $i \in [t]$. Then, there exists a polynomial-time algorithm that, given inputs $X_1, \dotsc, X_t$ and $\eta > 0$, outputs $\hat{\mu}^{\mathrm{RM}}_t(\{X_i\}_{i=1}^t; \eta)$. There exist absolute constants $C_{01}, C_{02} > 0$ such that, for any $0 < \eta < 1$, when $t \geq C_{01}\log(1/\eta)$, with probability at least $1 - \eta$, it holds that
    \[
        \|\hat{\mu}^{\mathrm{RM}}_t(\{X_i\}_{i=1}^t; \eta) - \mu\|_2 \leq C_{02} \Biggl\{\sqrt{\frac{p}{t}} + \left(\frac{p}{t}\right)^{\frac{\alpha-1}{\alpha}} + \left(\frac{\log(1/\eta)}{t}\right)^{\frac{\alpha-1}{\alpha}}\Biggr\}.
    \]
\end{prop}

 Similar to the limitations of using $\hat{\mu}^{\mathrm{RSM}}$ as discussed in Section~\ref{subSec:robust_sparse}, test statistics based on $\hat{\mu}^{\mathrm{RM}}$ only have theoretical guarantee when the change is sufficiently away from the boundary, due to the condition $t \geq C_{01}\log(1/\eta)$ in  \Cref{prop:weakmoment}. To cover the case when the potential change occurs near the boundary, we need to adopt a different strategy that does not have this limitation on the minimal sample size.
 
Recall that $Z_i = (X_i - X_{n-i+1})/2$, for $i \leq n/2$. We let
\[
\begin{cases}
\tilde{A}_t = \|\sum_{i=1}^t Z_i/t\|_2, &t \in \mathcal{T} \cap \{ t \leq \tilde{\Delta}_1\}, \\
    \tilde{A}_t^{\mathrm{RM}} = \|\hat{\mu}^{\mathrm{RM}}_t(\{Z_i\}_{i=1}^t; \eta_t)\|_2, &t \in \mathcal{T} \cap \{ t > \tilde{\Delta}_1\},
\end{cases}
\]
and define the test
\begin{equation}\label{eq:weakmoment-test}
    \phi^{\mathrm{RM}}_{\mathcal{W}_{\alpha}} := \mathbbm{1}_{\left\{\max_{t \in \mathcal{T}\cap \{t \leq \tilde{\Delta}_1\}} \tilde{A}_{t}/\tilde{r}_t > 1 \right\}} \vee \mathbbm{1}_{\left\{\max_{t \in \mathcal{T}\cap \{t > \tilde{\Delta}_1\}} \tilde{A}^{\mathrm{RM}}_{t}/\tilde{r}^{\mathrm{RM}}_t > 1 \right\}},
\end{equation}
where $\tilde{\Delta}_1, \tilde{r}_t, \tilde{r}^{\mathrm{RM}}_t$ and $\eta_t$ are specified later in \Cref{thm:lessthan2-app}.

\subsubsection{Proof of Theorem~\ref{thm:lessthan2} and Proposition~\ref{prop:lowerbound-lessthan2}}\label{sec:<2proof}

We prove the following theorem on the theoretical guarantee of the test $\phi^{\mathrm{RM}}_{\mathcal{W}_{\alpha}}$, constructed above in \Cref{sec:test-2moment}. Theorem~\ref{thm:lessthan2} then follows as an immediate consequence.

\begin{thm}\label{thm:lessthan2-app}
Assume $1 \leq \alpha \leq 2$. For any $\varepsilon \in (0, 1)$, there exist $C_1,C_2,C_3,C_4 > 0$ depending only on $\alpha$ and $\varepsilon$, such that the test defined in~\eqref{eq:weakmoment-test} with
\begin{align*}
& \tilde{r}_t = C_1 \tilde{\Delta}_1^{\frac{2-\alpha}{2\alpha}}\log^{\frac{1}{\alpha}}(\tilde{\Delta}_1) \sqrt{\frac{p}{t}}, \\
&\eta_t = \exp\biggl\{-\frac{t \wedge \tilde{\Delta}_2}{C_2}\biggr\}, \quad \tilde{r}^{\mathrm{RM}}_t = C_3 \left(\sqrt{\frac{p}{t}}+\Big(\frac{p}{t}\Big)^{\frac{\alpha-1}{\alpha}} + \biggl(\frac{\log(1/\eta_t)}{t}\biggr)^{\frac{\alpha-1}{\alpha}}\right), \\
&\tilde{\Delta}_1 = C_2\log(16/\varepsilon) \quad \text{and} \quad \tilde{\Delta}_2 = C_2\log(16\log(2n)/\varepsilon),
\end{align*}
satisfies that
\[
\mathcal{R}_{\mathcal{W}_{\alpha}^{\otimes}}(\{\rho_{t_0}\}_{t_0 \in[n-1]}, \phi^{\mathrm{RM}}_{\mathcal{W}_{\alpha}}) \leq \varepsilon,
\]
as long as
\begin{equation}\label{eq:rate-less-than2-app}
    \rho_{t_0} \geq C_4 \biggl\{\sqrt{\frac{p}{t_0 \wedge (n-t_0)}}+\Big(\frac{p}{t_0 \wedge (n-t_0)}\Big)^{\frac{\alpha-1}{\alpha}} + \Big(\frac{\log\log(n)}{t_0 \wedge (n-t_0)}\Big)^{\frac{\alpha-1}{\alpha}}\biggr\},
\end{equation}
for all $t_0 \in [n-1]$.
\end{thm}

\begin{proof}
    We denote $\tilde{\mathcal{T}}_1 := \{t \in \mathcal{T}: t \leq \tilde{\Delta}_1\}$, $\tilde{\mathcal{T}}_2 := \{t \in \mathcal{T}: \tilde{\Delta}_1 < t \leq \tilde{\Delta}_2\}$ and $\tilde{\mathcal{T}}_3 := \{t \in \mathcal{T}: t > \tilde{\Delta}_2\}$.
    \textbf{Null term.} Under the null hypothesis, we need to control 
    \[
    \mathbb{E}_\theta \phi^{\mathrm{RM}}_{\mathcal{W}_{\alpha}}  = \sum_{t \in \tilde{\mathcal{T}}_1}\mathbb{P}_\theta (\tilde{A}_t > \tilde{r}_t) + \sum_{t \in \tilde{\mathcal{T}}_2} \mathbb{P}_\theta (\tilde{A}_t^{\mathrm{RM}} > \tilde{r}^{\mathrm{RM}}_t) + \sum_{t \in \tilde{\mathcal{T}}_3} \mathbb{P}_\theta (\tilde{A}_t^{\mathrm{RM}} > \tilde{r}^{\mathrm{RM}}_t).
    \]
    We first control the first term. Notice that $Z_i$ are independent random vectors satisfying \eqref{eq:weakmoment}, since $\mathbb{E}_\theta Z_i = 0$ and
    \[
    \mathbb{E}_\theta|\langle Z_i, v\rangle|^{\alpha} = \frac{1}{2^{\alpha}}\mathbb{E}_\theta|\langle X_i-X_{n-i+1}, v\rangle|^{\alpha} \leq \mathbb{E}_\theta |\langle X_i-\mu, v\rangle|^{\alpha} \leq 1,
    \]
    for any unit vector $v$. Thus, we have by \Cref{lemma:weak_mean}
    \[
    \mathbb{P}_\theta (\tilde{A}_t > \tilde{r}_t) \leq \frac{\pi}{C_1^{\alpha} (\tilde{\Delta}_1 / t)^{\frac{2-\alpha}{2}}\log(\tilde{\Delta}_1)}.
    \]
Now, by setting $C_1 \geq (16\pi/\varepsilon)^{1/\alpha}$, we are guaranteed to have
\[
\sum_{t \in \tilde{\mathcal{T}}_1}\mathbb{P}_\theta (\tilde{A}_t > \tilde{r}_t) \leq \varepsilon/4.
\]
Now, for the second and third terms, we note that $t \geq C_2\log(1/\eta_t)$. Thus, by using \Cref{prop:weakmoment} with $\eta = \eta_t$, we have $\mathbb{P}(\tilde{A}_t^{\mathrm{RM}} > \tilde{r}^{\mathrm{RM}}_t) \leq \eta_t$ and thus
\begin{align*}
\sum_{t \in \tilde{\mathcal{T}}_2} \mathbb{P}_\theta (\tilde{A}_t^{\mathrm{RM}} > \tilde{r}^{\mathrm{RM}}_t) + \sum_{t \in \tilde{\mathcal{T}}_3} \mathbb{P}_\theta (\tilde{A}_t^{\mathrm{RM}} > \tilde{r}^{\mathrm{RM}}_t) &\leq \sum_{t \in \tilde{\mathcal{T}}_2} \exp\biggl\{-\frac{t}{C_2}\biggr\} + \sum_{t \in \tilde{\mathcal{T}}_3} \frac{\varepsilon}{16\log(2n)} \\
&\leq 2 \exp \biggl\{-\frac{\tilde{\Delta}_1}{C_2}\biggr\}+\frac{\varepsilon \log _2(n/2)}{16 \log(2n)}<\varepsilon/4.
\end{align*}
We therefore conclude $\mathbb{E}_\theta \phi^{\mathrm{RM}}_{\mathcal{W}_{\alpha}}  < \varepsilon/2$.

\vspace{1em}

\noindent\textbf{Alternative term.} Consider, as before, the point $\tilde{t} \in \mathcal{T}$ such that $t_0/2< \tilde{t} \leq t_0 \leq n/2$. We shall deal with the case $t_0 > n/2$ later. Note that $Z_1,\dotsc,Z_{\tilde{t}}$ are independent with mean $(\mu_1-\mu_2)/2$ and that $Z_i - (\mu_1-\mu_2)/2$ satisfies \eqref{eq:weakmoment} for each $i$. We first consider the case $\tilde{t} \leq \tilde{\Delta}_1$. In this case, we have 
\begin{equation}\label{eq:weak-type-ii-1}
    \mathbb{P}_\theta (\tilde{A}_{\tilde{t}} < \tilde{r}_{\tilde{t}}) = \mathbb{P}_\theta \Biggl(\biggl\|\sum_{i=1}^{\tilde{t}} Z_i/\tilde{t}\biggr\|_2 < \tilde{r}_{\tilde{t}}\Biggr) \leq \mathbb{P}_\theta \Biggl(\Biggl|\biggl\|\frac{\mu_1-\mu_2}{2}\biggr\|_2-\biggl\|\sum_{i=1}^{\tilde{t}} Z_i/\tilde{t} - \frac{\mu_1-\mu_2}{2}\biggr\|_2\Biggr|<\tilde{r}_{\tilde{t}}\Biggr).
\end{equation}
By \Cref{lemma:weak_mean} and the choice of $\tilde{r}_{\tilde{t}}$ with $C_1 \geq (16\pi/\varepsilon)^{1/\alpha}$, we have
\[
\mathbb{P}_\theta \Biggl(\biggl\|\sum_{i=1}^{\tilde{t}} Z_i/\tilde{t} - \frac{\mu_1-\mu_2}{2}\biggr\|_2 > \tilde{r}_{\tilde{t}}\Biggr) \leq \frac{\varepsilon}{16},
\]
Therefore, as long as 
\[
\biggl\|\frac{\mu_1-\mu_2}{2}\biggr\|_2 \geq 4C_1 \tilde{\Delta}_1^{\frac{2-\alpha}{2\alpha}}\log^{\frac{1}{\alpha}}(\tilde{\Delta}_1) \sqrt{\frac{p}{t_0}},
\]
we obtain $\bigl\|\frac{\mu_1-\mu_2}{2}\bigr\|_2 \geq 2\tilde{r}_{\tilde{t}}$ and thus $\mathbb{E}_\theta (1-\phi^{\mathrm{RM}}_{\mathcal{W}_\alpha})\leq \mathbb{P}_\theta (\tilde{A}_{\tilde{t}} < r_{\tilde{t}}) \leq \varepsilon/16$. In the case of $t_0 > n/2$, we should consider instead the point $\tilde{t} \in \mathcal{T}$ such that $(n-t_0)/2< \tilde{t} \leq n-t_0 \leq n/2$, and the same arguments as above show that $\mathbb{E}_\theta (1-\phi^{\mathrm{RM}}_{\mathcal{W}_\alpha})\leq \mathbb{P}_\theta (\tilde{A}_{\tilde{t}} < r_{\tilde{t}}) \leq \varepsilon/16$ as long as 
\[
\biggl\|\frac{\mu_1-\mu_2}{2}\biggr\|_2 \geq 4C_1 \tilde{\Delta}_1^{\frac{2-\alpha}{2\alpha}}\log^{\frac{1}{\alpha}}(\tilde{\Delta}_1) \sqrt{\frac{p}{n-t_0}}.
\]

Now, consider the case $\tilde{t} > \tilde{\Delta}_1$, when $t_0 \leq n/2$. Similar to~\eqref{eq:weak-type-ii-1}, we now have
\begin{align*}
\mathbb{P}_\theta (\tilde{A}_{\tilde{t}}^{\mathrm{RM}} < \tilde{r}^{\mathrm{RM}}_{\tilde{t}}) &= \mathbb{P}_\theta \biggl(\Bigl\|\hat{\mu}^{\mathrm{RM}}_{\tilde{t}}(\{Z_i\}_{i=1}^{\tilde{t}}; \eta_{\tilde{t}})\Bigr\|_2 < \tilde{r}^{\mathrm{RM}}_{\tilde{t}} \biggr) \\
&\leq \mathbb{P}_\theta \Biggl(\Biggl|\biggl\|\frac{\mu_1-\mu_2}{2}\biggr\|_2-\biggl\| \hat{\mu}^{\mathrm{RM}}_{\tilde{t}}(\{Z_i\}_{i=1}^{\tilde{t}}; \eta_{\tilde{t}}) - \frac{\mu_1-\mu_2}{2}\biggr\|_2\Biggr|< \tilde{r}^{\mathrm{RM}}_{\tilde{t}} \Biggr).
\end{align*}
Thus, as long as 
\[
\biggl\|\frac{\mu_1-\mu_2}{2}\biggr\|_2 \geq 4C_3 \biggl(\sqrt{\frac{p}{t_0}}+\Big(\frac{p}{t_0}\Big)^{\frac{\alpha-1}{\alpha}} + \Big(\frac{\log\log(n)}{t_0}\Big)^{\frac{\alpha-1}{\alpha}}\biggr),
\]
we have $\mathbb{P}_\theta (\tilde{A}_{\tilde{t}}^{\mathrm{RM}} < \tilde{r}^{\mathrm{RM}}_{\tilde{t}}) \leq \varepsilon/16$. In the case of $t_0 > n/2$, we should consider instead the point $\tilde{t} \in \mathcal{T}$ such that $(n-t_0)/2< \tilde{t} \leq n-t_0 \leq n/2$, and the same arguments as above show that  $\mathbb{P}_\theta (\tilde{A}_{\tilde{t}}^{\mathrm{RM}} < \tilde{r}^{\mathrm{RM}}_{\tilde{t}}) \leq \varepsilon/16$ as long as 
\[
\biggl\|\frac{\mu_1-\mu_2}{2}\biggr\|_2 \geq 4C_3 \biggl(\sqrt{\frac{p}{n-t_0}}+\Big(\frac{p}{n-t_0}\Big)^{\frac{\alpha-1}{\alpha}} + \Big(\frac{\log\log(n)}{n-t_0}\Big)^{\frac{\alpha-1}{\alpha}}\biggr). 
\]
Together, we obtain $\mathbb{E}_\theta (1-\phi^{\mathrm{RM}}_{\mathcal{W}_\alpha}) \leq \varepsilon/16$, as long as 
\[
\rho_{t_0} \geq C_4 \biggl(\sqrt{\frac{p}{t_0 \wedge (n-t_0)}}+\Big(\frac{p}{t_0 \wedge (n-t_0)}\Big)^{\frac{\alpha-1}{\alpha}} + \Big(\frac{\log\log(n)}{t_0 \wedge (n-t_0)}\Big)^{\frac{\alpha-1}{\alpha}}\biggr),
\]
for all $t_0 \in [n-1]$ and some $C_4 >0$. 
\end{proof}

\begin{prop}\label{prop:lowerbound-lessthan2}
    For the testing problem 
    \begin{equation*}
\mathrm{H}_0: \theta \in \Theta_0(p,n) \quad \mathrm{vs.}~\quad  \mathrm{H}_1: \theta \in  \bigcup_{t_0 = 1}^{n-1}\Theta(p,n, \rho_{t_0}),
\end{equation*}
when $p = 1$, it holds that 
\[
\inf_{\phi \in \Phi} \mathcal{R}_{\mathcal{W}_{\alpha}^{\otimes}}(\{\rho_{t_0}\}_{t_0 \in[n-1]}, \phi)\geq 1/2,
\]
if  
\begin{equation}\label{eq:lowerbound-lessthan2-rate}
    \rho_{t_0} \leq \bigl(t_0 \wedge (n-t_0) \bigr)^{-(1-\frac{1}{\alpha})},
\end{equation}
for some $t_0 \in [n-1].$
\end{prop}
\begin{proof}
    Consider two distributions $P_{+}$ and $P_{-}$ on $\mathbb{R}$ such that 
    \[
    P_{+}(\{0\}) = P_{-}(\{0\}) = 1-u, \quad P_{+}(\{c\}) = P_{-}(\{-c\}) = u,
    \]
    where $u \in [0,1]$ and $c>0$ are to be specified. Let $\mu_{+} = cu$ and $ \mu_{-} = -cu$ denote the mean of $P_+$ and  $P_{-}$, respectively. Note that we have $|\mu_{+}-\mu_{-}| = 2cu$. Also, we have the $\alpha$-th central moment of both distributions satisfies
    \[
    \mathbb{E}_{P_{+}}(|X - \mu_{+}|^{\alpha}) = \mathbb{E}_{P_{-}}(|X - \mu_{-}|^{\alpha}) = c^{\alpha}u(1-u)(u^{\alpha-1}+(1-u)^{\alpha-1}) \leq 2c^{\alpha}u,
    \]
 if $u \leq 1/2$.
 
   We first focus on the case where $t_0 \leq n/2$, i.e.\ when $\Big(\frac{1}{t_0}\Big)^{\frac{\alpha-1}{\alpha}}$ dominates.
   Consider the following two sequences of random variables
    \[
    X_1,\dotsc, X_{n} \overset{i.i.d}{\sim} P_{+}, \qquad Y_{1},\dotsc, Y_{t_0} \overset{i.i.d}{\sim} P_{-} , \;Y_{t_0+1},\dotsc, Y_{n} \overset{i.i.d}{\sim} P_{+} . 
    \]
    Let $P_{+}^{\otimes n}$ and $P_{-}^{\otimes n}$ denote the $n$-fold product distribution of $P_{+}$ and $P_{-} $, respectively. With the choice of $u = 1/(2t_0)$ and $c = (2u)^{-\frac{1}{\alpha}}$, we have $2c^{\alpha}u \leq 1$ and 
    \[
    1-\mathrm{TV}(P_{+}^{\otimes t_0}, P_{-}^{\otimes t_0}) \geq 1-t_0\mathrm{TV}(P_{+},P_{-}) = 1-t_0u \geq 1/2,
    \]
    with $|\mu_{+}-\mu_{-}| = t_0^{-\frac{\alpha-1}{\alpha}}$, which proves the claim. 
    
    In the case of $t_0 > n/2$, one simply chooses 
    \[
    Y_{1},\dotsc, Y_{t_0} \overset{i.i.d}{\sim} P_{+} , \;Y_{t_0+1},\dotsc, Y_{n} \overset{i.i.d}{\sim} P_{-}, 
    \]
    and the same arguments lead to the corresponding result.
\end{proof}

\subsection{Change away from boundary}\label{sec:app-away}
\subsubsection{Testing procedures}
\label{sec:test-awayfromboundary}

For $\mathcal{Q} = \mathcal{P}_{\alpha, K}^\otimes$ with $\alpha \geq 4$, we modify the median-of-means-type test proposed in Section~\ref{subSec:robust_sparse}. Recall that $\{Z_2, Z_4, \dotsc, Z_t\}$ is used for coordinate selection in \eqref{Eq:V_tga}. We now split this set into $G^{\mathrm{res}}$ groups of equal size, and use $\overline{Z}_{t, g, 2}$ to denote the sample mean of the $g$-th group. Our new test relies on a more robust coordinate selection step compared to \eqref{Eq:V_tga} by considering the following statistic:
\[
V^{\mathrm{res}}_{t,g,a^{\mathrm{res}}}(j) := \biggl( \overline{Z}^2_{t,g,1}(j) - \frac{2G_t}{t} \biggr) \mathbbm{1}_{ \Bigl\{ \sqrt{t/(2G^{\mathrm{res}})} \bigl| \mathrm{median}\bigl(\overline{Z}_{t, 1, 2}(j), \ldots, \overline{Z}_{t, G^{\mathrm{res}}, 2}(j)\bigr) \bigr| \geq a^{\mathrm{res}} \Bigr\} }, \quad j \in [p],
\]
where both the number of groups $G^{\mathrm{res}}$ and the threshold $a^{\mathrm{res}}$ are to be specified in \Cref{thm:awayfromboundaryP-app}. With $\mathcal{T}^{\mathrm{res}} := \mathcal{T} \cap \bigl\{\lceil (t^{\mathrm{res}}+1)/2 \rceil, \ldots, n+1-\lceil (t^{\mathrm{res}}+1)/2 \rceil  \bigr\}$, our test is 
\begin{equation}\label{eq:finitemoment_sparse_mom_restricted}
    \phi_{\mathcal{P}, \mathrm{sparse}}^{\mathrm{MoM+res}}:= \mathbbm{1}_{\{\max_{t \in \mathcal{T}^{\mathrm{res}}} A_{t,a^{\text{res}}}^{\mathrm{MoM}} / r_t > 1\}},
\end{equation}
where $A_{t,a^{\text{res}}}^{\mathrm{MoM}}$ is the same as \eqref{Eq:A_ta_MoM} but with each $V_{t,g,a}(j)$ replaced by $V^{\mathrm{res}}_{t,g,a^{\mathrm{res}}}(j)$, for $g \in [G_t]$.

For $\mathcal{Q} = \mathcal{G}_{\alpha, K}^\otimes$ with $0 < \alpha < 2$, we adopt a similar robust strategy for selecting signal coordinates in the sparse regime by replacing \eqref{eq:A_t,a} with the following statistic:
\[
A_{t,a}:=  \sum_{j=1}^p \bigl\{Y_{t,1}^2(j)-1\bigr\} \mathbbm{1}_{ \Bigl\{ \sqrt{t/(2G^{\mathrm{res}})} \bigl| \mathrm{median}\bigl(\overline{Z}_{t, 1, 2}(j), \ldots, \overline{Z}_{t, G^{\mathrm{res}}, 2}(j)\bigr) \bigr| \geq a^{\mathrm{res}} \Bigr\} },
\]
and the test takes the form of
\begin{equation}\label{eq:weibull_sparse_restricted}
    \phi_{\mathcal{G}, \mathrm{sparse}}^{\mathrm{res}}:= \mathbbm{1}_{\{\max_{t \in \mathcal{T}^{\mathrm{res}}} A_{t,a} > r\}}.
\end{equation}

\subsubsection{Proof of Theorem~\ref{thm:awayfromboundaryP}}

We prove the following theorem on the theoretical guarantee of the test $\phi_{\mathcal{P}, \mathrm{sparse}}^{\mathrm{MoM+res}}$, constructed above in \Cref{sec:test-awayfromboundary}. Theorem~\ref{thm:awayfromboundaryP} then follows as an immediate consequence.

\begin{thm} \label{thm:awayfromboundaryP-app}
    Let $t^{\mathrm{res}} = 32\bigl\{\log(e^2p/s) + s^{-1} \log \log (8n)\bigr\}$ and $\mathcal{Q} = \mathcal{P}_{\alpha, K}^\otimes$ with $\alpha \geq 4$. For any $\varepsilon \in (0, 1)$, there exist $C_1, C_2, C_3 > 0$ depending only on $\alpha$, $K$ and $\varepsilon$, such that the test $\phi_{\mathcal{P}, \mathrm{sparse}}^{\mathrm{MoM+res}}$ defined in \eqref{eq:finitemoment_sparse_mom_restricted} with
    \begin{equation*} 
    a^{\mathrm{res}} = C_1, \, G^{\mathrm{res}} = 2^ {\lfloor \log_2(t^{\mathrm{res}}/2) \rfloor}, \, r_t =  C_2 \sqrt{s} G_t, \,   
    G_t = (t \wedge \Delta)/2 \mbox{ and } \Delta = 2^{4 + \lceil\log_2 \log \log(8n)\rceil},
    \end{equation*}
satisfies that
\[
\mathcal{R}_{\mathcal{P}}(\rho,  \phi_{\mathcal{P}, \mathrm{sparse}}^{\mathrm{MoM+res}}) \leq \varepsilon,
\]
as long as $\rho^2 \geq C_3 v_{\mathcal{P}, \mathrm{sparse}}^{\mathrm{MoM+res}}$, where 
\[
v_{\mathcal{P}, \mathrm{sparse}}^{\mathrm{MoM+res}} := s\bigl\{\log(ep/s) + \log\log(8n)\bigr\}.
\]
\end{thm}

\begin{proof}
The proof is analogous to that of Proposition~\ref{thm:finitemoment_upperbound_sparse}; we thus omit many details and highlight those places where the arguments differ. \\
\noindent \textbf{Null term.} For any $\theta \in \Theta_0(p,n)$, we have by a union bound that
\begin{align} \label{Eq:sparse_MoMnull_master_restricted}
    \mathbb{E}_\theta \phi_{\mathcal{P},\mathrm{sparse}}^{\mathrm{MoM+res}} \leq \sum_{t\in \mathcal{T}^{\mathrm{res}}} \mathbb{P}_\theta\bigl(A_{t,a}^{\mathrm{MoM}} > r_t\bigr) \leq \sum_{t\in \mathcal{T}^{\mathrm{res}}} \mathbb{P}_\theta (|\mathcal{J}_{t, a^{\mathrm{res}}}| > s) + \sum_{t\in \mathcal{T}^{\mathrm{res}}} \sup_{J \subseteq [p]: |J| \leq s} \mathbb{P}_\theta(A_{t,*,J}^{\mathrm{MoM}} > r_t),
\end{align}
with $A_{t,*,J}^{\mathrm{MoM}}$ unchanged from the proof of Proposition~\ref{thm:finitemoment_upperbound_sparse} and $\mathcal{J}_{t,a^{\mathrm{res}}}$ modified to be
\[
\mathcal{J}_{t,a^{\mathrm{res}}} := \Bigl\{j \in [p]: \sqrt{t/2G^{\mathrm{res}}} \bigl| \mathrm{median}\bigl(\overline{Z}_{t, 1, 2}(j), \ldots, \overline{Z}_{t, G^{\mathrm{res}}, 2}(j)\bigr) \bigr| \geq a^{\mathrm{res}} \Bigr\},
\]
for $t \in \mathcal{T}^{\mathrm{res}}$. The second term in~\eqref{Eq:sparse_MoMnull_master_restricted} can still be bounded by $\varepsilon/8$ with $r_t, G_t$ and $\Delta$ unchanged from Proposition~\ref{thm:finitemoment_upperbound_sparse}. We now bound the first term. By Fuk--Nagaev inequality (\Cref{prop:fuknagaev}), we have for any $j \in [p]$, $g \in [G^{\mathrm{res}}]$ and $t \in \mathcal{T}^{\mathrm{res}}$
\begin{align*}
    \mathbb{P}_\theta \biggl( \sqrt{\frac{t}{2G^{\mathrm{res}}}} \bigl|\overline{Z}_{t,g,2}(j) \bigr| \geq a^{\mathrm{res}} \biggr) &\leq 2\Biggl( \frac{(\alpha+2)(K^\alpha 2^{\alpha/2}t/ (2G^{\mathrm{res}}))^{1/\alpha}}{\alpha a^{\mathrm{res}} \sqrt{t/(2G^{\mathrm{res}})}} \Biggr)^\alpha + 2\exp \biggl\{ -\frac{2(a^{\mathrm{res}} )^2}{(\alpha+2)^2e^\alpha} \biggr\}. \\
&\leq \frac{K^\alpha}{(a^{\mathrm{res}} /3)^\alpha (t/G^{\mathrm{res}})^{\alpha/2-1}} + \exp \biggl\{1 -\frac{(a^{\mathrm{res}})^2}{2\alpha^2e^\alpha} \biggr\}\\
&\leq (3K/a^{\mathrm{res}} )^\alpha + \exp \biggl\{1 -\frac{(a^{\mathrm{res}} )^2}{2\alpha^2e^\alpha} \biggr\} \leq \varepsilon/36,
\end{align*}
when $a^{\mathrm{res}} = C_1$ is sufficiently large, depending on $K, \alpha$ and $\varepsilon$. Consequently, by the multiplicative Chernoff bound, we have for $j \in [p]$
\begin{align}
\mathbb{P}_\theta (j \in \mathcal{J}_{t,a^{\mathrm{res}}}) &\leq \mathbb{P}_\theta \Biggl( \biggl| \biggl\{g \in [G^{\mathrm{res}}]: \sqrt{\frac{t}{2G^{\mathrm{res}}}} \bigl|\overline{Z}_{t,g,2}(j) \bigr| \geq a^{\mathrm{res}} \biggr\}\biggr| \geq G^{\mathrm{res}}/2\Biggr) \nonumber \\
&\leq \mathbb{P}_\theta \Biggl( \biggl| \biggl\{g \in [G^{\mathrm{res}}]: \sqrt{\frac{t}{2G^{\mathrm{res}}}} \bigl|\overline{Z}_{t,g,2}(j) \bigr| \geq a^{\mathrm{res}} \biggr\}\biggr| \geq \frac{\varepsilon G^{\mathrm{res}}}{36} \biggl(1 + \Bigl( \frac{18}{\varepsilon}-1 \Bigr) \biggr)\Biggr) \nonumber \\
&\leq \exp\biggl\{-\frac{\varepsilon G^{\mathrm{res}}}{36} \biggl( \frac{18}{\varepsilon }\log\Bigl(\frac{18}{\varepsilon}\Bigr) - \frac{18}{\varepsilon} + 1\biggr)  \biggr\} \leq \exp\biggl\{ -\frac{G^{\mathrm{res}}}{2} \log\bigl(6/\varepsilon\bigr)\biggr\} \leq (\varepsilon/6)^{G^{\mathrm{res}}/2}. \label{Eq:selection_prob_restricted}
\end{align}
Then, again by a standard binomial tail bound and a union bound
\begin{align*}
\sum_{t\in \mathcal{T}^{\mathrm{res}}} \mathbb{P}_\theta (|\mathcal{J}_{t, a^{\mathrm{res}}}| > s) \leq \log_2(n) \biggl( \frac{ep(\varepsilon/6)^{G^{\mathrm{res}}/2}}{s} \biggr)^s \leq \varepsilon/4,
\end{align*}
when $G^{\mathrm{res}} \geq 2\bigl(\log(e^2p/s) + s^{-1}\log\log(8n)\bigr)$. Putting things together, we reach $\mathbb{E}_\theta \phi_{\mathcal{P},\mathrm{sparse}}^{\mathrm{MoM+res}} \leq \varepsilon/2$.

\medskip
\noindent \textbf{Alternative term.} First, according to our definition of $\mathcal{T}^{\mathrm{res}}$, there still exists a unique $\tilde{t} \in \mathcal{T}^{\mathrm{res}}$ such that $t_0/2 < \tilde{t} \leq t_0$, where $t_0 \geq t^{\mathrm{res}}+1$ (and without loss of generality $t_0 \leq n/2$) is the true change point location. We retain most notation used in the proof of Proposition~\ref{thm:finitemoment_upperbound_sparse}, with one modification $\mathcal{H}_{\delta, a^{\mathrm{res}}} := \bigl\{j \in [p]: |\delta(j)| \geq 2a^{\mathrm{res}}\sqrt{G^{\mathrm{res}}}\bigr\}$. The reasoning behind this change is
\begin{align} \label{Eq:median_diff_restricted}
    \sqrt{ \frac{\tilde{t}}{2G^{\mathrm{res}}}} \mathrm{median}\bigl(\overline{Z}_{\tilde{t}, 1, 2}(j), \ldots, \overline{Z}_{\tilde{t}, G^{\mathrm{res}}, 2}(j)\bigr) = \sqrt{ \frac{\tilde{t}}{2G^{\mathrm{res}}}} \mathrm{median}\bigl({\overline{Z}}'_{\tilde{t}, 1, 2}(j), \ldots, {\overline{Z}}'_{\tilde{t}, G^{\mathrm{res}}, 2}(j)\bigr) + \frac{\delta(j)}{\sqrt{G^{\mathrm{res}}}}.
\end{align}
The only major difference in the proof lies in the argument for establishing
\begin{equation} \label{Eq:sum_delta2_again}
    \mathbb{P}_\theta \biggl( \sum_{j\in\mathcal{J}_{\tilde{t}, a^{\mathrm{res}}} \cap \mathcal{H}_{\delta, a^{\mathrm{res}}}} \delta(j)^2 <\frac{\|\delta\|_2^2}{12\log(8/\varepsilon)} \biggr) \leq \varepsilon/8.
\end{equation}
For $j \in \mathcal{H}_{\delta, a^{\mathrm{res}}}$, we have by~\eqref{Eq:median_diff_restricted} and~\eqref{Eq:selection_prob_restricted}
\begin{align*}
    \mathbb{P}_\theta(j \notin \mathcal{J}_{\tilde{t}, a^{\mathrm{res}}})  &= \mathbb{P}_\theta \Biggl( \sqrt{\tilde{t}/2G^{\mathrm{res}}} \bigl| \mathrm{median}\bigl(\overline{Z}_{\tilde{t}, 1, 2}(j), \ldots, \overline{Z}_{\tilde{t}, G^{\mathrm{res}}, 2}(j)\bigr) \bigr|  < a^{\mathrm{res}} \Biggr) \\
    &\leq \mathbb{P}_\theta \Biggl( \sqrt{\tilde{t}/2G^{\mathrm{res}}} \bigl( \mathrm{median}\bigl(\overline{Z}'_{\tilde{t}, 1, 2}(j), \ldots, \overline{Z}'_{\tilde{t}, G^{\mathrm{res}}, 2}(j)\bigr) \bigr)  < -a^{\mathrm{res}} \Biggr) \\
    &\leq (\varepsilon/6)^{G^{\mathrm{res}}/2} \leq \frac{\varepsilon}{2048} \leq \frac{1}{256\log(8/\varepsilon)},
\end{align*}
whenever $G^{\mathrm{res}} \geq 10$. Thus, we still have
\begin{align*}
\sum_{j \in \mathcal{H}_{\delta, a^{\mathrm{res}}}} \mathrm{Var}_\theta\bigl( \delta(j)^2\mathbbm{1}_{ \{ j \in \mathcal{J}_{\tilde{t}, a^{\mathrm{res}}} \} }  \bigr) \leq \frac{\|\delta\|_2^4}{256\log(8/\varepsilon)}.
\end{align*}
The condition $\rho^2 \geq 64a^2 s$ in the previous proof is replaced by $\rho^2 \geq 64(a^{\mathrm{res}})^2 sG^{\mathrm{res}}$ in order to obtain
\[
\sum_{j\in\mathcal{H}_{\delta, a^{\mathrm{res}}}} \delta(j)^2 \geq \|\delta\|_2^2 - s\bigl(2a^{\mathrm{res}}\sqrt{G^{\mathrm{res}}}\bigr)^2 \geq \|\delta\|_2^2/2.
\]
With these in place, \eqref{Eq:sum_delta2_again} can be established and we can deduce $\mathbb{E}_\theta (1 - \phi_{\mathcal{P},\mathrm{sparse}}^{\mathrm{MoM+res}}) \leq \varepsilon/2$ whenever
\[
\rho^2 \geq C\bigl(s\Delta + (a^{\mathrm{res}})^2 s G^{\mathrm{res}} \bigr)
\]
with a sufficiently large $C$.
\end{proof}

Theorem~\ref{thm:awayfromboundaryG} is again a direct consequence of the following result, which provides a theoretical guarantee for the test $\phi_{\mathcal{G}, \mathrm{sparse}}^{\mathrm{res}}$ constructed in \Cref{sec:test-awayfromboundary}.

\begin{thm} \label{thm:awayfromboundaryG-app}
    Let $t^{\mathrm{res}} = 32\bigl\{\log(e^2p/s) + s^{-1} \log \log (8n)\bigr\}$ and $\mathcal{Q} = \mathcal{G}_{\alpha, K}^\otimes$ with $0 < \alpha < 2$.  For any $\varepsilon \in (0, 1)$, there exist $C_1, C_2, C_3 > 0$ depending only on $\alpha$, $K$ and $\varepsilon$, such that the test $\phi_{\mathcal{G}, \mathrm{sparse}}^{\mathrm{MoM+res}}$ defined in \eqref{eq:weibull_sparse_restricted} with
    \begin{equation*}
    a^{\mathrm{res}} = C_1, \quad G^{\mathrm{res}} = 2^ {\lfloor \log_2(t^{\mathrm{res}}/2) \rfloor} \quad \mbox{and} \quad r = C_2\bigl\{\sqrt{s\log\log(8n)} + \log\log(8n)  \bigr\},
\end{equation*}
satisfies that
\[
\mathcal{R}_{\mathcal{G}}(\rho,  \phi_{\mathcal{G}, \mathrm{sparse}}^{\mathrm{res}}) \leq \varepsilon,
\]
as long as $\rho^2 \geq C_3 v_{\mathcal{G}, \mathrm{sparse}}^{\mathrm{res}}$, where 
\[
v_{\mathcal{G}, \mathrm{sparse}}^{\mathrm{res}} := s\log(ep/s) + \log\log(8n).
\]
\end{thm}
The proof of Theorem~\ref{thm:awayfromboundaryG-app} is omitted, as it is very similar to the last proof.

\paragraph*{Comparison with \texorpdfstring{\cite{yu2022robust}}{Yu and Chen (2022)}:}

We compare Theorem \ref{thm:awayfromboundaryG} with the minimum signal strength requirement for a bootstrapped U-statistics-based test \citep[][Theorem 3.3]{yu2022robust}. Their result indicates that, for the sub-exponential noise class $\mathcal{G}_{1,K}^\otimes$, if $\log(p) = o(n^{1/7})$ (a mild dimension condition required for size control in Theorem 3.1 therein) and the change is away from the boundary by at least $O(\log^{5/2}(np))$, then the sum of Type I error and Type II error probabilities can be controlled provided that
\begin{equation} \label{Eq:YuChen_power}
    t_0(n-t_0)\|\mu_1 - \mu_2\|_\infty \gtrsim n^{3/2}\log^{1/2}(np).
\end{equation}
One immediate observation is that our requirement on the boundary removal, $O\bigl(\log(ep/s) + s^{-1}\log\log(8n)\bigr)$, is much smaller than their required $O(\log^{5/2}(np))$. 

Now, suppose the $s$-sparse mean shift $\mu_1 - \mu_2$ takes the form
\[
\mu_1 - \mu_2 = (\underbrace{a, \ldots, a}_{s}, \underbrace{0, \ldots, 0}_{p-s})^\top.
\]
Without loss of generality, assume $t_0 \leq n/2$. Under our framework, the signal strength condition $\rho^2 \geq C_3 v_{\mathcal{G}, \mathrm{sparse}}^{\mathrm{res}}$ implies 
\[
a \gtrsim t_0^{-1/2}\bigl(\log^{1/2}(ep/s) + s^{-1/2}\log^{1/2}\log(8n)\bigr),
\]
while \eqref{Eq:YuChen_power} requires
\[
a \gtrsim (n/t_0)^{1/2}t_0^{-1/2}\bigl(\log^{1/2}(p) + \log^{1/2}(n)\bigr).
\]
Our rate is clearly smaller across all parameter settings, with the advantage becoming particularly pronounced when the change location is not near the middle of the sequence, i.e.\ when $t_0$ is not of the same order as $n$.

\section{Auxiliary results} \label{Sec:aux}
We first present the definition and some basic properties of sub-Weibull random variables. For a more in-depth introduction and discussion, we refer to \citet{vladimirova2020sub} and \citet[][Section 2]{Kuchibhotla2022beyondsubG}.
\begin{defn}[Orlicz norms] \label{def:Orlicz}
Let $f: [0, \infty) \rightarrow [0, \infty)$ be a non-decreasing function with $f(0) = 0$. The $f$-Orlicz norm of a real-valued random variable $X$ is
\[
\|X\|_f := \inf\{t > 0: \mathbb{E}f(|X|/t) \leq 1 \}.
\]
\end{defn}
\begin{defn}[sub-Weibull random variables] \label{def:subWeibull}
A random variable $X$ is sub-Weibull of order $\alpha > 0$, denoted sub-Weibull($\alpha$), if it has mean zero and
\[
\|X\|_{\psi_\alpha} < \infty,
\]
with the function $\psi_\alpha$ defined by $\psi_\alpha(x):=\exp(x^\alpha) - 1$ for $x \geq 0$.
\end{defn}

\begin{prop}[\citealp{vladimirova2020sub}, Theorem 2.1] \label{prop:subweibull-basic-prop}
Let $X$ be a sub-Weibull($\alpha$) random variable with $0 < \alpha \leq 2$ and $\|X\|_{\psi_\alpha} = K < \infty$. Then, we have the following properties: \\
(a) the tails of $X$ satisfy
\[ 
\mathbb{P}(|X| \geq x) \leq 2\exp\{-(x/K)^\alpha\} \qquad \text{for all } x \geq 0;
\]
(b) Let $\|X\|_k := \mathbb{E}(|X|^k)^{1/k}, k \geq 1$, then
\[
\|X\|_k \leq K' k^{1/\alpha}
\]
for some absolute constant $K' >0$. \\
(c)
Conversely, if a random variable $X$ has mean zero and satisfies $\mathbb{P}(|X| \geq x) \leq 2\exp\{-(x/K)^\alpha\}$ for all $x \geq 0$, then there exists $K'' > 0$, depending only on $\alpha$ and $K$, such that
\[
\mathbb{E}\exp\bigl\{(|X|/K'')^\alpha\bigr\} \leq 2.
\]
In other words, $X$ is a sub-Weibull($\alpha$) random variable with $\|X\|_{\psi_\alpha} \leq K'' < \infty$.
\end{prop}

\begin{prop}[\citealp{vladimirova2020sub}, Proposition 2.1] \label{lem:subweibull_ordering}
Let $\alpha > \alpha' > 0$ and $X$ be a sub-Weibull($\alpha$) random variable with $\|X\|_{\psi_\alpha} = K < \infty$. Then there exists $K'>0$, depending only on $\alpha'$ and $K$, such that $X$ is a sub-Weibull($\alpha'$) random variable with $\|X\|_{\psi_{\alpha'}} \leq K' < \infty$.
\end{prop}

We now provide two tail bound results from literature for sums and quadratic forms of independent sub-Weibull random variables respectively. \Cref{prop:quadraticweibulltail} below can be viewed as an extension of the Hanson--Wright inequality \citep{Hanson1971bound}.
\begin{prop}[\citealp{Kuchibhotla2022beyondsubG}, Theorem 3.1] \label{prop:sumofweibulltail}
Let $\alpha > 0$ and $n \in \mathbb{Z}^+$. Let $X_1, \dotsc, X_n$ be independent mean zero sub-Weibull random variables of order $\alpha$, with $\|X_i\|_{\psi_\alpha} \leq K$ for all $i \in \mathbb{Z}^+$ and for some $K > 0$. Then, there exists a constant $C > 0$, depending only on $\alpha$ and $K$, such that for any vector $u = (u_1, \dotsc, u_n)^\top \in \mathbb{R}^n$ and $x \geq 0$, we have
\begin{equation*}
\mathbb{P}\Bigl(\Bigl|\sum_{i=1}^n u_iX_i\Bigr| \geq x \Bigr) \leq \exp\biggl\{ 1 - \min\biggl\{ \biggl(\frac{x}{C\|u\|_2}\biggr)^2, \biggl(\frac{x}{C\|u\|_{\beta(\alpha)}}\biggr)^\alpha   \biggr\}   \biggr\},
\end{equation*}
where $\beta(\alpha) = \infty$ when $\alpha \leq 1$ and $\beta(\alpha) = \alpha/(\alpha-1)$ when $\alpha > 1$.
\end{prop}
\begin{prop}[\citealp{Gotze2021concentration}, Proposition 1.5] \label{prop:quadraticweibulltail}
Let $\alpha \in (0,1] \cup \{2\}$, $A = (a_{ij}) \in \mathbb{R}^{n \times n}$ be a symmetric matrix and $X_1, \dotsc, X_n$ be independent mean zero sub-Weibull random variables of order $\alpha$, with $\mathbb{E}X_i^2 = \sigma_i^2$ and $\|X_i\|_{\psi_\alpha} \leq K$ for all $i \in \mathbb{Z}^+$ and for some $K > 0$. Then, there exists a constant $C > 0$, depending only on $\alpha$ and $K$, such that for any $x \geq 0$, we have
\begin{equation*}
    \mathbb{P}\biggl(\Bigl|\sum_{1 \leq i, j \leq n} a_{ij}X_i X_j - \sum_{i=1}^n a_{ii} \sigma_i^2\Bigr| \geq x\biggr) \leq \exp\bigl( 1 - \eta_\alpha(x/C; A) \bigr),
\end{equation*}
where
\[
\eta_\alpha(x; A) := \min\Biggl\{ \biggl(\frac{x}{\|A\|_{\mathrm{F}}}\biggr)^2, \frac{x}{\|A\|_2}, \biggl(\frac{x}{\|A\|_{2 \rightarrow \infty}}\biggr)^{\frac{2\alpha}{2+\alpha}}, \biggl(\frac{x}{\|A\|_{\max}}\biggr)^{\frac{\alpha}{2}}  \Biggr\}.
\]
\end{prop}

The following proposition presents a concentration inequality for sums of independent random variables with only finite certain number of moments. We use the form of the Fuk--Nagaev type inequalities appeared in \citet{rio2017FukNagaev}.
\begin{prop}[\citealp{Fuk1973inequalities, nagaev1979deviations}] \label{prop:fuknagaev}
Let $X_1, \dotsc, X_n$ be independent random  variables, each having mean $0$ and variance $\sigma^2$. Assume further that for some $q \geq 2$ and $C_q > 0$, we have for all $i \in [n]$
\[
\mathbb{E}[\{\max(X_i, 0)\}^q] \leq C_q.
\]
Then for any $x > 0$, we have
\[
\mathbb{P}\biggl(\sum_{i=1}^n X_i \geq x\biggr) \leq \Biggl( \frac{(q+2)(nC_q)^{1/q}}{qx} \Biggr)^q + \exp \Biggl\{ -\frac{2x^2}{n(q+2)^2e^q\sigma^2} \Biggr\}.
\]
\end{prop}

\begin{prop} \label{prop:2sample}
    Let $X_1,\dotsc, X_p$ be independent random variables, each with mean zero and unit variance. Let $a \geq 0$ and $Z := \sum_{i=1}^p (X_i^2 - 1) \mathbbm{1}_{\{|X_i| \geq a\}}$.\\
    (a) Let $\alpha > 0$, $K > 0$, $0 < \varepsilon \leq 1$ and $1 \leq s \leq \sqrt{p}$. Assume that $X_1,\dotsc, X_p$ are independent sub-Weibull random variables of order $\alpha$, with $\|X_i\|_{\psi_\alpha} \leq K$ for all $i \in [p]$. By setting
    \[
    a \geq K\log^{1/\alpha}\Bigl(\frac{4ep}{\varepsilon s}\Bigr) \qquad \text{and} \qquad r = 2^{2/\alpha}K^2s\log^{2/\alpha}\Bigl(\frac{4p}{\sqrt{\varepsilon} s}\Bigr),
    \]
    we have $\mathbb{P}(Z > r) \leq \varepsilon$.\\
    (b) Let $\alpha \geq 2$, $K > 0$, $0 < \varepsilon \leq 1$ and $1 \leq s \leq p$. Assume that $\mathbb{E}|X_i/K|^\alpha \leq 1$ for all $i \in [p]$. By setting
    \[
    a \geq K \Bigl(\frac{2ep}{\varepsilon s}\Bigr)^{1/\alpha} \qquad \text{and} \qquad r = \frac{\alpha-2}{\alpha}K^2s\Bigl(\frac{2ep}{\varepsilon s}\Bigr)^{2/\alpha},
    \]
    we have $\mathbb{P}(Z > r) \leq \varepsilon$.\\
    (c) Assume the same conditions as in (b). Write $Z_s := \sum_{i=1}^p (X_i^2 - 1) \mathbbm{1}_{\{|X_i| \geq a_s\}}$ to make the dependence on $s$ explicit. By choosing the same $a_s$ and $r_s$ as in (b), we have 
    \[
    \mathbb{P}\bigl(\max_{s \in [p]} Z_s/r_s > 1\bigr) \leq 2\varepsilon.
    \]
\end{prop}
\begin{proof}
    We denote the order statistics of $|X_1|,\dotsc, |X_p|$ as $|X|_{(1)} \leq \dotsc \leq |X|_{(p)}$. For $x \geq 0$, we write $q_x := \min_{i \in [p]} \mathbb{P}(|X_i| \geq x)$ and $\mathcal{J}_x := \{i \in [p]: |X_i| \geq x\}$.\\
    (a)  Note that $q_x \leq 2\exp\{-(x/K)^\alpha\}$ by Proposition~\ref{prop:subweibull-basic-prop}(a). Since $s \leq \sqrt{p}$, we observe that
    \[
    \sum_{j=1}^s K^2\log^{2/\alpha} \Bigl(\frac{4ep}{\varepsilon j}\Bigr) \leq \sum_{j=1}^s K^2\log^{2/\alpha} \Bigl(\frac{4ep^2}{\varepsilon s^2}\Bigr) \leq 2^{2/\alpha}K^2s\log^{2/\alpha} \Bigl(\frac{4p}{\sqrt{\varepsilon}s}\Bigr) = r.
    \]
    Then, by a union bound and a binomial tail bound, we have
    \begin{align*}
    \mathbb{P}(Z > r) &\leq \mathbb{P}(|\mathcal{J}_a| > s) +  \mathbb{P}\Biggl(\sum_{j=1}^s \bigl(|X|_{(p-j+1)}^2-1\bigr) > r\Biggr) \\
    &\leq \mathbb{P}(|\mathcal{J}_a| > s) + \sum_{j=1}^s\mathbb{P}\biggl(|X|_{(p-j+1)}> K\log^{1/\alpha} \Bigl(\frac{4ep}{\varepsilon j}\Bigr)\biggr)\\
    &\leq \Bigl(\frac{ep}{s}\Bigr)^s \Biggl\{2\exp\biggl\{- \Bigl( \frac{a}{K} \Bigr) ^\alpha\biggr\}\Biggr\}^s + \sum_{j=1}^s\Bigl(\frac{ep}{j}\Bigr)^j\Biggl\{2\exp\biggl\{- \biggl( \frac{K\log^{1/\alpha} \bigl(\frac{4ep}{\varepsilon j}\bigr)}{K} \biggr) ^\alpha\biggr\}\Biggr\}^j \\
    &\leq (\varepsilon/2)^s + \sum_{j=1}^s (\varepsilon/2)^j \leq \varepsilon. 
    \end{align*}
    (b) Note that $q_x \leq (K/a)^\alpha$ by Chebyshev's inequality. We now observe that
    \[
    \sum_{j=1}^s K^2 \Bigl(\frac{2ep}{\varepsilon j}\Bigr)^{2/\alpha} \leq K^2 \Bigl(\frac{2ep}{\varepsilon}\Bigr)^{2/\alpha} \biggl\{1+\int_{1}^s x^{-2/\alpha} \, dx\biggr\} \leq \frac{\alpha-2}{\alpha}K^2s\Bigl(\frac{2ep}{\varepsilon s}\Bigr)^{2/\alpha} = r
    \]
    The rest then follows from the proof for part (a).\\
    (c) By a union bound and the proof for the previous parts, we have
    \begin{align*}
        \mathbb{P}\bigl(\max_{s \in [p]} Z_s/r_s > 1\bigr) &\leq \Biggl(\sum_{s=1}^p \mathbb{P}(|\mathcal{J}_{a(s)}| > s)\Biggr) + \mathbb{P}\Biggl( \max_{s\in [p]} \frac{\sum_{j=1}^s \bigl(|X|_{(p-j+1)}^2-1\bigr)}{r_s} > 1\Biggr) \\
        &\leq \sum_{s=1}^p (\varepsilon/2)^s + \sum_{j=1}^p\mathbb{P}\biggl(|X|_{(p-j+1)}> K \Bigl(\frac{2ep}{\varepsilon j}\Bigr)^{2/\alpha} \biggr) \\
        &\leq \sum_{s=1}^p (\varepsilon/2)^s + \sum_{j=1}^p (\varepsilon/2)^j \leq 2\varepsilon.
    \end{align*}
\end{proof}

\begin{lemma} \label{lemma:exp_decay_sum}
Let $\gamma > 0$. Then, for all $x \geq (2^\gamma - 1)^{-1/\gamma}$ we have
\[
\sum_{i = 0}^\infty  \exp\bigl\{-(x2^i)^\gamma\bigr\} \leq 2\exp(-x^\gamma).
\]
\end{lemma}
\begin{proof}
By the convexity of $y \mapsto 2^{\gamma y}$, we have that $2^{(i+1)\gamma} - 2^{i\gamma} \geq 2^{i\gamma} -2^{(i-1)\gamma}$
and thus
\[
2^{i\gamma} = 1 + \sum_{j=1}^i (2^{j\gamma} - 2^{(j-1)\gamma}) \geq 1 + i(2^\gamma - 1).
\]
for all $i \in \mathbb{Z}^+$. Denote $\tilde{x} := \exp(x^\gamma)$. We hence deduce that when $\tilde{x} > 2^{\frac{1}{2^\gamma -1}}$,
\[
\sum_{i = 0}^\infty  \exp\bigl\{-(x2^i)^\gamma\bigr\} = \sum_{i = 0}^\infty \tilde{x}^{-2^{i\gamma}} \leq \sum_{i = 0}^\infty \tilde{x}^{-1 - i(2^\gamma - 1)}  = \frac{1}{\tilde{x} \bigl(1 - \tilde{x}^{-(2^\gamma - 1)}\bigr)} \leq 2\tilde{x}^{-1}.
\]
\end{proof}

\begin{lemma} \label{Lemma:bounded2kmoments}
    Let $Z_1,\dotsc, Z_n$ be independent mean zero random variables. \\
    (a) Assume that there exists $C > 0$ such that $\mathbb{E}Z_i^{4} \leq C$ for all $i \in [n]$. Then for any $v=(v_1, \dotsc, v_n)^\top \in \mathbb{R}^n$, we have
    \[
    \mathbb{E}\biggl[\Bigl(\sum_{i=1}^n v_iZ_i\Bigr)^4\biggr] \leq 3C\|v\|_2^4.
    \]
    \noindent (b) Assume that there exists $C > 0$ such that $\mathbb{E}\bigl(|Z_i|^{2k}\bigr) \leq C$ for some $k \geq 1$. Then, there exists a constant $C_k > 0$, depending only on $k$ and $C$ such that 
\[
\mathbb{E}\biggl[\Big|\sum_{i=1}^nZ_i/n \Big|^{2k} \biggr] \leq C_k  n^{-k}. 
\]
   
\end{lemma}
\begin{proof}
(a) Since $\mathbb{E}Z_i^0 = 1$ and $\mathbb{E}Z_i^1 = 0$ for all $i\in [n]$, we have
\begin{align*}
    \mathbb{E}\biggl[\Bigl(\sum_{i=1}^n v_iZ_i\Bigr)^4\biggr] &= \sum_{i=1}^n v_i^4 \mathbb{E}Z_i^4 + \sum_{1\leq i < j \leq n} 6v_i^2v_j^2 \mathbb{E}(Z_i^2Z_j^2) \leq C\Bigl(\sum_{i=1}^n v_i^4 + \sum_{1\leq i < j \leq n} 6v_i^2v_j^2\Bigr) \\
    &\leq 3C\Bigl(\sum_{i=1}^n v_i^2\Bigr)^2,
\end{align*} 
where the first inequality follows from Jensen's inequality. \\
\noindent (b) Note that $S_j := \sum_{i=1}^j Z_i$ is a martingale (adapted to the natural filtration) and $[S]_j := \sum_{i=1}^j Z_i^2$ can be viewed as the quadratic variation of this martingale. By Burkholder--Davis--Gundy inequality \cite[e.g.][Theorem 1.1]{beiglbock2015pathwise}, we have for any $k \geq 1$, 
    \[
    \mathbb{E}\bigl[|S_n|^{2k}\bigr] \leq \mathbb{E}\biggl[\Bigl(\max_{j \leq n}|S_j|\Bigr)^{2k}\biggr] \leq C_{k,1}\mathbb{E}\Bigl[\bigl([S]_n\bigr)^k\Bigr],
    \]
    for some constant $C_{k,1} >0$, depending only on $k$. Thus, we have
    \[
    \mathbb{E}\biggl[\Big|\sum_{i=1}^nZ_i/n \Big|^{2k} \biggr] \leq \frac{C_{k,1}}{n^{2k}}\mathbb{E}\biggl[\Big(\sum_{i=1}^n Z_i^2\Big)^{k}\biggr] \leq \frac{C_{k,1}}{n^{k}}\mathbb{E} \bigg[\frac{\sum_{i=1}^n|Z_i|^{2k}}{n}\bigg] \leq \frac{C C_{k,1}}{n^{k}},
    \]
    where we have used Jensen's inequality in the second inequality. 
\end{proof}

    \begin{lemma}\label{lemma:lowmoment-2}
        Let $k \geq 1$ and $V_1,\dotsc,V_L$ be independent random vectors in $\mathbb{R}^p$, each having zero mean and independent coordinates. Assume that there exists $C > 0$ such that $\mathbb{E}\bigl[|V_i(j)|^{2k}\bigr] \leq C$ for all $i \in [L]$ and $j \in [p]$. Denote $\overline{V} := \sum_{i=1}^L V_i/L$. Then for any $ \delta \in (0,1)$, we have 
        \[
        \mathbb{P}\Bigg(\bigg|\sum_{j=1}^p L\bigl(\overline{V}^2(j) - 1/L\bigr)\bigg| > C_k \frac{p^{\frac{1}{2} \vee \frac{1}{k}}}{\delta^{1/k}}\Bigg) \leq \delta
        \]
        for some constant $C_k > 0$, depending only on $C$ and $k$.
    \end{lemma}
    
    \begin{proof}
    We first prove the result for $1 \leq k \leq 2$. Note that for any $\eta > 0$
    \begin{multline}\label{eq:mom-bound-1+epsilon}
        \mathbb{P}\bigg(\bigg|\sum_{j=1}^p L\bigl(\overline{V}^2(j) - 1/L\bigr)\bigg| >\eta \bigg) \leq p \mathbb{P}\Bigl( \bigl|L\bigl(\overline{V}^2(1) - 1/L\bigr)\bigr| > \eta\Bigr) \\ + \mathbb{P}\bigg(\bigg|\sum_{j=1}^p L\bigl(\overline{V}^2(j) - 1/L\bigr)\mathds{1}\Bigl\{\bigl|L\bigl(\overline{V}^2(j) - 1/L\bigr)\bigr| \leq \eta\Bigr\} \bigg|>\eta \bigg).
    \end{multline}
    We control the two terms separately. For the first term, we have
    \begin{equation}\label{eq:bound1}
    \mathbb{P}\Bigl(\bigl|L\bigl(\overline{V}^2(1) - 1/L\bigr)\bigr| > \eta\Bigr) \leq \frac{\mathbb{E}\Bigl[\bigl|L\overline{V}^2(1) - 1\bigr|^{k}\Bigr]}{\eta^{k}} \leq \frac{2^{k-1}\Bigl(\mathbb{E} \bigl|L\overline{V}^2(1)\bigr|^{k} +1\Bigr)}{\eta^{k}} \leq \frac{2^{k-1}(C_{0,k}+1)}{\eta^k},
    \end{equation}
    where the three inequalities follow, respectively, from Markov's inequality, Jensen's inequality and \Cref{Lemma:bounded2kmoments}(b), with $C_{0,k}$ being the constant that depends only on $C$ and $k$ in that lemma. For convenience, we denote $C_{1,k} := 2^{k-1}(C_{0,k}+1)$ hereafter. For the second term in \eqref{eq:mom-bound-1+epsilon}, we have 
    \begin{align}\label{eq:bound2}
        &\mathbb{P}\bigg(\bigg|\sum_{j=1}^p L\bigl(\overline{V}^2(j) - 1/L\bigr)\mathds{1}\Big\{\bigl|L(\overline{V}^2(j) - 1/L)\bigr| \leq \eta\Big\}\bigg| >\eta \bigg) \nonumber \\ 
        &\leq \frac{1}{\eta^2} \Bigg\{ p\mathbb{E}\bigg[\bigl(L\overline{V}^2(1) - 1\bigr)^2\mathds{1}\Big\{\bigl|L(\overline{V}^2(1) - 1/L)\bigr| \leq \eta \Big\}\bigg] \nonumber\\
        &\qquad \quad + p^2 \bigg(\mathbb{E}\bigg[L\bigl(\overline{V}^2(1) - 1/L\bigr)\mathds{1}\Big\{\bigl|L(\overline{V}^2(1) - 1/L)\bigr| > \eta\Big\}\bigg]\bigg)^2\Bigg\} \nonumber \\
        &\leq \frac{p}{\eta^2}\mathbb{E}\Big[\bigr|L\overline{V}^2(1) - 1\bigr|^{k}\eta^{2-k}\Big] + \frac{p^2}{\eta^2}\bigg\{\mathbb{E}\Big[\bigl|L\overline{V}^2(1) - 1\bigr|^{k}\Big]\bigg\}^{2/k} \bigg\{\mathbb{P}\Big(\bigl|L\bigl(\overline{V}^2(1) - 1/L\bigr)\bigr| > \eta \Big)\bigg\}^{2(k-1)/k} \nonumber \\
       & \leq \frac{C_{1,k} p}{\eta^k} + \frac{C_{1,k}^{2/k}p^2}{\eta^2} \bigg(\frac{C_{1,k}}{\eta^k}\bigg)^{2(k-1)/k},
    \end{align}
    where we have used Markov's inequality for the first inequality, H\"older's inequality for the second one and \eqref{eq:bound1} for the last one. Combining \eqref{eq:mom-bound-1+epsilon}, \eqref{eq:bound1} and \eqref{eq:bound2}, we have 
    \[
    \mathbb{P}\bigg(\bigg|\sum_{j=1}^p L\bigl(\overline{V}^2(j) - 1/L\bigr)\bigg| >\eta \bigg) \leq \frac{2C_{1,k}p}{\eta^k} + \frac{C_{1,k}^2p^2}{\eta^{2k}}.
    \]
   Note that if $C_{1,k}p/\eta^k > 1$, the bound above holds trivially. Therefore we obtain
   \[
   \mathbb{P}\bigg(\sum_{j=1}^p L(\overline{V}^2(j) - 1/L) >\eta \bigg) \leq \frac{3C_{1,k}p}{\eta^k}, 
   \]
    for any $\eta>0$, which is equivalent to the claimed bound.

    For $k > 2$, by Markov's inequality,~\eqref{eq:bound1} and \Cref{Lemma:bounded2kmoments}(b), there exists a constant $C_{2,k} > 0$, depending only on $k$ and $C$ such that
    \[
    \mathbb{P}\Bigl(\bigl|L\bigl(\overline{V}^2(1) - 1/L\bigr)\bigr| > \eta\Bigr) \leq \frac{\mathbb{E}\bigg[\Big|\sum_{j=1}^p L\bigl(\overline{V}^2(j) - 1/L\bigr)\Big|^k \bigg]}{\eta^k} \leq \frac{C_{2,k} p^{k/2}}{\eta^k},
    \]
    which proves the desired result.
    
    \end{proof}

\begin{lemma} \label{Lemma:mediandifference}
Let $n \in \mathbb{Z}^+$ and $c \in \mathbb{R}$. Let $a_1, \dotsc, a_n, b_1, \dotsc, b_n$ be $2n$ real numbers. Suppose that $a_i-b_i \leq c$ for all $i \in [n]$. Then 
\[
\mathrm{median}(a_1, \dotsc, a_n) - \mathrm{median}(b_1,\dotsc, b_n) \leq c.
\]
\end{lemma}
\begin{proof}
We sort the two arrays respectively and obtain $a_{(1)} \leq \dotsc \leq a_{(n)}$ and $b_{(1)} \leq \dotsc \leq b_{(n)}$. We show that $a_{(i)}-b_{(i)} \leq c$ for all $i \in [n]$. Indeed, there exists a set $\mathcal{I}_{i} \subseteq [n]$ with $|\mathcal{I}_i| \geq i$ such that
\[
b_{(i)} = \max\{b_{j}: j \in \mathcal{I}_i\} \geq \max\{a_{j}-c: j \in \mathcal{I}_i\} \geq a_{(i)} -c.
\]
The desired results follows by observing that the median is a convex combination of the order statistics.
\end{proof}

\begin{lemma}\label{lemma:weak_mean}
Let $1 \leq \alpha \leq 2$ and $X_1, \dotsc, X_n$ be independent random vectors in $\mathbb{R}^p$ with mean $\mu$. Assume that the distribution of $X_i - \mu$ belongs to $\mathcal{W}_\alpha$ for each $i \in [n]$. Then 
    \[
    \mathbb{P}\left(\Big|\|\overline{X}\|_2 - \|\mu\|_2\Big| \geq u\right) \leq \frac{\pi p^{\frac{\alpha}{2}}}{n^{\alpha-1} u^{\alpha}},
    \]
for any $u >0$.
\end{lemma}
\begin{proof}[Proof of \Cref{lemma:weak_mean}]
    By \citet[Lemma~4.2]{cherapanamjeri2022optimal}, we have
    \[
    \mathbb{E}|\langle \overline{X}-\mu,v \rangle|^{\alpha} \leq \frac{2}{n^{\alpha-1}},
    \]
    for any unit vector $v \in \mathbb{R}^p$. This means that $\overline{X}-\mu$ also satisfy \eqref{eq:weakmoment} but with a different constant instead of $1$. Then, we use \citet[Lemma~4.1]{cherapanamjeri2022optimal} to deduce
    \[
    \mathbb{E}\|\overline{X} - \mu\|_2^{\alpha} \leq \frac{\pi}{n^{\alpha-1}}p^{\frac{\alpha}{2}}. 
    \]
    By triangle inequality and Markov's inequality, we have 
    \[
    \mathbb{P}\left(\Big|\|\overline{X}\|_2 - \|\mu\|_2\Big| \geq u\right) \leq \mathbb{P}(\|\overline{X} - \mu\|_2 \geq u) \leq  \frac{\pi p^{\frac{\alpha}{2}}}{n^{\alpha-1} u^{\alpha}},
    \]
    as desired.
\end{proof}

\end{document}